\def\draft{n}
\documentclass[12pt]{amsart}
\usepackage[headings]{fullpage}
\usepackage{amssymb,epic,eepic,epsfig,amsbsy,amsmath,amscd,graphicx}
\usepackage[all]{xy}
\usepackage{url}

\usepackage{mathrsfs}
\usepackage[bookmarks=true,%
    colorlinks=true,%
    linkcolor=blue,%
    citecolor=blue,%
    filecolor=blue,%
    menucolor=blue,%
    urlcolor=blue,%
    breaklinks=true]{hyperref}

\newtheorem{theorem}{Theorem}[section]
\theoremstyle{definition}
\newtheorem{proposition}[theorem]{Proposition}
\newtheorem{lemma}[theorem]{Lemma}
\newtheorem{definition}[theorem]{Definition}
\newtheorem{remark}[theorem]{Remark}
\newtheorem{corollary}[theorem]{Corollary}

\newtheorem{question}[theorem]{Question}

\newtheorem{exercise}[theorem]{Exercise}

\def\printname#1{
        \if\draft y
                \smash{\makebox[0pt]{\hspace{-0.5in}
                        \raisebox{8pt}{\tt\tiny #1}}}
        \fi
}

\newcommand{\psdraw}[2]
         {\begin{array}{c} \hspace{-1.3mm}
        \raisebox{-4pt}{\epsfig{figure=draws/#1.eps,width=#2}}
        \hspace{-1.9mm}\end{array}}

\newlength{\standardunitlength}
\catcode`\@=11
\long\def\@makecaption#1#2{%
     \vskip 10pt

\setbox\@tempboxa\hbox{
       \small\sf{\bfcaptionfont #1. }\ignorespaces #2}%
     \ifdim \wd\@tempboxa >\captionwidth {%
         \rightskip=\@captionmargin\leftskip=\@captionmargin
         \unhbox\@tempboxa\par}%
       \else
         \hbox to\hsize{\hfil\box\@tempboxa\hfil}%
     \fi}
\font\bfcaptionfont=cmssbx10 scaled \magstephalf
\newdimen\@captionmargin\@captionmargin=2\parindent
\newdimen\captionwidth\captionwidth=\hsize
\catcode`\@=12

\def\lbl#1{\label{#1}\printname{#1}}

\def\BN{\mathbb N}
\def\BZ{\mathbb Z}

\def\BQ{\mathbb Q}
\def\BR{\mathbb R}

\def\calB{\mathcal B}

\def\D{\Delta}

\def\calT{\mathcal T}

\def\bb{\beta}

\def\l{\lambda}

\def\S{\Sigma}
\def\s{\sigma}

\def\ga{\gamma}

\def\w{\omega}

\def\d{\delta}

\def\th{\theta}

\def\s{\sigma}

\def\ti{\widetilde}

\def\longto{\longrightarrow}
\def\w{\omega}

\def\calB{\mathcal{B}}

\def\sgn{\mathrm{sgn}}

\def\sign{\operatorname{sign}}
\def\Zp{\BZ_{\succ}[q] }
\def\fsl{\mathfrak{sl}}
\def\wp{w_{\succ}}
\def\wlt{w_{\mathrm{lt}}}
\def\mindeg{\mathrm{mindeg}}
\def\coeff{\mathrm{coeff}}
\def\calE{\mathcal E}
\def\nbd{\mathrm{nbd}}
\def\calP{\mathcal P}

\def\be{  \begin{equation} }
\def\ee{  \end{equation} }
\def\cP{\mathcal P}

\def\ve{\varepsilon}

\def\bb{\bar b }

\def\bD{\hat D}
\def\U{\mathcal U}
\newcommand{\qbinom}[2]{\begin{bmatrix}#1\\ #2\end{bmatrix}}
\def\ho{\hat \otimes}

\def\cH{\mathcal U}
\def\cF{\mathcal F}
\def\id{\mathrm{id}}
\def\bb{\mathfrak b}
\def\Hom{\mathrm{Hom}}

\def\ev{\mathrm{ev}}
\def\cE{\mathcal E}
\def\cV{\mathcal V}

\def\cV{\mathcal V}

\def\cT{\mathcal T}

\def\nbd{\mathrm{nbd}}
\def\calP{\mathcal P}

\def\vv{|\!|}

\def\cF{\mathcal F}
\def\calS{\mathcal S}

\def\oor{\mathrm{or}}
\def\cS{\calS}
\def\cQ{\mathcal Q}

\def\cD{D^*}
\def\Adj{\operatorname{Adj}}
\def\Deg{\operatorname{Deg}}
\def\cL{\mathcal L}
\def\Adm{\operatorname{Adm}}
\def\tF{\tilde F}
\def\bb{\mathbf{b}}
\def\te{\tilde e}

\def\ce{\breve{e}}

\begin{document}


\title[Nahm sums, stability and the colored Jones polynomial]{
Nahm sums, stability and the colored Jones polynomial}
\author{Stavros Garoufalidis}
\address{School of Mathematics \\
         Georgia Institute of Technology \\
         Atlanta, GA 30332-0160, USA \newline
         {\tt \url{http://www.math.gatech.edu/~stavros }}}
\email{stavros@math.gatech.edu}
\author{Thang T.Q. L\^e}
\address{School of Mathematics \\
         Georgia Institute of Technology \\
         Atlanta, GA 30332-0160, USA \newline
         {\tt \url{http://www.math.gatech.edu/~letu }}}
\email{letu@math.gatech.edu}
\thanks{The authors were supported in part by NSF. \\
\newline
1991 {\em Mathematics Classification.} Primary 57N10. Secondary 57M25.
}

\date{May 14, 2012}


\begin{abstract}
Nahm sums are $q$-series of a special hypergeometric type that appear in
character formulas in Conformal Field Theory, and give rise to elements
of the Bloch group, and have interesting modularity properties.
In our paper, we show how Nahm sums arise naturally in Quantum Knot Theory,
namely we prove the stability of the coefficients of the colored Jones
polynomial of an alternating link and present a Nahm sum formula for the
resulting power series, defined in terms of a reduced diagram of the
alternating link. The Nahm sum formula comes with a computer implementation,
illustrated in numerous examples of proven or conjectural identities
among $q$-series.
\end{abstract}

\maketitle

\tableofcontents


\section{Introduction}
\lbl{sec.intro}

The colored Jones polynomial of a link is a sequence of Laurent polynomials
in one variable with integer coefficients. We prove in full a conjecture
concerning the stability of the colored Jones polynomial for all alternating
links.

A weaker form of stability ($0$-stability, defined below) for the colored
Jones polynomial of an alternating knot was conjectured by Dasbach and
Lin.  The $0$-stability is also proven independently by Armond for all
adequate links \cite{Ar2}, which include alternating links and closures
of positive braids, see also \cite{Ar}. The advantage of our approach is it proves
stability up to all order, and gives explicit formulas (in the form of
generalized Nahm sums) for the limiting series,  which in
particular implies convergence in the open unit disk in the $q$-plane.

Stability was observed in some examples by Zagier, and conjectured by the
first author to hold for all knots, assuming that we restrict the sequence
of colored Jones polynomials to suitable arithmetic progressions, dictated
by the quasi-polynomial nature of its $q$-degree \cite{Ga1,Ga3}. Zagier asked
about modular and asymptotic properties of the limiting $q$-series.
In a similar direction, Habiro asked about $0$-stability of the cyclotomic
function of alternating links in \cite{Ha}.

 Our
generalized Nahm sum
formula comes with a computer implementation (using as input a planar
diagram of a link), and allows the study of its asymptotics when $q$
approaches radially a root of unity. Our Nahm sum formula is reminiscent
to the cohomological Hall algebra of motivic Donaldson-Thomas invariants
of Kontsevich-Soibelman \cite{KS}, and complement recent work
of Witten \cite{Wi} and Dimofte-Gaiotto-Gukov \cite{DGG2}.

\subsection{Nahm sums}
\lbl{sub.nahmsum}

Recall the {\em quantum factorial} and {\em quantum Pochhammer symbol}
defined by \cite{An}:
$$
(x;q)_n=\prod_{k=0}^{n-1} (1-xq^k), \qquad (x;q)_\infty =\prod_{k=0}^\infty (1-xq^k)
$$
We will abbreviate $(x;q)_n$ by $(x)_n$.
\lbl{sub.nahm}

In \cite{Nahm0} Nahm studied $q$-hypergeometric series $f(q) \in \BZ[[q]]$
of the form
\begin{equation*}
f(q) =\sum_{n_1,\dots,n_r \geq 0} \frac{q^{\frac{1}{2} n^t \cdot A \cdot n + b \cdot n}}{
(q)_{n_1} \dots (q)_{n_r}}
\end{equation*}
where $A$ is a positive definite even integral symmetric matrix and
$b \in \BZ^r$.

Nahm sums appear in character formulas in Conformal Field Theory, and
define analytic functions in the complex unit disk $|q|<1$ with
interesting asymptotics at complex roots of unity, and with sometimes
modular behavior.  Examples of Nahm sums are the  seven famous,
mysterious $q$-series of Ramanujan that are nearly modular (in modern terms,
mock modular). For a detailed discussion, see \cite{Za.mock}. Nahm sums give
rise to elements of the Bloch group, which governs the leading radial
asymptotics of $f(q)$ as $q$ approaches a complex root of unity.
Nahm's Conjecture concerns the modularity of a Nahm sum $f(q)$, and was
studied extensively by Zagier, Vlasenko-Zwegers and others
\cite{VZ,Za.dilog}.

The limit of the colored Jones function of an alternating link
leads us to consider generalized Nahm sums of the form
\be
\lbl{eq.nahmgen}
\Phi(q)=\sum_{ n \in C \cap \BN^r} (-1)^{a \cdot n}
\frac{q^{\frac{1}{2} n^t \cdot A \cdot n + b \cdot n}}{(q)_{n_1}
\dots (q)_{n_r}}
\ee
where $C$ is a rational polyhedral cone in $\BR^r$, $b, a \in \BZ^r$
and $A$ is a symmetric (possibly indefinite) symmetric matrix.
We will say that the generalized Nahm sum \eqref{eq.nahmgen} is
{\em regular} if the function
$$
n \in C \cap \BN^r \mapsto \frac{1}{2} n^t \cdot A \cdot n + b \cdot n
$$
is proper and bounded below. Regularity ensures that the series
\eqref{eq.nahmgen}
is a well-defined element of the Novikov  ring
$$
\BZ((q ))=\{\sum_{n \in \BZ} a_n q^n \, | \, a_n=0, \, n \ll 0 \}
$$
of power series in $q$ with integer coefficients and
bounded below minimum degree.
In the remaining of the paper, the term Nahm sum will refer to a generalized
Nahm sum. The paper is concerned with a new source of Nahm sums, namely
Quantum Knot Theory.

\subsection{Stability of a sequence of polynomials}
\lbl{sec.stable}

For $f(q)= \sum a_j q^j\in \BZ((q))$ let $\mindeg_q f(q)$ denote the smallest
$j$ such that $a_j \neq 0$ and let $\coeff(f(q),q^j)= a_j$ denote
the coefficient of $q^j$ in $f(q)$.

\begin{definition}
\lbl{def.lim}
Suppose $f_n(q),f(q) \in \BZ((q))$. We write that
$$
\lim_{n\to \infty} f_n(q) = f(q)
$$
if
\begin{itemize}
\item
there exists $C$ such that $\mindeg_q(f_n(q)) \geq C$ for all $n$, and
\item
for every $j\in \BZ$,
\be
\lbl{e002}
\lim_{n\to \infty} \coeff(f_n(q), q^j) = \coeff(f(q), q^j).
\ee
\end{itemize}
\end{definition}
Since Equation \eqref{e002} involves a limit of integers, the above
definition implies that for each $j$, there exists $N_j$ such that
$$
f_n(q)-f(q) \in q^j \BZ[[q]]
$$
(and in particular, $\coeff(f_n(q), q^j)=\coeff(f(q), q^j)$)
for all $n > N_j$.

\begin{remark}
Although for every integer $j$ we have $\lim_{n\to\infty}\coeff(q^{-n^2},q^j)=0$,
it is not true that $\lim_{n\to\infty} q^{-n^2}=0$.
\end{remark}

\begin{definition}
\lbl{def.stable}
A sequence $f_n(q) \in \BZ[[q]]$ is {\em $k$-stable} if there
exist $\Phi_j(q) \in \BZ((q))$ for $j=0,\dots,k$ such that
\be
\lbl{eq.kstable}
\lim_{n \to \infty} \,\, q^{-k(n+1)}
\left(f_n(q) -\sum_{j=0}^k \Phi_j(q) q^{j(n+1)}\right)=0
\ee
We say that $(f_n(q))$ is {\em stable} if it is $k$-stable for all $k$.
Notice that if $f_n(q)$ is $k$-stable, then it is $k'$-stable for all $k' <k$
and moreover $\Phi_j(q)$ for $j=0,\dots,k$ is uniquely determined by $f_n(q)$.
We call $\Phi_k(q)$ the $k$-limit of $(f_n(q))$.
For a  stable sequence $(f_n(q))$, its associated series is given by
\begin{equation*}
F_f(x,q)=
\sum_{k=0}^\infty \Phi_k(q) x^k \in \BZ((q ))[[x]]\,.
\end{equation*}
\end{definition}

It is easy to see that the pointwise sum and product of $k$-stable
sequences are $k$-stable.

\subsection{Stability of the colored Jones function for alternating links}
\lbl{sub.results}

 Given a link $K$, let $J_{K,n}(q) \in \BZ[q^{\pm 1/2}]$
denote its colored Jones polynomial (see e.g.
\cite{Ohtsuki,Tu1}) with each component colored by the
$(n+1)$-dimensional irreducible representation of $\fsl_2$ and
normalized by
$$
J_{\text{Unknot},n}(q)=(q^{(n+1)/2}-q^{-(n+1)/2})/(q^{1/2}-q^{-1/2})\,.
$$
When $K$ is an {\em alternating} link, the lowest degree of $J_{K,n}(q)$
is known and the lowest coefficient is $\pm 1$ (see \cite{Le,Lickorish} and
Section \ref{sec.0stability}).
We  divide $J_{K,n}(q)$ by its lowest monomial to obtain $\hat J_{K,n}(q)$.
Although $J_{K,n}(q) \in \BZ[q^{1/2}]$, we have
$\hat J_{K,n}(q) \in \BZ[q]$; see \cite{Le_Duke}.

Our main results  link the colored Jones polynomial and its stability with
Nahm sums. The first part of the result, with proof given in Section
\ref{sec.bounded}, is the following.

\begin{theorem}
\lbl{thm.2}
For every alternating link $K$, the sequence $(\hat J_{K,n}(q))$
is stable and its associated $k$-limit
$\Phi_{K,k}(q)$ and series $F_K(x,q)$ can be effectively computed from
any reduced, alternating diagram $D$ of $K$.
\end{theorem}

Let us give some remarks regarding Theorem \ref{thm.2}.

\begin{remark}
\lbl{rem.thm2a}
If one uses the new normalization where with $J_{\text{Unknot}, n}(q) =1$,
the above theorem still holds. The new $F_K(x,q)$ is equal to the old one times $(1-q)/(1-x)$.
\end{remark}

\begin{remark}
\lbl{rem.thm2b}
If $\bar K$ is the mirror image of $K$, then
$J_{\bar K,n}(q^{-1}) = J_{K,n}(q)$. If $K$ is alternating, then so is $\bar K$.
Hence, applying Theorem \ref{thm.2}   to $\bar K$, we see that
similar stability result holds for the head of the colored Jones polynomial
of alternating link.
\end{remark}

\begin{remark}
\lbl{rem.thm2c}
The weaker $0$-stability (conjectured by Dasbach and Lin) is proven
independently by Armond \cite{Ar}. In \cite{Ar}, $0$-stability is
proved for all $A$-adequate links, which include all alternating links, but
no stability in full is proven there, nor any formula for the $0$-limit
is given.
As we will see, the proof of stability in full is more complicated than
that of $0$-stability and occupies the more difficult part of our paper,
given in Sections \ref{sec.kbounded}-\ref{sec.sk}.
\end{remark}

\begin{remark}
\lbl{rem.thm2d}
A sharp estimate regarding the rate of convergence of the stable sequence
$\hat J_{K,n}(q)$ is given in Theorem \ref{thm.qholo}.
\end{remark}

\subsection{Explicit Nahm sums for the 0-limit and 1-limit}
\lbl{sub.nahmalt}

Throughout this subsection  $D$ is a reduced diagram of a {\em non-split}
alternating link $K$ with $c$ crossings.

\subsubsection{Laplacian of a graph}

In this paper a graph is a finite one-dimensional CW-complex.
{\em A plane graph} is a graph $\Gamma$ (with loops and multiple edges
allowed) together with an embedding of $\Gamma$ into $\BR^2 \subset S^2$.
A plane graph $\Gamma$ gives rise to
a polygonal complex structure of $S^2$, and
its set of {\em vertices}, set of {\em edges}, and set of {\em polygons}
are denoted respectively by $\cV(\Gamma)$, $\cE(\Gamma)$ and $\cP(\Gamma)$.

The {\em adjacency matrix}
$\Adj(\Gamma)$ is the $\cV(\Gamma) \times \cV(\Gamma)$ matrix defined
such that $\Adj(\Gamma)(v,v')$ is the number of edges connecting $v$ and
$v'$. Let $\Deg(G)$ be the diagonal $\cV(\Gamma) \times \cV(\Gamma)$ matrix
such that $\Deg(\Gamma)(v,v)$ is the degree of the vertex $v$, i.e. the
number of edges incident to $v$, with the convention that each loop edge
at $v$ is counted twice.

The {\em Laplacian} $\cL(\Gamma): = - \Deg(\Gamma) + \Adj(\Gamma)$ plays an
important role in graph theory.

\subsubsection{Graphs associated to a reduced alternating non-split link diagram $D$}

The diagram $D$ gives rise to a polygonal complex of  $S^2= \BR^2 \cup \infty$
with $c$ vertices,
$2c$ edges, and  $c+2$ polygons. Since $D$ is alternating, there is a way to
assign a color $A$ or $B$ to each polygon such that
in a neighborhood of each crossing the colors are as in the following
figure, see e.g. \cite[p.217]{Tu3}.

\begin{figure}[htpb]
$$
\psdraw{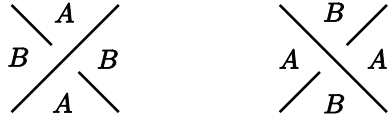}{2in}
$$
\caption{A checkerboard coloring of alternating planar projections}
\lbl{f.pntype}
\end{figure}

This is the usual checkerboard coloring of the regions of an alternating
link diagram, used already by Tait.
When we rotate the overcrossing arc at a crossing
counterclockwise (resp. clockwise), we swap a $A$-type (resp. $B$-type)
angle. Note that orientation dose not take part in the definition of
$A$-angles and $B$-angles.

Let $\cD$ be the dual of the plane graph $D$. By definition, each polygon
$p$ of $D$ contains in its interior the dual vertex $p^*\in \cV(\cD)$, and
$p^*$ is called an
$A$-vertex (respectively, $B$-vertex) of $\cD$ if $p$ is an $A$-polygon
(resp. $B$-polygon). Then $\cV(\cD)$  has a partition into $A$-vertices and
$B$-vertices
$\cV(\cD) = \cV_A \sqcup \cV_B$. The colors $A$ and $B$ give $\cD$ a
bipartite structure. Since the degree of each vertex of $D$ is 4, the
each polygon of $\cD$ is a quadrilateral, having 4 vertices, two of which
are $A$-vertices and two are $B$-vertices. Moreover, the two $B$-vertices
of each polygon of $\cD$ are opposite. Connect the two $B$-vertices of each
quadrilateral of $\cD$ by a diagonal inside that quadrilateral, and call it
a  $\cT$-edge.

{\em The Tait graph of $D$} is defined to be the plane graph $\cT$ whose set of
vertices is $\cV_B$, and whose set of edges is the set of $\cT$-egdes.  The plane
graph Tait graph totally determines the alternating link $K$ up to
orientation.
The graph $\cT$ can be defined for any link diagram, and is studied
extensively, see e.g. \cite{FKP,Ozawa,Thistlethwaite}.

Note that for a vertex $v \in \cV_B$, its degrees in $\cD$ and in $\cT$
are the same.

\subsubsection{The lattice and the cone}

Fix an $A$-vertex of $\cD$ and call it $v_\infty$.
We will focus on $\Lambda:=\BZ[\cV(\cD)]$, the $\BZ$-lattice of rank $c+2$
freely spanned by the vertices of $\cD$. Let $\Lambda_0 = \BZ[\cV(\cD)
\setminus \{ v_\infty\}]$, a sublattice of $\Lambda$ of rank $c+1$.

For an edge  $e \in \cE(\cD)$, define the $\BZ$-linear map
$e: \Lambda\to \BZ$ by
$$
e(v)= \begin{cases} 1 \quad &\text{if $v$ is a vertex of $e$},
\\ 0 & \text{otherwise}. \end{cases}
$$
An element $x \in \Lambda_0$ is {\em admissible} if $e(x) \ge 0$ for
every edge $e \in \cE(\cD)$. The set $\Adm \subset \Lambda_0$ of all
admissible elements is the intersection of $\Lambda_0$ with a rational convex
cone in $\Lambda \otimes \BR$.

Define the $\BZ$-linear map $L:\Lambda \to \frac 12 \BZ$ by
$$
L(v) = \begin{cases} 1 \quad  & \text{if } v \in \cV_B
\\ \frac{\deg(v)}2 -1  & \text{if } v \in \cV_A.
\end{cases}
$$

Let $\cQ$ be the symmetric $\cV(\cD) \times \cV(\cD)$ matrix defined by

\be
\label{e0001}
\cQ := \Deg(\cD) + \Adj(\cD) + \cL(\cT ).
\ee

Note that a priori $\cL(\cT )$ is a $\cV_B \times \cV_B$ matrix, and is
considered as a $\cV(\cD) \times \cV(\cD)$ matrix in the right hand side
of \eqref{e0001} by the trivial extension, i.e. in the extension, any entry
outside the block
$\cV_B \times \cV_B$ is 0.

The symmetric matrix $\cQ$ defines a symmetric bilinear form
$\cQ(x,y): \Lambda \otimes_\BZ \Lambda \to \BZ$. Let $Q: \Lambda \to
\frac 12\BZ$ be the corresponding
quadratic form, i.e.

$$
Q(\lambda) := \frac 12 \cQ(\lambda,\lambda).
$$

\begin{remark}
Although $Q(\lambda)$ and $L(\lambda)$ take value in
$\frac12\BZ$, we later show that $Q(\lambda) + L(\lambda) \in \BZ$.
While $Q,L$ depend only on $D$, the set $\Adm$ depends on the choice of
an $A$-vertex $v_\infty$.
\end{remark}

Examples that illustrate the above definitions are given in Section
\ref{sub.applications}.

\subsubsection{Nahm sum for the $0$-limit}

The next theorem is proven in Section \ref{sec.0stability}.

\begin{theorem}
\lbl{thm.1}
Suppose $D$ is a reduced alternating diagram of a non-split link $K$. Fix
any choice of $v_\infty$. Then the $0$-limit of $\hat J_{K,n}(q)$ is equal to
\be
\lbl{eq.nahm0}
\Phi_{K,0}(q) = (q)_\infty^{c}
\sum_{\lambda \in \Adm} (-1)^{2 L(\lambda)} \frac{q^{Q(\lambda)
+ L(\lambda)}}{\displaystyle{\prod_{e\in \cE(\cD)}\,  (q)_{e(\lambda)}}}.
\ee
The generalized Nahm sum on the right hand side is regular and belongs
to $\BZ[[q]]$.
\end{theorem}
A categorification of the above theorem was given recently by Rozansky
\cite{Rozansky}. Here are two consequences of this explicit formula.
The next corollary is proven in Section \ref{sub.cor.thm2.1}.

\begin{corollary}
\lbl{cor.thm2.1}
For every alternating link $K$, 
$\Phi_{K,0}(q) \in \BZ[[q]]$ is analytic in the unit disk
$|q|<1$.
\end{corollary}

The next corollary is shown in Section \ref{sec.tait}.

\begin{corollary}
\lbl{c001} If the reduced Tait graphs of two alternating links $K_1,K_2$ are isomorphic as abstract graphs, then
they have the same $0$-limit, 
$\Phi_{K_1,0}(q)=\Phi_{K_2,0}(q)  $.
\end{corollary}
Here the reduced Tait graph $\cT'$ is obtained from
$\cT $ by replacing every set of parallel edges by an edge;
and two edges are parallel if they connect the same two vertices.
This corollary had been proven by Armond and Dasbach: in \cite{ArDas}, it is
proved that if two alternating links have the same reduced Tait graph, and the
0-limit of the first link exists, then
the 0-limit of the second one exists and is equal to that of the first one.
In section \ref{sec.tait} we will derive  Corollary \ref{c001} from the
explicit formula of Theorem \ref{thm.1}.

We end this section with a remark on normalizations.

\begin{remark}
\lbl{rem.normalizations1}
The colored Jones polynomial $J_{K,n}(q)$ (and consequentrly, its shifted
version $\hat J_{K,n}(q) \in 1 + q \BZ[q]$) is independent of the orientation
of the components of a link $K$ \cite{Tu2}. With our normalization we have
\begin{eqnarray*}
\Phi_{\text{Unknot},0}(q) &=& \frac{1}{1-q}, \qquad F_{\text{Unknot}}(x,q)=\frac{1-x}{1-q}\\
\Phi_{K_1 \sqcup K_2,0}(q) &=& \Phi_{K_1,0}(q)  \Phi_{K_2,0}(q) \\
\Phi_{K_1 \sharp K_2,0}(q) &=& (1-q) \Phi_{K_1,0}(q) \Phi_{K_2,0}(q)
\end{eqnarray*}
where $\sqcup$ and $\sharp$ denotes the disjoint union and the connected
sum respectively.
\end{remark}

\subsubsection{The $1$-limit}

For a quadrilateral $p$ of $\cD$, define a $\BZ$-linear map
$p: \Lambda \to \BZ$ by
\begin{align*}
p(v)= \begin{cases}
1 \quad &\text{if $v$ is one of the four vertices of $p$}  \\
 0  &\text{otherwise.}
\end{cases}
\end{align*}
The next theorem is proven in Section \ref{sub.Phi1}.

\begin{theorem}
\lbl{thm.1a}
Suppose $D$ is a reduced alternating diagram of a non-split link $L$.
Fix any choice of $v_\infty$. The 1-limit of $\hat J_{K,n}(q)$ is
\begin{eqnarray}
\lbl{eq.Phi1}
\Phi_{K,1}(q) &=& \frac{(q)_\infty^{c}}{1-q}
\left(
\sum_{\lambda \in \Adm} (-1)^{2 L(\lambda)} \frac{q^{Q(\lambda)+L(\lambda)}}{\prod_{e \in \calE(\cD)}\,
(q)_{e(\lambda)}}
\left(\sum_{e \in \calE(\cD)} q^{-e(\lambda)} - \sum_{p \in \cP(\cD )}  q^{-p(\lambda)} \right)
\right.
\\
& &
\left. - \sum_{v \in \cV_B}  \frac{1}{(q)_\infty^{\deg(v)}} \,\sum_{\lambda \in \Adm_v}
(-1)^{2L(\lambda)} \, \,
\,
\frac{q^{Q(\lambda)+L(\lambda)}}{\prod_{e \in \calE(\cD)}\, (q)_{e(\lambda)}}
\right)\,. \notag
\end{eqnarray}
where $\Adm_v$ is the set of all admissible $x$ such that $p(x)=0$ for
every $p \in \cP(\cD )$ incident to $v$.
\end{theorem}

For an example illustrating Theorems \ref{thm.1} and \ref{thm.1a}, see
Section \ref{sub.applications}.

\subsection{$q$-holonomicity}
\lbl{sub.qholo}
Recall the notion of a $q$-holonomic sequence
sequence and series from \cite{Z,PWZ}.
 We say that $f_n(q)$, belonging to a
$\BZ[q^{\pm 1}]$-module for $n=1,2,\dots$, is $q$-{\em holonomic} if it
satisfies a linear recursion of the form
\be
\lbl{eq.def.qholo}
\sum_{j=0}^d c_j(q^n,q) f_{n+j}(q) =0
\,,
\ee
for all $n \in \BN$ where $c_j(u,v) \in \BZ[u,v]$ for all $j$ and
$c_d \neq 0$.

The next theorem (proven in Section \ref{sec.thm.qholo}) shows the
$q$-holonomicity of $\Phi_{K,n}(q)$ for an alternating link, and gives
a sharp improvement of the rate of convergence in the definition of stability.

\begin{theorem}
\lbl{thm.qholo}
\rm{(a)} For every alternating link $K$, $\Phi_{K,n}(q)$ is $q$-holonomic.
\newline
\rm{(b)} Moreover, there exist constants $C$ and $C'$ such that
\be
\lbl{eq.degPhik}
\mindeg_q(\Phi_{K,k}(q)) \geq -C k^2-C'
\ee
for all $k$ and
\be
\lbl{eq.quadstable}
\left(f_n(q)-\sum_{j=0}^k \Phi_k(q) q^{j(n+1)}\right)q^{-k(n+1)} \in
q^{n+1-C(k+1)^2-C'}\BZ[[q]]
\ee
for all $k$ when $n$ is sufficiently large (depending on $k$).
\end{theorem}

Equation \eqref{eq.quadstable} is sharp when $K=4_1$ knot \cite{GZ}.

\begin{question}
\lbl{que.11}
Does $F_{K,n}(x,q)$ uniquely determine the sequence
$(\hat J_{K,n}(q))$ for the case of knots?
\end{question}

\subsection{Applications:  $q$-series identities}
\lbl{sub.applications}

In this section we illustrate Theorem \ref{thm.1} 
explicitly for the $4_1$ knot. Consider the planar projection $D$ of $4_1$
given in Figure \ref{f.41draw}. This planar projection is $A$-infinite.

\begin{figure}[htpb]
$$
\psdraw{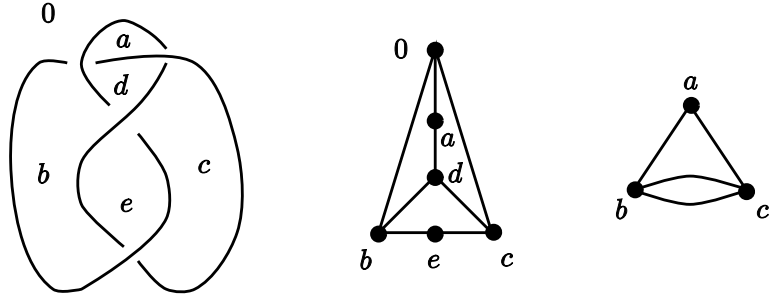}{4in}
$$
\caption{A planar projection $D$ of the $4_1$ knot on the left,
the dual graph $D^*$ in the middle and the Tait graph $\calT$ on the right.}
\lbl{f.41draw}
\end{figure}

To compute $\Phi_{4_1,0}(q)$, proceed as follows:
\begin{itemize}
\item
Checkerboard color the regions of $D$ with $A$ or $B$ with the unbounded
region colored by $A$.
\item
Assign variables $a,b,c$ to the three $B$-regions and $e,f$ to the two
bounded $A$-regions, and assign $0$ to the unbounded $A$-region. Let
$\l=(a,b,c,d,e)^T$.
\item
Color each arc  of the diagram $D$ with the sum of the colors of
its two neighboring regions. $\l$ is admissible if the color of each
arc is a nonnegative integer number, i.e., $\l \in \BZ^5$ satisfies
$$
a,\,b,\,c,\,a+d,\,b+d,\,c+d,\,b+e,\,c+e \geq 0 \,.
$$
\item
Construct a square matrix (and a corresponding quadratic form $Q(\l)$)
which consists of four blocks: $BB$-block,
$AB$-block, $BA$-block and $AA$ block. On the $BB$, $AB$ and $BA$ blocks we
place the adjacency matrix of the corresponding regions: the adjacency number between two distinct $B$-regions is the number of common vertices, whereas the
adjacency number between an $A$-region and a $B$-region is the number of common edges. In the case when two regions share common vertices, the adjacency number is the number of common vertices.
On the $AA$-block
we place the diagonal matrix whose diagonal entries are  the number of sides of each $A$-region.
\item
We construct a linear form $L(\l)$ in $\l$ where the
coefficient of each $B$-variable $a,b,c$ is one, and the coefficient of each
$A$-variable $d,e$ is half the number of the sides of the corresponding region
minus $1$.
\end{itemize}
Explicitly, with the conventions of Figure \ref{f.41draw} we have
$$
Q(\l)=\frac{1}{2}\l^T
\left(\begin{array}{c|c}
\begin{array}{ccc}
0 & 1 & 1 \\ 
1 & 0 & 2 \\ 
1 & 2 & 0  
\end{array}
&
\begin{array}{cc}
 1 & 0 \\
 1 & 1 \\
1 & 1 \\
\end{array} \\ \hline
\begin{array}{ccc}
1 & 1 & 1 \\ 
0 & 1 & 1 
\end{array}
&
\begin{array}{cc}
3 & 0 \\
 0 & 2
\end{array}
\end{array}\right)
\l,
\qquad
L(\l)=(1,1,1,\frac{1}{2},0) \l\,\,.
$$
Then,
$$
\Phi_{4_1,0}(q)=(q)_\infty^4 \sum_{\l \in \Adm}
(-1)^d \frac{q^{Q(\l)+L(\l)}}{(q)_{a} (q)_{b}(q)_{c}(q)_{a+d}
(q)_{b+d} (q)_{c+d}(q)_{b+e}(q)_{c+e}} \,.
$$
Alternative formulas for the colored Jones polynomial of $4_1$ lead to
identities among $q$-series. For instance, the Habiro formula for $4_1$ 
\cite{Habiro5} combined with the above formula for $\Phi_{4_1,0}(q)$
leads to the following identity:
\be
\lbl{eq.phi410}
\frac{1}{(1-q)(q)_\infty^{3}} = \sum_{\l \in \Adm}
(-1)^d \frac{q^{Q(\l)+L(\l)}}{(q)_{a} (q)_{b}(q)_{c}(q)_{a+d}
(q)_{b+d} (q)_{c+d}(q)_{b+e}(q)_{c+e}} \,.
\ee
The above identity has been proven by Armond-Dasbach. A detailed list
of identities for knots with knots with at most 8 crossings is given
in Appendix \ref{sec.experiment}.

\subsection{Extensions of stability}
\lbl{sub.extensions}

The methods that prove Theorem \ref{thm.2} are general and apply to several
other circumstances of $q$-holonomic sequences that appear in Quantum
Topology. We will list two results here, whose proofs will be discussed
in detail in a later publication \cite{GLfuture}.

\begin{theorem}
\lbl{thm.positive}
If $K$ is a positive link, then $\hat J_{K,n}(q)$ is stable and the
corresponding limit $F_{K}(x,q)$ is obtained by a Nahm sum associated to a
positive downwards diagram of $K$. Moreover, for every $k \in \BN$ we have
$\Phi_{K,k}(q) \in \BZ[q^{\pm 1}]$.
\end{theorem}

The proof of the above theorem is easier than that of Theorem \ref{thm.2}
since it does not require to center the states of the R-matrix state sum
of a positive link. An example that illustrates the above theorem is taken
from \cite[Sec.1.1.4]{HL}: for the right handed trefoil $3_1$, its
associated series is
$$
F_{3_1}(x,q)=\frac{1-x}{1-q} \sum_{k=0}^\infty x^k \left(1-\frac{x}{q}\right)\dots
\left(1-\frac{x}{q^k}\right) \in \BZ[q^{\pm 1}][[x]]\,.
$$
Some results related to the $0$-stability of a class of positive knots
are obtained in \cite{AK}.

Next we discuss an extension of Theorem \ref{thm.2} to evaluations
of quantum spin networks. For a detailed discussion of those, we refer the
reader to \cite{Co,KL,GV}. Using the notation of \cite{GV},
let $\ga=(a,b,c,d,e,f)$ be an admissible coloring
of the edges of the standard tetrahedron
$$
\psdraw{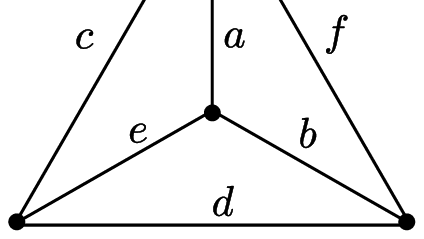}{1.3in}
$$
Consider the standard spin network evaluation
$J_{,n\ga}(q) \in \BZ[q^{\pm 1}]$ \cite{GV,Co}.

\begin{theorem}
\lbl{thm.q6j}
For every admissible $\ga$, the sequence $\hat J_{\psdraw{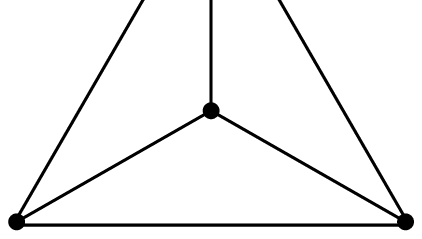}{0.18in},n\ga}(q)$
is stable, and its limit is given by a Nahm sum.
\end{theorem}
For example, if $\gamma=(2,2,2,2,2,2)$, then
$$
\hat J_{\psdraw{tetra}{0.18in},n\ga}(q) =
\frac{1}{1-q} \sum_{k=0}^{n} (-1)^k
\frac{ q^{\frac{3}{2}k^2+\frac{1}{2}k}}{(q)_k^3}
\frac{(q)_{4n+1-k}}{(q)_k^3 (q)_{n-k}^4}\,,
$$
and
$$
F_{\psdraw{tetra}{0.18in}}(x,q)=\frac{1}{(1-q)(q)_\infty^3}
\sum_{k=0}^\infty
(-1)^k
\frac{ q^{\frac{3}{2}k^2+\frac{1}{2}k}}{(q)_k^3}
\frac{(x q^{-k})_\infty^4}{(x^4 q^{-k+1})_\infty} \in \BZ((q))[[x]]\,,
$$
where $x=q^{n+1}$.
In particular,
$$
\Phi_{\psdraw{tetra}{0.18in},0}(q)=
\frac{1}{(1-q)(q)_\infty^3}
\sum_{k=0}^{\infty} (-1)^k \frac{ q^{\frac{3}{2}k^2+\frac{1}{2}k}}{(q)_k^3} \in
\BZ[[q]]\,.
$$
The proof of the above theorem follows easily from the fact that the quantum
6j-symbol is given by a 1-dimensional sum of a $q$-proper hypergeometric
summand, and the sum is already centered.
The analytic and arithmetic properties of the corresponding Nahm sum will
be discussed in forthcoming work \cite{GZ}.

\subsection{Plan of the proof}
\lbl{sub.plan}

The strategy to prove Theorems \ref{thm.1} and \ref{thm.2} is the following.

We begin with the $R$-matrix state sum for the colored Jones polynomial,
reviewed in Sections \ref{sub.morseknot}-\ref{sub.weights}.

We center the downward diagram, its corresponding states and their weights
in Section \ref{sub.local.antistate}.

We factorize the weights of the centered states as the product of a monomial
and an element of $\Zp$ in Section \ref{sub.local.antistate}. The advantage
of using centered states is that the lowest $q$-degree of their weights
is the sum of a quadratic function $Q(s)$ of $s$ with a quadratic function
of $n$.

Although $Q(s)$ is not a positive definite quadratic form,
in Section \ref{s101} we show that $Q(s)$ is copositive on the cone
of the centered states. The proof uses the combinatorics of alternating
downward diagrams, and their centered states, reminiscent to the Kauffman
bracket.

Section \ref{sec.0stability} we prove the $0$-stability Theorem
\ref{thm.1}.

If $Q(s)$ were positive definite, then it would be easy to deduce
Theorem \ref{thm.2}. Unfortunately, $Q(s)$ is never positive definite,
and it always has directions of linear growth in the cone of centered
states. In Sections \ref{sec.kbounded} we state a partition of the
set of $k$-bounded states, and prove stability away from the region of
linear growth. In Section \ref{sec.bounded} deal with stability in the
region of linear growth.

Section \ref{sec.sk} is rather technical, and gives a proof of the key
Proposition \ref{prop.sk1} that partitions the set of $k$-bounded states.

Section \ref{sec.thm.qholo} deduces the $q$-holonomicity of the sequence
$\Phi_{K,k}(q)$ of an alternating link from the $q$-holonomicity of the
corresponding colored Jones polynomial. As a result, we obtain sharp
quadratic lower bounds for the minimum degree of $\Phi_{K,k}(q)$ and
sharp bounds for the convergence of the colored Jones polynomial stated
in Theorem \ref{thm.qholo}.

In Section \ref{sec.Phi1} we give an algorithm for computing $\Phi_{K,k}(q)$
from a reduced alternating planar projection.

In Section \ref{sec.tait} we prove that $\Phi_{K,0}(q)$ is determined by
the reduced Tait graph of an alternating link $K$.

In Section \ref{sec.examples} we give some illustrations of Theorems
\ref{thm.1} and \ref{thm.2}.

\subsection{Acknowledgment}
An early version of the paper was presented in talks of the first author to
a Spring School in Geometry and Quantum Topology in the Diablerets 2011,
and in the Mathematische Arbeitstagung in Bonn, 2011. The authors wish to
thank the organizers of the above conferences for their hospitality,
C. Armond and O. Dasbach for explaining to us their beautiful work. The
first named author wishes to thank T. Dimofte and
D. Zagier for their interest, encouragement and for the generous sharing
of their ideas.


\section{The $R$-matrix state-sum of the colored Jones
polynomial}
\lbl{sub.Rmatrix}

In this section we review the $R$-matrix state sum of the colored Jones
function, discussed in detail in  \cite{Tu1,Tu2,Ohtsuki}. We will
use the following standard notation in $q$-calculus.

\begin{eqnarray*}
\binom{a}{b}_q &=& \frac{(q;q)_a}{(q;q)_b (q;q)_{a-b}},  \quad
\text{ for } a, b \in \BN, b \le a.
\end{eqnarray*}

\subsection{Downward link diagram} Recall that a link diagram
$D \subset \BR^2$ is {\em alternating} if walking along it, the sequence
of crossings alternates from overcrossings to undercrossings. A
diagram $D$ is {\em reduced} if it is not of the form
$$
\psdraw{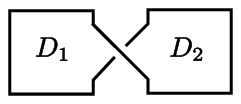}{1in}
$$
where $D_1$ and $D_2$ are diagrams with at least one crossing.

A {\em downward link diagrams of links} is an oriented link  diagram  in
the standard plane in general position (with
its height function) such that at every crossing the orientation of both
strands of the link is downward. A usual link diagram may not satisfy the
downward requirement on the orientation at a crossing. However, it is easy
to convert a link diagram into a downward one  by rotating the non-downward
crossings as follows:
$$
\psdraw{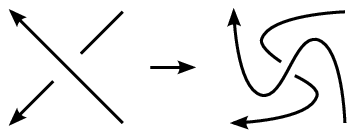}{1.5in}
$$

\subsection{Link diagrams and states}
\lbl{sub.morseknot}

Fix a downward link diagram $D$ of an oriented link $K$ with $c_D$ crossing.
Considering $D$ as a 4-valent graph, it has $2c_D$ edges.
A {\em state} of $D$ is a map
$$
r: \{\text{edges of} \,\, D \}  \to \BR
$$
such that at every crossing we have
$$ a+b = c+d,
$$
where $a,b,c,d$ are the values of $s$ of the edges incident to the crossing
as in the following figure
\be
\lbl{eq.abcd}
\psdraw{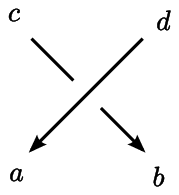}{0.6in} \qquad\qquad
\psdraw{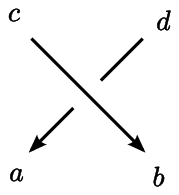}{0.6in}
\ee
The set $S_{D,\BR}$ of all states of $D$ is a vector space.
For a state $r\in S_{D,\BR}$ and a crossing $v$ of $D$ define
$$
r(v) = \sign(v)\, (a-d)\,,
$$
where as usual the sign of the crossing on the
left hand side of \eqref{eq.abcd} is positive and the sign of the one on
the right hand side is
negative. For a positive integer $n$, a state
$r \in S_{D,\BR}$ is called {\em $n$-admissible} if the values of $r$ are
integers in $[0,n]$ and $r(v)\ge 0$ for every crossing $v$.
Let $S_{D,n}$ be  the set of all $n$-admissible states.

\begin{remark}
\lbl{rem.Snp} Later we will prove that $\dim S_{D,\BR}= c_D+1$.
By definition, $S_{D,n}$ in 1-1 correspondence with the set
$n P_D \cap \BZ^{2c_D}$
of lattice points of $n P_D$ for a lattice polytope $P_D$ in $\BR^{2c_D}$
where $c_D$ is the number of crossings of $D$.
\end{remark}

\subsection{Winding number and its local weight}
\def\al{\alpha}
Suppose $\al$ is an oriented simple closed curve in the standard plane.
By the  winding number $W(\al)$ we mean the winding number of $\al$ with
respect to a point in the region
bounded by $\al$. Observe that $W(\al)=1$  if $\al$ is counterclockwise,
$-1$ if otherwise.

The winding number $W(\al)$ can be calculated by a local weight sum as
follows. A {\em local part} of $\al$ is a small neighborhood of a local
maximum or minmum.
For a local part $X$ define $W(X)=1/2$ if $X$ is winding counterclockwise,
$-1/2$ if otherwise. In other words, we have

\begin{equation*}
W\left(\psdraw{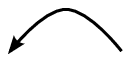}{0.5in}\right)=
W\left(\psdraw{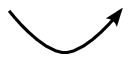}{0.5in}\right)=1/2,
\qquad
W\left(\psdraw{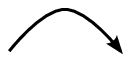}{0.5in}\right)=
W\left(\psdraw{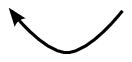}{0.5in}\right)=-1/2 \,.
\end{equation*}

The next lemma is elementary.

\begin{lemma}
 For every simple closed curve $\al$,
\be
\lbl{e91}
W(\al) = \sum _X W(X)\,,
\ee
where the sum is over all the local parts of $\al$.
\end{lemma}

\subsection{Local weights, the colored Jones polynomial, and their
factorization}
\lbl{sub.weights}

Consider the monoid
\begin{equation*}
\Zp=1+q \BZ[q]\,.
\end{equation*}
Fix a natural number $n \geq 1$ and a downward link diagram $D$.

{\em A local part of $D$} is a small neighborhood of
a crossing or a local extreme of $D$. There are six types of local parts
of $D$: two types of crossings
(positive or negative) and four types of local extrema (minima or maxima,
oriented clockwise, or counterclockwise):
\be
\lbl{f.6types}
\psdraw{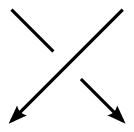}{0.5in} \qquad
\psdraw{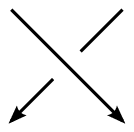}{0.5in} \qquad
\psdraw{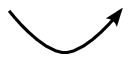}{0.5in} \qquad
\psdraw{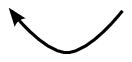}{0.5in} \qquad
\psdraw{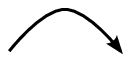}{0.5in} \qquad
\psdraw{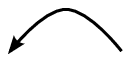}{0.5in}
\ee

For an $n$-admissible state $r$ and
a local part $X$, the weight $w(X,r)$ is defined by
\begin{equation*}
w(X,r)=\wlt(X,r) w_\succ(X,r)\,,
\end{equation*}
where $\wlt(X,r) \in \{\pm q^{m/4} \, | \, m \in \BZ\}$ is a monomial,
$w_\succ(X,r) \in \Zp$, and $\wlt(X,r)$ and $w_\succ(X,r)$ are given by
Table \ref{eq.wwsucc}.

{\tiny
\begin{table}[htpb]
\caption{The local weights $\wlt$ and $w_\succ$ of a state.}
\lbl{eq.wwsucc}
\text{
\begin{tabular}{|c|c|c|c|c|c|c|} \hline
& $\psdraw{pabcd.down}{0.6in}$ & $\psdraw{nabcd.down}{0.6in}$
& $\psdraw{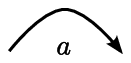}{0.5in}$ & $\psdraw{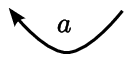}{0.5in}$
& $\psdraw{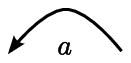}{0.5in}$ & $\psdraw{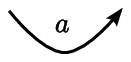}{0.5in}$
\\ \hline
$\wlt$ & $q^{(n+nd+nb -ab-dc)/2 }$ & $(-1)^{b-c}   q^{(-n- nb-n d +bd +ac - b+c)/2} $ & $ q^{-(2a-n)/4} $ & $ q^{-(2a-n)/4} $ & $ q^{(2a-n)/4}$ & $ q^{(2a-n)/4}$
\\ \hline
$\w_\succ$ & $(q;q)_{c-b} \binom{n-d}{a-d}_{q}   \binom{c}{c-b}_{q} $ & $(q;q)_{b-c}     \binom{n-c}{b-c}_{q}    \binom{d}{d-a}_{q} $ & $ 1$ & $ 1$ & $ 1$ & $ 1$ \\ \hline
\end{tabular}
}
\end{table}
}

For a local extreme point $X$ with the value of the state $a$, we have the  convenient formula
\begin{align*}
w(X,a)= q^{W(X) (2a-n)/2}\,.
\end{align*}
Let the weight of a state be defined by
$$
w(r) = \prod_X w(X,r)\,,
$$
where the product is over all the local parts of
$D$. Then the unframed version of the colored Jones polynomial of the link
$K$, each component of which is colored by the $n+1$-dimensional
$sl_2$-module, is given by

\be
\lbl{eq.statesum}
J_{K,n}(q)=\sum_{r \in S_{D,n}} w(r)\,,
\ee
where $S_{D,n}$ is the set of all $n$-admissible states of $D$.
For example, the value of the unknot is
$$
J_{\text{Unknot},n}(q)=
[n+1]:=\frac{q^{(n+1)/2} - q^{-(n+1)/2}}{q^{1/2}-q^{-1/2}}
\,.
$$
Note that $J_{K,0}(q)=1$ for all links and
$J_{K,1}(q^{-1})/J_{\text{Unknot},1}(q^{-1})$ is the Jones polynomial of $K$
\cite{Jo}. Since we could not find a reference for the state sum formula
\eqref{eq.statesum} in the literature, we will give a proof
in the Appendix.


\section{Alternating link diagrams and centered states}
\lbl{sec.comb}

In this section we will discuss the combinatorics of alternating
diagrams.

\subsection{Alternating link diagrams and $A$-infinite type}
\lbl{sub.alt.knot}

Recall that a link diagram $D$ gives rise to a polygonal complex structure
of $S^2= \BR^2 \cup \{ \infty\}$, and if $D$ is alternating and connected,
then the checkerboard coloring with colors $A$ and $B$ at each crossing
looks like Figure \ref{f.pntype}.

If $K$ is non-split, then $D$ is a connected graph. If $K$ is split, then
$D$ has several connected components.
We will say that an alternating diagram $D$ is {\em $A$-infinite} if the
point $\infty\in S^2$ is contained in an $A$-polygon of every connected
subgraph of $D$.
It is clear that by moving the connected components of $D$ around in $S^2$,
we can assume that $D$ is $A$-infinite. This will make the colors of
different connected components compatible.

We will use the following obvious property of an $A$-infinite alternating
link diagram: all the $B$-polygons are finite, i.e. in
$\BR^2= S^2\setminus \{\infty\}$.

For example, the left-handed trefoil given by the standard
closure of the braid $s_1^{-3}$ is $A$-infinite, whereas the right-handed
trefoil given by the standard closure of the braid $s_1^3$ is not.
Here $s_1$ is the standard generator of the braid group in two strands.

\subsection{The digraph $\bD$ of an alternating diagram $D$}
\lbl{sub.alt.knot.rotate}

Let $D$ be an oriented link diagram.
Recall that we consider $D$ also as a graph whose edges are oriented.
We say that an edge of a   $D$ is {\em of type $O$} if it begins
as an overpass, and {\em of type $U$} if it begins as an underpass.If $D$ is
alternating and one travel along the the link the edges alternate from type
$U$ to type $O$ and vice-versa.

For a link diagram $D$ let $\bD$ be the  directed graph on $\BR^2$ obtained
from the projection of $D$ on the plane by reversing the orientation of all
edges of type $O$. Of course the projection of $D$ is simply $D$ without the
over/under crossing information.

If $D$ is {\em downward alternating}, it is easy to see that $\bD$ is
obtained from $D$ by the following changing of orientations near
a crossing point,
\be
\lbl{eq.bD}
\psdraw{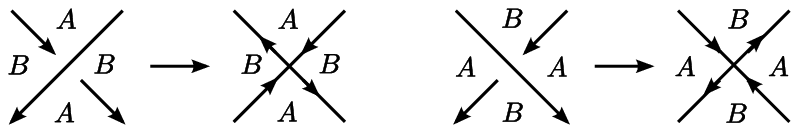}{4in}
\ee
i.e., if the crossing is a positive one, then the two left edges incident to
it get orientation reversed, and if the crossing is negative, then the two
right edges incident to it get orientation reversed.
We will retain the markings $A$ and $B$ for angles and regions of
complements of $\bD$.

At every vertex of $\bD$ (or a crossing of $D$) there are two ways to
smoothen the diagram. Following Kauffman \cite{Kf2} we call the
$A$-{\em smoothening} (resp., $B$-{\em smoothening})
the one where the two $A$-regions (resp., $B$-regions) get connected.
See the following figure for two examples of an $A$-smoothening.

\begin{equation*}
\psdraw{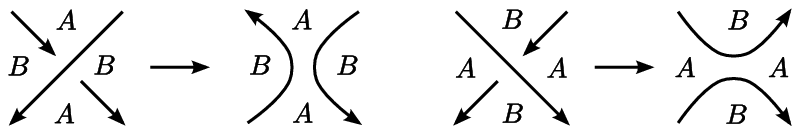}{4in}
\end{equation*}

Note that after either type of smoothening, the orientation of the edges of
$\bD$ is still well-defined.

\begin{remark}
\lbl{rem.resolution}
Doing an $A$-resolution (resp. $B$-resolution) on a vertex of $\bD$ is the
same as doing an $A$-resolution (resp. $B$-resolution) on the original
diagram $D$ in the sense of Kauffman \cite{Kf}. The advantage here, with
directed graph $\bD$ for the case of alternating links, is that the resulting
graph of any resolution is still oriented.
\end{remark}

Part (b) of the following lemma is where $A$-infinity is used in an essential
way.

\begin{lemma}
\lbl{l67}
Suppose $D$ is an alternating link diagram.

(a) To the right of every oriented edge of $\bD$ is an $A$-polygon, and to
the left of every oriented edge of $\bD$ is a $B$-polygon.

(b) Suppose $D$ is $A$-infinite, then every circle obtained from $\bD$ by
after doing $A$-resolution at every vertex of $\bD$  bounds a polygonal
region of type $B$. Moreover every such circle is
winding counterclockwise, i.e., it has winding number $1$.

(c) If $D$ is reduced, then each circle in (b) does not self-touch, i.e.,
the two arcs resulting from the $A$-resolution at one vertex do not belong to
the same circle.

\end{lemma}

\begin{proof}
(a) follows easily by inspecting the directions of the edges
and the markings of the regions at the two types of vertices  of $\bD$.

(b)  The boundaries of the $B$-polygons are exactly
the  circles obtained from $D$ after doing  $A$-resolution at every vertex
of $D$. Since the infinity region is not a $B$-type region, every circle does
bound a $B$-type region in the plane $\BR^2$. From part (a) it follows that
each circle, which is the boundary of a polygonal region of type $B$, is
counterclockwise.

(c) This is a well-known fact. A link diagram having the property that no
circle obtained after doing $A$-resolution at every crossing has a
self-touching point is known as an {\em $A$-adequate diagram}. In
\cite[Prop.5.3]{Lickorish} it was proved that every reduced alternating
link diagram is $A$-adequate.
\end{proof}

\subsection{Centered states}
\lbl{sub.centered}

Fix an alternating downward diagram $D$ with $c_D$ crossings and its
directed graph $\bD$. Recall that  $\cE(\bD)$ and $\cV(\bD)$ denote
respectively the set of oriented edges of $\bD$ and  the set of vertices of
$\bD$.

A {\em centered state} of $\bD$ is a map
$
s: \cE(\bD) \to \BR
$
such that  at every vertex $v$, we have
\be
a+d=b+c\,,
\lbl{c2}
\ee
with the convention that $a,b,c,d$ are the values of $s$
as indicated in the following figure

\be
\lbl{f.pntypeb}
\psdraw{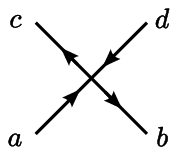}{0.8in} \qquad\qquad
\psdraw{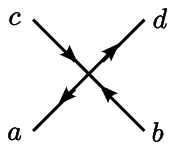}{0.8in}
\ee

For the above vertex we define
\be
\lbl{eq.sv}
s(v) = a+d=b+c\,,
\ee
thus extending $s$ to a map $s: \cE(\bD) \cup \cV(\bD)) \to \BR$.

Let $S_{\bD,\BR}$ and $S_{\bD,\BN}$ the sets of all centered states of $\bD$ with
values respectively in $\BR$ and in $\BN$.  For a fixed
positive integer $n$, define a map
\be
\lbl{eq.sbs}
S_{D,\BR} \longto S_{\bD,\BR} \qquad r \mapsto \hat r
\ee
by
$$
\hat r(e)=
\begin{cases} n-r(e) & \text{if the edge $e$ is of type $O$} \\
r(e) & \text{if $e$ is of type $U$}.
\end{cases}
$$
It is easy to see the map \eqref{eq.sbs} is a vector space isomorphism.
If $r \in S_{D,n}$, i.e., $r$ is $n$-admissible, then $\hat r$ is called
{\em $n$-admissible}. Let $S_{\bD, n}$ be  the set of all $n$-admissible
centered states.

To characterize $n$-admissible centered states let us introduce the
following norm for  $s\in S_{\bD,\BN}$:
$$
|s|= \max_{v \in \cV(\bD)} s(v)= \max_{u \in \cV(\bD) \cup \cE(\bD) } s(u)
$$
The following is a reformulation  of $n$-admissibility in terms of  centered
states.

\begin{lemma}
\lbl{l7a}
A centered state $s$ is $n$-admissible if and only if $s\in S_{\bD,\BN}$ and
$|s| \le n$. In other words,
$$
S_{\bD,n}= \{ s\in S_{\bD,\BN}\, : \,|s| \le n\}
$$
\end{lemma}

\begin{proof}
This follows immediately from the definition, since for any state $r$
and for every vertex $v$ of $\bD$ we have $r(v) = n- \hat r(v)$.
\end{proof}

It follows that if a centered state is $n$-admissible, then it is
$(n+1)$-admissible.


\section{Local weights in terms of centered states}
\lbl{sec.factor}

In this section we will give an explicit formula for the weight of an
centered state. It turns out that the state sum of the colored Jones
polynomial in terms of centered states has the important property of
separation of variables needed in the proof of the stability.
See Remark \ref{why}.

\subsection{Local weights of centered states and their factorization}
\lbl{sub.local.antistate}

For an $n$-admissible centered state $s= \hat r$, let us define
$w(s):= w(r)$. From the state sum of $w(r)$ we get the following state sum
for $w(s)$
\be
w(s) =\sum _{X} w( X, s)\,,
\lbl{e201}
\ee
where the sum is over all the local parts $X$ of $\bD$. Here a local parts
of $\bD$ is a neighborhood of either a vertex or an extreme point of $\bD$,
and the value of
$$
w(X,s)=\wlt(X,s)w_\succ(X,s)
$$
is obtained by replacing Table \eqref{eq.wwsucc} with Table
\eqref{eq.wwsucc2},

\begin{table}[htpb]
\caption{The local weights $\wlt$ and $w_\succ$ of centered states}
\lbl{eq.wwsucc2}
\text{
\begin{tabular}{|c|c|c|c|c|c|c|} \hline
& $\psdraw{barDpabcd.down}{0.6in}$ & $\psdraw{barDnabcd.down}{0.6in}$
& $\psdraw{locmaxk}{0.5in}$ & $\psdraw{locmink}{0.5in}$
& $\psdraw{locmax2k}{0.5in}$ & $\psdraw{locmin2k}{0.5in}$
\\ \hline
$\wlt$ & $ q^{\frac{n +ab + cd}2} $ & $(-1)^{n-a-d} q^{\frac{-n^2-2n +ac +bd +b+c} 2 } $ & $ q^{-(2a-n)/4} $ & $ q^{-(2a-n)/4} $ & $q^{(2a-n)/4} $ & $q^{(2a-n)/4} $
\\ \hline
$\w_\succ$ & $\wp(X)_{a,b}^{c,d} $ & $\wp(X)_{a,b}^{c,d} $ & $ 1$ & $ 1$ & $ 1$ & $ 1$ \\ \hline
\end{tabular}
}
\end{table}
where
\be
\lbl{eq.wabcd}
\wp(X)_{a,b}^{c,d}=
(q;q)_{n-a-d} \,\binom{n-d}{n-a-d}_{q} \, \binom{n-c}{n-a-d}_{q}
\ee

Note that  $w_\succ(X,s)$ is independent of the sign of the local crossing,
and takes the same value $1$ at all local extrema.
Hence, we use the notation $\wp(v,s)$ for
the right hand side of
\eqref{eq.wabcd}, where $v \in \cV$ is the involved vertex. The following is
a convenient way rewrite the value of $\wp(v,s)$.
\begin{lemma}  For a vertex $v$ in  \eqref{f.pntypeb} and $x= q^{n+1}$, we have
\be
\lbl{eq.wbs}
\wp \left(\psdraw{barDpabcd.down}{0.6in}\right)=
\wp \left(\psdraw{barDnabcd.down}{0.6in}\right)=
\frac{(x\, q^{-a-d})_\infty}{(x\, q^{-d})_\infty(x\, q^{-c})_\infty}
\frac{(q)_\infty}{(q)_a (q)_b}\,.
\ee
\end{lemma}

\begin{proof}
The identity follows from Equation \eqref{eq.wabcd},
and the following (easy to check) identities
\begin{eqnarray*}
{\binom nk}_q &=& \frac{(q^{n-k+1})_\infty}{(q)_k \, (q^{n+1})_\infty}
\\
(q)_k &=& \frac{(q)_\infty}{(q^{k+1})_\infty}.
\end{eqnarray*}
\end{proof}

\begin{remark}
The right hand side of Equation \eqref{eq.wbs} can also be written in
the following form:
\be
\lbl{eq.wbsalt}
\frac{(x\, q^{-a-d})_\infty}{(x\, q^{-d})_\infty(x\, q^{-c})_\infty}
\frac{(q)_\infty}{(q)_a (q)_b}=
\frac{(q)_\infty }{(q)_a (q)_b} \,
\frac{(x\, q^{-s(v)})_a \, (x\, q^{-s(v)})_b }{ (x\, q^{-s(v)})_\infty }
= \frac{(q)_\infty }{(q)_a (q)_b} \,
\frac{(x\, q^{-s(v)})_b}{(x\, q^{-d})_\infty} \,.
\ee
\end{remark}

\subsection{The functionals $P_0, P_1,Q, L_0,L_1$}
To study the power of $q$ in Table \eqref{eq.wwsucc2},
let us introduce the following functionals $P_0, P_1,Q, L_0,L_1$ on centered
states, defined by local weights as in Table \ref{t0}.

\begin{table}[htpb]
\caption{The definition of $L_0, L_1, Q, P_0$ and $P_1$.}
\lbl{t0}
\begin{tabular}{|c|c|c|c|c|c|c|} \hline
& $\psdraw{barDpabcd.down}{0.6in}$ & $\psdraw{barDnabcd.down}{0.6in}$
& $\psdraw{locmaxk}{0.5in}$ & $\psdraw{locmink}{0.5in}$
& $\psdraw{locmax2k}{0.5in}$ & $\psdraw{locmin2k}{0.5in}$
\\
\hline
$L_1$ & $0$ & $(b+c)/2 = (a+b+c+d)/4$ & $-a/2$ & $-a/2$ & $a/2$ & $a/2$
\\ \hline
$L_0$ & $0$ & $a+d=b+c$ & $0$ & $0$ & $0$ & $0$ \\ \hline
$Q$ & $(ab+ cd)/2$ & $(ac+bd)/2$ & $0$ & $0$ & $0$ & $0$ \\ \hline
$P_0$ & $0$ & $n$ & $0$ & $0$ & $0$ & $0$ \\ \hline
$P_1$ & $n/2$ & $-n^2/2-n$ & $n/4$ & $n/4$ & $-n/4$ & $-n/4$ \\ \hline
\end{tabular}
\end{table}
If $F$ is one of the functionals $P_1,P_2,Q,L_0,L_1$, and $s$ is a centered
state, then we define
$$
F(s) = \sum_{X} F(X,s)\,,
$$
where the sum is over all local parts $X$, with the value of $F$ at a
local part is given in Table \eqref{t0}.
These functionals are introduced so that for a local part $X$ with centered
state $s$ we have

\begin{eqnarray*}
\lbl{e5}
w(X,s) = (-1)^{P_0(X,s) + L_0(X,s)} \, q^{P_1(X,s)+ Q(X,s) + L_1(X,s)}\, \wp(X,s)
\,.
\end{eqnarray*}
From Equation \eqref{e201} we have
\be
\lbl{e5a}
w(s) = (-1)^{P_0(s)  + L_0(s) } \, q^{P_1(s) + Q(s)  + L_1(s)} \, \wp(s) \,.
\ee
Here $\wp(s) = \prod_{X} \wp(X,s)$, where the product is over all local
parts of $D$.
Note that  $\wp(X,s) \in \Zp$.  The functionals $L_0,L_1$ are linear
forms on $S_{\bD,\BR}$ and do not depend on $n$ in the sense that the value
of each of $L_0,L_1$ will be the same if we consider $s$ as an
$(n+1)$-admissible centered state instead of an $n$-state. The functional
$Q_2=Q+L_1$ is a quadratic form on $S_{\bD,\BR}$ not depending on $n$. The two
functionals $P_0, P_1$ depend only on $n$, i.e., if $s,s'$ are $n$-admissible
centered states, then $P_i(s)=P_i(s')$. Hence we will also write $P_i(n)$
instead of $P_i(s)$, for $i=0,1$.

\begin{lemma}
\lbl{l.degree}
We have

\be
J_{K,n}(q)= (-1)^{P_0(n)}\, q^{P_1(n)} \sum_{s \in S_{\bD,n}} F(q^{n+1},q,s),
\lbl{e24a}
\ee
where
\be
\lbl{eq.Fxqs2}
F(x,q,s)=(q)_\infty^{c_D} \, (-1)^{L_0(s)} \,  \,
\frac{q^{Q_2(s)}}{  \prod_{e\in \cE(\bD)}\,  (q)_{s(e)} }\,
 \frac{\prod_{v\in \cV(\bD)} \, \left(x q^{-s(v)}\right)_\infty
}{\prod_{e\in \cE(\bD)} \, \left(x q^{-s(e)}\right)_\infty }\,.
\ee

\end{lemma}

\begin{proof}

By \eqref{eq.wbs}  we have
\be
\prod_{v \in \cV(\bD)} \wp(v,s)= \frac{(q)_\infty^{c_D}}{
\prod_{v\in \cV(\bD)}
\, (q)_a (q)_b  }\,
 \frac{
 \prod_{v\in \cV(\bD)} \, \left(q^{n+1-s(v)}\right)_\infty
 }{\prod_{v\in \cV(\bD)}  ( q^{n+1-d})_\infty\, (q^{n+1-c})_\infty   }
\in \Zp\,.
\notag
\ee

 Here $a$ and $b$ (respectively $c$ and $d$) are the $s$-values of the two
lower (respectively upper) edges incident to $v$. When $v$ runs the set
$\cV$ of vertices, the  two lower edges of
 $v$ run the set $\cE$ of all edges, as do the two upper edges of $v$. Hence

\be  \prod_{v \in \cV(\bD)} \wp(v,s) = \frac{(q)_\infty^{c_D}}{
\prod_{e\in \cE(\bD)}\,  (q)_{s(e)} }\,
 \frac{\prod_{v\in \cV(\bD)} \, \left(q^{n+1-s(v)}\right)_\infty
 }{\prod_{e\in \cE(\bD)} \, \left(q^{n+1-s(e)}\right)_\infty }\,.
\lbl{e24b}
\ee

From Equation
\eqref{e5a} and $J_{K,n}(q)= \sum_{s \in S_{\bD,n}} w(s)$, we have

\begin{align*}
J_{K,n}(q)&= (-1)^{P_0(n)}\, q^{P_1(n)} \sum_{s \in S_{\bD,n}} (-1)^{L_0(s)}
\, q^{Q_2(s)} \, \prod_{v \in \cV(\bD)} \wp(v,s)\,,
\end{align*}
which is equal to the right hand side of \eqref{e24a} by identity
\eqref{e24b} and the definition of $F(x,q,s)$.
\end{proof}

\begin{remark}
\lbl{why}
(a) It is important for the stability that there is no mixing between
$n$ and $s$ in the formulas of the functionals $P_0,P_1,Q,L_0,L_1$.
In the states-sum using states in $D$, misxing occurs, and this
is the reason why we introduce centered states.

(b) The quadratic form $Q$ has the following simple description. Suppose
$\al$ is an $A$-angle of the digraph $\bD$, and the $s$-values of the two
edges of $\al$ are $a$ and $b$.
Define $Q(\al,s)= ab/2$. Then
\be Q(s) = \sum _\al Q(\al,s)\,,
\lbl{Q.anlge}
\ee
where the sum is over all $A$-angles $\al$.
\end{remark}


\section{Positivity of $Q_2$ and the lowest degree of the colored
Jones polynomial}
\lbl{Q2}

In this section we prove the copositivity of $Q_2:= Q+ L_1$ on the cone
$S_{\bD,\BN}$ and derive a formula for the lowest degree of the colored Jones
polynomial. Again we fix a reduced, alternating, $A$-infinite downward
diagram $D$  with $c_D$ crossings.

\subsection{A Hilbert basis for $S_{\bD,\BN}$:
elementary centered states}
\lbl{sub.elementary}

From its very definition, the set $S_{\bD,\BN}$ of $\BN$-valued centered
states of $\bD$ can be identified with the set
of lattice points of a lattice cone in $\BR^{2 c_D}$. In general, the set
of lattice points of a rational cone is a monoid, and a generating set
is called a {\em Hilbert basis} which plays an important role in
{\em integer programming}; see for instance \cite[Sec.13]{St} and also
\cite[Sec.16.4]{Sch}. Note that every element of a finitely generated
additive monoid is an $\BN$-linear combination of a Hilbert basis. Although
the natural number coefficients are not unique, this is not a problem
for applications.

The goal of this section is to describe a useful Hilbert basis for
$S_{\bD,\BN}$.

Recall that $\bD$ is a directed graph. Suppose $\gamma$ is a directed cycle
of $\bD$, i.e., closed path consisting of a sequence of distinct edges
$e_1,\dots e_n$ of $D$ such that the ending point of $e_j$ is the starting
points of $e_{j+1}$ (index is taken modulo $n$) and there is no repeated vertex
along the path except for the obvious case where the first vertex is also the
last vertex. An example of a cycle of $\bD$ is the boundary of a polygon
in the complement of $\bD$.

\begin{definition}
\lbl{def.elementary}
For a directed cycle $\gamma$ of $\bD$ let $s_\gamma$ be the function on the
set of edges of $\bD$ which assigns 1 to every edge of $\gamma$ and 0 to every
other edge. Such a centered state is called {\em elementary}, and $\gamma$
is called its support. Let $\calB$
denote the (finite set) of all elementary centered states of $\bD$.

For a polygon  $p\in \cP(\bD)$ the boundary $\partial p$ is a directed
cycle of $\bD$, and we will use the notation $s_p := s_{\partial p}$.
\end{definition}
From Lemma \ref{l7a} we see that $s_\gamma$ is an $n$-admissible centered
state for every $n \ge 1$.

\begin{lemma}
\lbl{lem.005}
$\calB$ is a Hilbert basis of $S_{\bD,\BN}$.

\end{lemma}

\begin{proof}
Let $s$ be a $\BN$-valued centered state of $\bD$.
Suppose $e$ is an oriented edge such that $s(e) >0$. At the ending vertex
$v$ of $e$
let $e'$ and $e''$ be the two edges which are perpendicular to $e$.
Inspection of Figure \eqref{eq.bD} shows that
$v$ is the starting vertex for both $e'$ and $e''$. Equation \eqref{c2}
shows that $s(e') + s(e'') \ge s(e)$. Hence one of them, say $s(e')>0$. This
means if $e$ is an edge with $s(e)>0$, we can continue $e$ to another edge
$e'$ for which $s(e') >0$. Repeating this process we can construct a
cycle $\gamma$ of $\bD$ such that the value of $s$ is positive on any edge
of $\gamma$. This means $s-s_\gamma$ is an $\BN$-valued centered state.
Induction completes the proof of the lemma.
\end{proof}
\begin{remark}
It is easy to see that any $s \in \calB$ is not a $\BN$-linear combination
of the other elements in $\calB$. Thus there is no redundant element in
$\calB$. Of course $\calB$ is linearly dependent over $\BR$ (or over $\BZ$),
and we will extract a $\BR$-basis from the set $\calB$ later.
\end{remark}

\subsection{Values of $L_1$ and $Q$ on elementary centered states}

Suppose $\gamma$ is a directed cycle of $\bD$ and $v$ is a vertex of
$\gamma$.  Among the four edges of $\bD$ incident to $v$, the two edges of
$\gamma$ are not two opposite edges because of the
orientation constraint, see  \eqref{eq.bD}. In other words, at each vertex
$v$, $\gamma$ is an angle.
We say that a vertex $v$ of $\gamma$ is of type $A$ or $B$ according as
the two edges of $\gamma$ at $v$ form an angle of type $A$ or $B$.
Let
$N_{\gamma,A}$ be the number of vertices of $\gamma$ of type $A$. The fact that
$D$ is reduced
is used in the proof of part (b) of the next lemma.

\begin{lemma}
\lbl{lem.L1}
Suppose $s=s_\gamma \in \calB$ is an elementary centered state.

(a) We have
\be
\lbl{e151}
L_1(s) = W(\gamma) +\frac{1}{2}N_{\gamma,A}
 \ee
(b)  Moreover,
$L_1(s)  \ge  0$,
and $L_1(s)=0$ if and only if  $\gamma$
is clockwise and has exactly two vertices of type $A$.

\end{lemma}

\begin{proof}
(a)
For a local part $X$ of $\bD$, let $\gamma_X = \gamma\cap X$. Clearly
$L_1(X,s)=0$ if $\gamma_X= \emptyset$. If $X$ is a small neighborhood
of a vertex of $\bD$, then $\gamma_X$ is two sides of an angle of $\gamma$,
and we will smoothen $\gamma_X$ at the corner to get an oriented smoothed arc.
See Row 1 and Row 2 of Table \ref{t1}  for various $X$ and smoothened
$\gamma_X$. In the table, $X$ is a small neighborhood of a vextex. The two
edges incident to the vertex with label 1 belong to $\gamma$. The marking
$A$ or $B$ at one of the angles of $X$ indicates the type of the vertex,
which appear in Row 3. In Row 3 we also indicate the sign of the crossing
of $X$ (as it appeared originally in $D$); this makes the computation of
$L_1$ easier.

\begin{table}[htpb]
\caption{The calculation of $L_1, W, Q$.}
\begin{tabular}{|c|c|c|c|c|c|c|c|c|} \hline
$X$ & $\psdraw{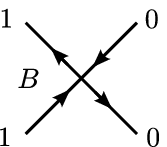}{0.5in}$ & $\psdraw{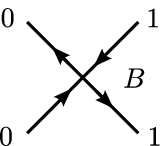}{0.5in}$
    & $\psdraw{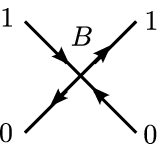}{0.5in}$ & $\psdraw{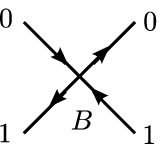}{0.5in}$
    & $\psdraw{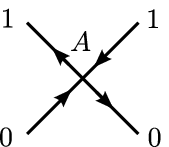}{0.5in}$ & $\psdraw{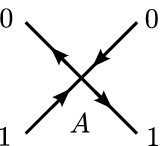}{0.5in}$
    & $\psdraw{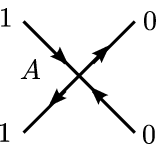}{0.5in}$ & $\psdraw{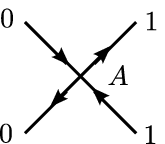}{0.5in}$
\\ \hline
$\gamma_X $ & \empty \hskip -.6 cm $\psdraw{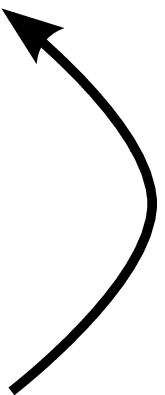}{0.15in}$  &$ \empty \hskip .6 cm \psdraw{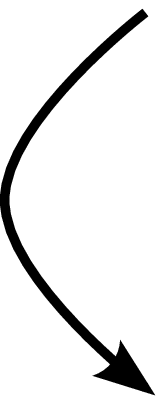}{0.15in}$  &  $\psdraw{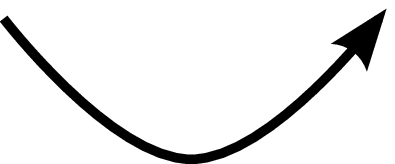}{0.4in}$ & $\psdraw{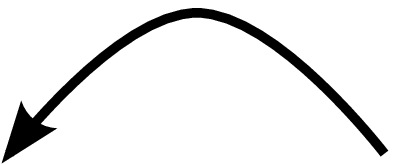}{0.4in}$ & $\psdraw{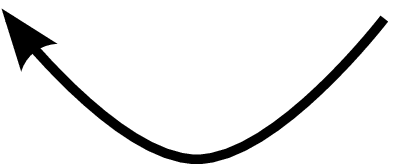}{0.4in}$ & $\psdraw{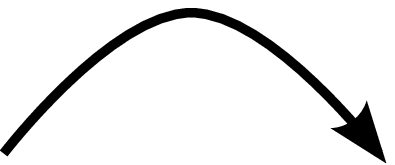}{0.4in}$&
$\empty \hskip -.6 cm \psdraw{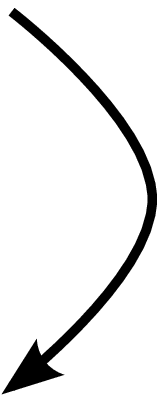}{0.15in}$ & \empty \hskip .6 cm $\psdraw{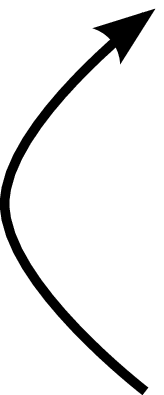}{0.15in}$
\\ \hline
vertex type & $B,+$& $B,+$ & $B,-$ & $B,-$ & $A,+$ & $A,+$& $A,-$& $A,-$  \\ \hline
$L_1(X,s_\gamma)$ & 0 & 0  & $1/2$ & $1/2$ & 0 &  0 & $1/2$& $1/2$\\ \hline
$W(\gamma_X)$ & 0 & 0 & $1/2$ & $1/2$ & $-1/2$  & $-1/2$  & 0& 0 \\ \hline
$Q(X)$ & 0 & 0 & $0$ & $0$ & $1/2$  & $1/2$  & 1/2& 1/2 \\ \hline
\end{tabular}
\lbl{t1}
\end{table}

We define $W(\gamma_X)$ to be its local winding number if $\gamma_X$
contains a local extreme point, 0 otherwise. Rows 4 and 5 of Table \ref{t1}
gives the values of $L_1(X,s_\gamma)$ and $W(\gamma_X)$.
From the table, together with the obvious case when $X$ is a neighborhood of
a local extreme point of $\bD$, we have
$$
L_1(X,s_\gamma)=
\begin{cases}
W(\gamma_X) + \frac{1} 2 \quad &
\text{if $X$ is a vertex of $\gamma$ of type $A$}\\
W(\gamma_X)  &\text{otherwise}.
\end{cases}
$$
Summing up the above identity over all the local parts $X$ and using
\eqref{e91}, we get \eqref{e151}.

(b) Case 1: $\gamma$ is counterclockwise. From \eqref{e151} we have
$L_1(s) \ge W(\gamma)=1 >0$. In this case $L_1$ is strictly positive.

Case 2:  $\gamma$ is clockwise. Then $L_1(s) = -1 + N_{\gamma,A}/2$. We will
show that $N_{\gamma,A} \ge 2$.

If $N_{\gamma,A}=0$, then $\gamma$ is one of the circles obtained from $\bD$ by
doing $A$-resolution at every vertex. By part (b) of Lemma \ref{l67},
$\gamma$ is counterclockwise. Thus $N_{\gamma,A}\neq 0$ if $\gamma$ is clockwise.

Suppose $N_{\gamma,A}=1$, i.e. $\gamma$ has exactly one  vertex of type $A$,
say  $v$; all other
vertices of $\gamma$ are of type $B$. If one does $A$-resolution at every
vertex of $\bD$, then $\gamma \setminus \{v\}$ is
part of one of the resulting circles, and this circle has a self-touching
point at $v$. This is impossible if the diagram $D$ is reduced, see part (c)
of Lemma \ref{l67}. Thus $N_{\gamma,A}\neq 1$.

We have shown that if $\gamma$ is clockwise then $N_{\gamma,A} \ge 2$. Hence
$L_1(s) = -1 + N_{\gamma,A}/2 \ge 0$, and equality happens if and only
$N_{\gamma,A}=2$.

\end{proof}
\begin{remark}\lbl{rem.adequate}
 We see that for the proof of part (b), we needs only the fact that $D$
is $A$-adequate.
\end{remark}

\begin{lemma}
\lbl{lem.Q}
\rm{(a)} For all $\BN$-valued centered states $s$ and $s'$ we have
\be
\lbl{eq.Qrr}
Q(s+s') \geq Q(s)+Q(s')
\ee
\rm{(b)} Suppose $s=s_\gamma\in \calB$ is elementary centered state. Then
\be
 Q(s) = \frac {N_{\gamma,A}}2.
 \lbl{e151a}
 \ee
It follows that $Q(s) \ge 0$,  with equality if and only if $\gamma$ is the
boundary of a polygonal region of type $B$.
\end{lemma}

\begin{proof}
(a) Since $Q$ is defined by an expression with positive coefficients, we
have $Q(s+s') \geq Q(s)+Q(s')$.

(b) Row 6 of Table \ref{t1} shows that every vertex of type $A$ of $\gamma$
contributes $1/2$ to the value of $Q$, while others contribute 0. Hence
$
Q(s) = \frac {N_{\gamma,A}}2.
$
\end{proof}

\subsection{Copositivity of $Q_2$}
\lbl{s101}

Recall $Q_2=Q+L_1$.

\begin{proposition}
\lbl{prop.Q+L1} (a) For $s,s' \in S_{\bD,\BN}$, we have
$$ Q_2(s + s') \ge Q_2(s) + Q_2(s').$$

(b) If $s= \sum_{j=1 }^l m_j s_j$, where $s_j\in \calB$ and $m_j \in \BN$, then
\be Q_2(s) \ge \sum_j m_j \ge |s|.
\lbl{e330}
\ee
 In particular,
$Q_2$ is copositive in the cone $S_{\bD,\BN}$, i.e., for every
$s \in S_{\bD,\BN}$, $Q_2(s) \geq 0$  and equality happens if and if only $s=0$.
\end{proposition}

\begin{proof} \
(a) follows immediately form
 Lemma \ref{lem.Q}(a), noting that $L_1(s+s')= L_1(s) + L_1(s')$.

(b) The second inequality of \eqref{e330} follows immediately from the
definition.

From part (a) one needs only  to prove the first inequality of \eqref{e330}
for $s\in \calB$ an elementary centered state with support $\gamma$. By
\eqref{e151} and \eqref{e151a},
$$ Q_2(s) = W(\gamma) + N_{\gamma,A}.$$
In particular, $Q_2(s)$ is an integer.

By Lemmas
\ref{lem.L1}(b) and \ref{lem.Q}(b), we have $ L_1(s) + Q(s) \ge 0$, and
equality happens
only when $ L_1(s) = Q(s) = 0$. However, if $L_1(s) = 0$, then by Lemma
\ref{lem.L1}(b), $N_{\gamma,A} =2$, and then $Q(s) = N_{\gamma,A}/2=1 >0$. Thus,
we have proved that if $s$ is an elementary
 centered state, then $Q(s) >0$.  Since $Q(s)\in \BZ$, we have $Q(s) \ge 1$.
\end{proof}

\begin{remark}
In general $Q(s), L_1(s) \in \frac12\BZ$. In the proof we
shown that $Q_2(s)= Q(s) + L_1(s) \in \BZ$ for  any elementary centered
state $s$. One can also show that
$Q_2(s)\in \BZ$ for all $s \in S_{\bD,\BN}$. This can be deduced from the fact
that $\hat J_{K,n}(q) \in \BZ[q]$, see the discussion on fractional powers of
$J_{K,n}$ in \cite{Le_Duke}.
\end{remark}

\subsection{The lowest degree of the colored Jones polynomial}
\begin{proposition}
\lbl{c.degree}
(a) The minimal degree of $q$ in of $J_{K,n}(q)$ is
$$
P_1(n) = \frac{n}{2} c_+ - \frac{n^2 +2n}{2} c_- - \frac{n}{2} \sum_M  W(M)\,,
$$
where the last sum is over all the local extreme points of $D$.

(b) With $F(x,q,s)$ defined by \eqref{eq.Fxqs2}, we have
\be
 \lbl{e242}
\hat J_{K,n}(q)=
\sum_{s \in S_{\bD,n}} F(q^{n+1},q,s).
\ee
\end{proposition}

\begin{proof}(a)
By \eqref{e24a}, the minimal degree of $q$ in $w(s)$ is
$P_1(n) + Q_2(s)$. When $s\neq 0$, Proposition \ref{prop.Q+L1}
implies that $Q_2(s) >0$. Hence the
smallest degree of $J_{K,n}(q)$ is $P_1(n)$. From the values of
$P_1(X,s)$ in Table \ref{t0} we see that
$P_1(n) = \frac{n}{2} c_+ - \frac{n^2 +2n}{2} c_- - \frac{n}{2} \sum_M  W(M)$.

(b) follows easily from part (a) and \eqref{e24b}.
\end{proof}

The value of $P_0$ in Table \ref{t1}, Equation \ref{e24a}, and
Proposition \ref{c.degree} imply the following.

\begin{corollary}
We have
\be
J_{K,n}(q)= (-1)^{n c_-} \, q^{P_1(n)}\, \hat J_{K,n}(q).
\lbl{eq.decomposition}
\ee
\end{corollary}

\begin{remark}
\lbl{rem.mindeg}
The minimal degree of the colored Jones polynomial $J_{K,n}(q)$ had been
calculated using the Kauffman bracket skein module, and is given by
$P_1'(n):=\frac{n}{2} c_+ - \frac{n^2 +n}{2} c_-  - \frac{n}{2} s_A$, where $s_A$
is the number of circles obtained from $\bD$ by doing $A$-resolution at
every vertex; see \cite[Proposition 2.1]{Le}\footnote{Note that the
framing of $K$ in \cite{Le} is different.}. Our result implies that
$P_1(n)=P_1'(n)$.
We will give a direct proof of this identity in the Appendix.
Note also that $s_A-c_+ = \sigma +1$ (see \cite{Tu3,Murasugi}),
where $\sigma$ is the signature of the link. Hence the lowest degree of $q$
is given by $- \frac{n^2 +n}{2} c_-  -\frac n2 (\sigma +1)$.
\end{remark}


\section{From $\bD$ to the dual graph $D^*$}
\lbl{sec.bDtoD*}

In this section we connect the centered states on $\bD$ with the admissible
colorings of the dual graph $D^*$. The main idea of this section is
summarized in the following figure. If a crossing has coloring $a,b,c,d$
at the four regions of it counterclockwise, then there is a coloring
of the four arcs such that the sum of the colors of the two overarcs is
equal to the sum of the colors of the two underarcs:
\begin{figure}[htpb]
$$
\psdraw{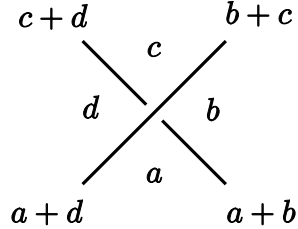}{1.5in}
$$
\caption{From a coloring of the regions to a coloring of the arcs.}
\lbl{f.4regions.arcs}
\end{figure}

Recall from Section \ref{sub.nahmalt} that $D^*$ is the dual graph of
$\bD$, considered as an unoriented graph. We have defined the Tait graph
$\cT$, the lattice $\Lambda_0$ with its subsets $\Adm, \Adm(n)$,
and functions $L, Q$ on $\Lambda_0$. Note that one does not need to bring
$D$ to a downward position by twisting in small neighborhood of
crossing points in order to construct $D^*$.

For $\lambda\in \Lambda_0$ let $\tau(\lambda): \cE(\bD)\to \BR$ be the
linear map defined by
 $$\tau(\lambda)(e)= e^*(\lambda),$$
where $e^*\in \cE(\D^*)$ is the dual edge of $e$.
Then $\tau: \Lambda_0 \otimes \BR \to S_{\bD,\BR}$ is a $\BR$-linear map.

\begin{proposition}
\lbl{p.210}
\rm{(a)} The map $\tau: \Lambda_0 \otimes \BR \to S_{\bD,\BR}$ is
vector space isomorphism.
\newline
\rm{(b)} $\tau$ maps $\Adm$ and $\Adm(n)$ isomorphically onto
respectively $S_{\bD,\BN}$ and $S_{\bD,n}$.
\newline
\rm{(c)} We have
\begin{align}
L_1(\tau(\lambda))& = L(\lambda) \lbl{eq.lambda}\\
Q(\tau(\lambda))& = Q(\lambda)  \lbl{eq.lambda2}
\end{align}
\newline
\rm{(d)} For every centered state $s$, we have
\be
\lbl{eq.lambda0}
L_0(\tau(s)) \equiv 2L(s) \pmod 2.
\ee
\end{proposition}

\begin{proof} (a)
Fix $s\in S_{\bD,\BR}$. We will show that the equation
\be \tau(\lambda)=s
\lbl{eq.745}
\ee
has one and exactly one solution $\lambda \in \Lambda_0$ such that. This
will prove the bijectivity of $\tau$.

With the basis $\bb:=\cV(D^*) \setminus \{v_\infty\}$ of $\Lambda_0$, every
$\lambda\in \Lambda_0$ has a unique presentation $\lambda= \sum_{v\in \cV(D^*)}
k_v v$ with $k_v=0$ for $v= v_\infty$. We need to solve for $k_v, v \in \bb$
from Equation \eqref{eq.745}.

Equation \eqref{eq.745} is the same as the following linear system of $2c_D$
equations: For  every edge $e^*\in \cE(D^*)$ whose end points are $v$ and $v'$,
\be
\lbl{eq.755}
k_{v} + k_{v'} = s(e) \,.
\ee
If $k_v$ is known, and $v'$ is connected to $v$ by an edge, then there is
only one possible value for  $k_{v'}$, namely $k_{v'}= s(e)- k_v$. We call
such $k_{v'}$ the extension of the value $k_v$ at $v$ along the edge $e^*$.
Since the graph $D^*$ is connected, and $k_{v_\infty}=0$, we see that there is
at most one solution $\lambda \in \Lambda_0$ of \eqref{eq.745}.

Now let us look at the existence of solution of \eqref{eq.745}. Given
$v\in \cV(D^*)$, $a\in \BR$, and a path $\al$ of the graph $\D^*$
connecting $v$ to $v'\in \cV(\D^*)$, there is only one way to extend
$k_v=a$ at $v$ to $v'$ along the path $\al$. Denote by $\lambda_{\al,a}(v')$
the value at $v'$ of this extension. When $\al$ is a closed path, i.e.
$v'=v$, let $\Delta(\al,a)= \lambda_{\al,a}(v')- a$. We will show that
$\Delta(\al,a)=0$ for any closed path $\al$. This will prove the
existence of the solution.

On $\BR^2$, the closed path $\al$ encloses a region $R$. When the region
is just a polygon of $D^*$ (which must be a quadrilateral), the fact that
$\Delta(\al,a)=0$ follows easily from \eqref{c2}.
For general closed path $\al$, since $\Delta(\al,a)$ is the sum of
$\Delta_{\al_j,a_j}$, where $\al_j$'s are the boundaries of all the polygons
of $D^*$ in $R$, we also have $\Delta(\al,a)=0$.

The above fact shows that if we begin with $k_{v_\infty}=0$, we can uniquely
extend $k_v$ to all vertices of $D^*$, and obtain in this way an inverse
of $s$.

The proof actually shows that $\tau$ is a $\BZ$-isomorphism between
$\Lambda_0$ and $S_{\bD\BZ}$.

(b) Because $\tau(\lambda)(e) = \lambda(e^*)$, this follows easily from
the definitions.

(c) To prove \eqref{eq.lambda}, it is enough to consider the case
$\lambda=v\in \bb=\cV(D^*) \setminus \{v_\infty\}$, a basis vector. Let
$p=v^*\in \calP(D)$ be the dual  polygon.
From the definition we have $\tau(v)= s_{p}$, where $s_p$ is the elementary
centered state with support the boundary of $p$. Now the identity
$L_1(\tau(v)) = L(v)$ follows from  the value
 of $L_1$ given in Lemma \ref{lem.L1} and the definition of $L$. Actually,
the definition of $L$ was built so that \eqref{eq.lambda} holds.

Let us turn to \eqref{eq.lambda2}.  To show that two quadratic forms on a
vector space are the same it is enough to show that
they agree on the set $v+v'$, where $v,v'$ are elements in a basis of the
vector space. A basis of $\Lambda_0$ is $\cV(D^*)\setminus \{v_\infty\}$.
Hence we need to check that if $v_1,v_2\in \cV(D^*)\setminus \{v_\infty\}$,
\be
\lbl{e103d}
Q(\tau(v_1 + v_2)) = Q(v_1 + v_2)\,.
\ee
There are three cases to consider: both $v_1, v_2$ are $A$-vertices, both
are $B$-vertices, and exactly one of them is an $A$-vertex. In each case,
the identity  \eqref{e103d} can be verified easily.
Actually, the matrix $\cQ$ in Introduction was built so that \eqref{e103d}
holds.

(d) We only need to check \eqref{eq.lambda0} for $s=\tau(v), v \in \bb$.
Let $p=v^*\in \cV(\bD)$ be the dual  polygon. We already saw that
$\tau(v)= s_p$.
From the definition of $L_0$ given by Table \ref{t0}, we have that
$L_0(s_p)$ is the number of negative vertices of $p$. Here a vertex is
negative if it is negative as a crossing of the link diagram $D$.

There are two cases.

{\bf Case 1:}
$p$ is a $B$-polygon. Suppose $\oor$ is an arbitrary orientation on edges
of $p$. A vertex $v$ of $p$ is $\oor$-incompatible if the orientations of
the two edges incident to $v$ are incompatible, i.e. the two incident
edges are both going out from $v$ or both coming in to $v$. Let $f(\oor)$
be the number of all $\oor$-incompatible vertices. It is easy to see that
if $\oor'$ is obtained from $\oor$ by changing the orientation at exactly
one edge, then $f(\oor)= f(\oor') \pmod 2$. It follows that
$f(\oor) =0 \pmod 2$ for any orientation $\oor$, since if we orient all
the edges counterclockwise then $f=0$.

Let the orientation of $D$ on the edges of $p$ be denoted by $\oor_D$. By
inspection Figure \eqref{eq.bD} one sees that a vertex $v$ of $p$ is a
negative crossing if and only $v$ is $\oor_D$-incompatible.
Thus $L_0(s_p)= f(\oor_D)$, which is even by the above argument. On the
other hand, $2L_1(s_p)= 2$ by Lemma \ref{lem.L1}.

{\bf Case 2:}
$p$ is an $A$-polygon. By inspection Figure \eqref{eq.bD} one sees that
a vertex $v$ of $p$ is a positive crossing if and only $v$ is
$\oor_D$-incompatible. This means $L_0(s_p)= \deg(v) - f(\oor_D)
\equiv \deg(v) \pmod 2$, where $\deg(v)$ is the number of vertices of $p$,
which is equal to the degree of $v$ in the graph $D^*$. By Lemma
\ref{lem.L1}, $2L_1(s_p) = -2 + \deg(v)$. Hence we also have
\eqref{eq.lambda0}.
\end{proof}

\begin{corollary}
The dimension of $S_{D,\BR}$ (or $S_{\bD,\BR}$) is $c_D+\ell$, where $\ell$
is the number of connected components of the graph $D$.
\end{corollary}

\begin{remark}
\lbl{rem.conerays}
One can show that the the integer-valued admissible colorings of $D^*$
are the lattice points in a $2 c_D$ dimensional cone with $2 c_D$
independent rays.
\end{remark}


\section{0-stability}
\lbl{sec.0stability}

In this section we give a proof of the 0-stability of colored Jones
polynomial of alternating link and Theorem  \ref{thm.1}, which describes
the 0-limit as a generalized Nahm sum.

\subsection{Expansion of $F$ and adequate series}
\lbl{sub.adseries}

\begin{definition}
We say that a series $G(x,q)=\sum_{m=0}^\infty a_m(q) x^m \in \BZ((q))[[x]]$ is
$x$-{\em adequate} of order $\le t$ if $G(xq^t,q) \in \BZ[[q]][[x]]$, i.e.
for every $m$, we have
$$
\mindeg_q(a_m(q)) \ge -mt \,.
$$
\end{definition}

\begin{lemma}
\lbl{lem.x}
\rm{(a)} For every $t \in \BN$, the set of $x$-adequate series of order
$\le t$ is a subring of $\BZ((q))[[x]]$.
\newline
\rm{(b)} If $G(x,q)$ is $x$-adequate of order $\le t$, then $x$-adequate of
order $\le t'$ for any $t' \ge t$.
\newline
\rm{(c)} If $G(x,q)$ is $x$-adequate of order $\le t$, then the series
$f_n(q)=G(q^{n},q)$ converges in the $q$-adic topology and defines an element in
$\BZ[[q]]$ for every $n >t$.
\newline
\rm{(d)} The sequence $(f_n(q))$ is stable 
and its associated series $F_f(x,q)$ satisfies $F_f(x,q)=G(x,q)$.
\end{lemma}

\begin{proof}
Parts (a), (b) and (c) follow easily from the definition of an $x$-adequate
series.

For (d), let $G(x,q)=\sum_{m=0}^\infty a_m(q) x^m \in \BZ((q))[[x]]$ and
define $\Phi_k(q)=a_k(q)$ for all $k \in \BN$. Then, we have for $n>t+1$
$$
q^{-k(n+1)}\left(f_n(q)-\sum_{j=0}^k \Phi_j(q) q^{(n+1)j}\right) =
q^{-k(n+1)} \sum_{m=k+1}^\infty a_m(q) q^{(n+1)m} \,.
$$
The minimum degree of the summand is bounded below by
$$
f(m)=-k(n+1)-mt+(n+1)m \,.
$$
Since $f(m)$ is a linear function of $m$ and the coefficient of $m$ in
$f(m)$ is $n+1-t>0$, it follows that
$$
f(m) \geq f(1)=n+1-t(k+1) \,.
$$
Thus,
$$
q^{-k(n+1)}\left(f_n(q)-\sum_{j=0}^k \Phi_j(q) q^{(n+1)j}\right)
\in q^{n+1-(k+1)t}\BZ[[q]]
$$
which implies (d).
\end{proof}

Recall that for a centered state $s$, $F(x,q,s)$ defined by
\eqref{eq.Fxqs2}, satisfies

\begin{eqnarray}
\notag
F(x,q,s) &=&  q^{Q_2(s)} \tF(x,q,s) \\
\tF(x,q,s) &:=&
\frac{(-1)^{L_0(s)}(q)_\infty^{c_D} }{  \prod_{e\in \cE(\bD)}\,  (q)_{s(e)} }\,
 \frac{\prod_{v\in \cV(\bD)} \, \left(x q^{-s(v)}\right)_\infty
}{\prod_{e\in \cE(\bD)} \, \left(x q^{-s(e)}\right)_\infty }\,.
\lbl{eq.tildeF}
\end{eqnarray}
Using the well-known identities (see e.g \cite{Kac})

\begin{align}
(x)_\infty = \sum_{j=0}^\infty \frac{(-1)^j \, q^{\binom j2}}{(q)_j} x^j, \qquad
\frac1{(x)_\infty} =  \sum_{j=0}^\infty \frac{x^j }{(q)_j}
\lbl{eq.773}
\end{align}
we can expand $\tF$ into power series in $x$,
\be
\lbl{Fxqs}
\tF(x,q,s) =
\sum_{m=0}^\infty a_m(q,s) x^m \in \BZ((q))[[x]]\,.
\ee
The negative powers of $q$ in $a_m(q,s)$ come from the negative powers
$q^{-s(v)}, q^{-s(e)}$ that appear in the expression of $\tF$. Since
$|s| \ge \max(s(v), s(e))$, we have the following.

\begin{lemma}
\lbl{lem.amb}
For every $s\in S_{\bD,\BN}$, $\tF(x,q,s)$ is $x$-adequate of order $\le |s|$.
\end{lemma}

\subsection{Proof of $0$-stability}
\lbl{sub.thm1}

Now we show that $\hat J_{K,n}(q)$ is 0-stable, and identify its 0-limit.
Recall $F(x,q,s)$ given by \eqref{eq.Fxqs2} or \eqref{eq.tildeF}.
By \eqref{e242}
$$
\hat J_{K,n}(q) = \sum_{s\in S_{\bD,n}} F(x,q,s)\big|_{x=q^{n+1}}\, .
$$
Hence we expect that the 0-limit is
\be
\lbl{e501}
\Phi_{0}(q):= \sum_{s\in S_{\bD,\BN}} F(0,q,s) \,.
\ee

We have
\be
F(0,q,s)= (q)_\infty^{c_D} \, (-1)^{L_0(s)} \, \,
\frac{   q^{Q_2(s)}}{  \prod_{e\in \cE(\bD)}\,  (q)_{s(e)} }.
\lbl{e503}
\ee

Part (b) of Proposition \ref{prop.Q+L1} shows that the  right hand side of
\eqref{e501} is regular, and defines an element in $\BZ[[q]]$. We will show
that
\be
\lbl{eq.0limit1}
\hat J_{K,n}(q)-\Phi_0(q) \in q^{n+1}\BZ[[q]]
\ee
for all $n$. This certainly implies that the 0-limit of $\hat J_{K,n}(q)$
exists and is equal to $\Phi_{0}(q)$.  We have
\begin{eqnarray}
\hat J_{K,n}(q)-\Phi_0(q) &=&  \sum_{s: |s| \leq n}
\big[F(q^{n+1},q,s) - F(0,q,s)\big] -  \sum_{s: n < |s|}
F(0,q,s) \,.
\lbl{e504}
\end{eqnarray}

By part (b) of Proposition \ref{prop.Q+L1}, $Q_2(s) \geq |s|$. Then
\eqref{e503} implies that $F(0,q,s) \in q^{|s|}\BZ[[q]]$, and hence
the second sum on the right hand side of \eqref{e504} is in $q^{n+1}\BZ[[q]]$.

Let us look at the first term. Using the expansion \eqref{Fxqs}, we have
\begin{align}
F(q^{n+1},q,s) - F(0,q,s) & =
 q^{Q_2(s)} \,
\sum_{m=1}^\infty a_m(q,s) q^{m(n+1)} \notag \\
& =
\sum_{m=1}^\infty \left[ a_m(q,s)\, q^{m|s|}\right] \, q^{f(m)}, \lbl{e505}
\end{align}
where $f(m)= Q_2(s) + m(n+1)-m|s|= m(n+1-|s|) + Q_2(s)$, which is linear in
$m$. Since $n \ge |s|$, $f(m)$ achieves minimum when $m=1$:
$$ f(m) \ge f(1) = n+1 -|s| + Q_2(s) \ge n+1.$$

By Lemma \ref{lem.amb}, $a_m(q,s)\, q^{m|s|}$ has only non-negative power of $q$.
It follows that the  right hand side of \eqref{e505} belongs to
$q^{n+1}\BZ[[q]]$. This completes the proof of Equation \eqref{eq.0limit1}.
\qed

\begin{remark}
\lbl{rem.0limit1}
Equation \eqref{eq.0limit1} is stronger than 0-stability, and implies that
for every $m \in \BN$, the coefficient of $q^m$ in $\hat J_{K,n}(q)$
is independent of $n$ for all $n>m$.
\end{remark}

\subsection{End of the proof of Theorem \ref{thm.1}}
\lbl{sub.end.pf.thm1}

To complete the proof of Theorem \ref{thm.1},
it remains to prove that the right hand side of \eqref{e501} is equal to
that of \eqref{eq.nahm0}. This follows from
Proposition \ref{p.210}. \qed

\begin{remark}
 The fact that $D$ is reduced is used only in the proof of Lemma
\ref{lem.L1}. As seen in Remark \ref{rem.adequate}, Lemma \ref{lem.L1}
holds if $D$ is $A$-adequate, hence Theorem \ref{thm.1} holds
if $D$ is not necessarily reduced, but $A$-adequate.
\lbl{rem.ade}
\end{remark}

\subsection{Proof of Corollary \ref{cor.thm2.1}}
\lbl{sub.cor.thm2.1}

Fix a complex number $q$ with $|q|=a <1$.
We only need to show that the sum on the right hand side of  \eqref{e501}
is  absolutely convergent.

Choose $ 0 < \ve < 1-a$ such that
\be
a < (a+\ve)^{2c_D}. \lbl{e508}
\ee
This is possible by continuity since if $\ve=1-a$, the right hand side
of the above inequality is $1$ and $a<1$. Since $\lim_{j\to \infty }
(1-q)^j=1$, $|1-q^j| > a+\ve$ for $j$ big enough. It follows that
there is constant $C_1 >0$ such that for every $n$,
\be
\lbl{e506}
|(q)_n| > C_1 (a+\ve)^n \,.
\ee
Since $Q_2(s) \ge |s| \ge s(e)$ for every $e\in \cE(\bD)$, we have
\be
\lbl{e507}
\sum_{e \in \cE(\bD)} s(e) \le 2c_D \, Q_2(s)\,.
\ee

We have
\begin{align*}
\left| \frac{q^{Q_2(s)}}{\prod_{e\in \cE(\bD)} (q)_{s(e)}}\right| &
< \frac{a^{Q_2(s)}}{\prod_{e\in \cE(\bD)} C_1 \, (a+\ve)^{s(e)} }
\quad\text{by \eqref{e506}} \\
& <  (C_1)^{-2c_D} \frac{a^{Q_2(s)}}{(a+\ve)^{2c_D Q_2(s)} } \quad
\text{by \eqref{e507}}\\
&=  (C_1)^{-2c_D} \left( \frac{a} {(a+\ve)^{2c_D}}\right)^{Q_2(s)} \,.
\end{align*}
Thus,
\begin{align}
\sum_{s\in S_{\bD,\BN}} \left| \frac{q^{Q_2(s)}}{\prod_{e\in \cE(\bD)} (q)_{s(e)}}\right|
& < (C_1)^{-2c_D} \sum_{s\in S_{\bD,\BN}}  \left( \frac{q} {(a+\ve)^{2c_D}}\right)^{Q_2(s)}
\notag \\
& = (C_1)^{-2c_D} \sum_{m=0}^\infty  g(m) \left( \frac{a}
{(a+\ve)^{2c_D}}\right)^{m}, \lbl{e509}
\end{align}
where $g(m)$ is the number of  $s\in S_{\bD,\BN}$ such that $Q_2(s)=m$.
Because $Q_2(s)$ is quadratic and co-positive  in $S_{\bD,\BN}$,  $g(m)$ is
bounded above by a quadratic function of $m$ for large enough $m$.
From Equation \eqref{e508} it follows that the right hand side of
\eqref{e509} is absolutely convergent. This completes the proof of
Corollary \ref{cor.thm2.1}.


\section{Linearly bounded states}
\lbl{sec.kbounded}

In this section we will introduce a partition  of
the set of linearly bounded centered states, which will be key to the
$k$-stability
of the colored Jones polynomial. Throughout this section we fix a reduced,
alternating, $A$-infinite downward
alternating link diagram $D$ with $c_D$ crossings.  Let $\calS:= S_{\bD,\BN}$.
Recall that for a polygon $p\in \cP(\bD)$, $s_p$ is the elementary centered
state with support the boundary of $p$.

If $Q: S_{\bD,\BR} \to \BR$ were positive definite, it would be easy to
prove the stability of $\hat J_{K,n}(q)$. Unfortunately,  $Q$ is not positive
definite, and the summation cone $ S_{\bD,[0,\infty)}$ {\em always} contains
directions where $Q_2(s)=Q(s) +L_1(s)$ grows linearly, and not
quadratically. For instance, if $p$ is a $B$-polygon, then $Q_2(n s_p)=n$
is a linear function of $n$.

\begin{definition}
\lbl{def.kstate.refined}
We say that a centered state $s \in \calS$ is {\em $k$-bounded} for a
natural number $k$ if
$$
Q_2(s) \leq (k+1/3)|s|
$$
\end{definition}

For a subset $\mathcal M \subset \calS$ let  $\mathcal M^{(k)}$ denote the set
of $k$-bounded centered  states in $\mathcal M$.

\subsection{Balanced states at $B$-polygons}
\lbl{sub.Bstates}

Suppose $e$ is an edge of a  $B$-polygon $p$. In this section we always
use the orientation on $e$ coming from the directed graph $\bD$. The
orientation of $e$ is counterclockwise with respect to the
interior of $p$. Incident to the ending vertex of $e$ are four edges of $\bD$,
and let $\te$ be the one opposite to $e$, i.e. in a small neighborhood of
the vertex, $\bD$ looks like a cross, and $e$ and $\te$ are on a line
as in the following figure
$$
\psdraw{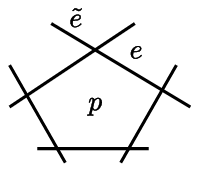}{1in}
$$
Suppose  $s\in \calS$ is a centered state. Recall that $s$ is a function on the
set of edges of $\bD$ and that we already extended $s$ to the vertices
of $\bD$; see Equation \eqref{eq.sv}. Now we further extend $s$ to the set
of $B$-polygons of $\bD$. Suppose $p$ is a $B$-polygon of $\bD$. Let
\be
\lbl{def.spoly}
s(p)=  \sum_{e \in \cE(p)} s(\te).
\ee

\begin{definition}
\lbl{def.sbal}
We will say that a state $s \in \calS$ is {\em balanced} at a $B$-polygon $p$
if $s(v) \le s(p)$ for every vertex $v$ of $p$, and equality holds for at
least one vertex.
\end{definition}

\subsection{Seeds}
\lbl{sub.seeds}

In this section we introduce seeds, their partial ordering, and
relative seeds.

We say that two $B$-polygons are {\em disjoint} if they do not have
a common vertex. Suppose $\Pi $ is a collection of disjoint $B$-polygons.
Let $\nbd(\Pi )$ be the set of all edges of $\bD$ incident to a vertex of
a polygon in $\Pi $. Observe that every edge of a polygon in $\Pi $ is in
$\nbd(\Pi )$.

\begin{definition}
\lbl{def.seed}
\rm{(a)} A seed $\th =(\Pi ,\sigma)$ consists of a collection $\Pi $
of disjoint $B$-polygons and a map
$ \sigma: \nbd(\Pi ) \to \BN$ such that $\sigma$  can be extended to a
centered state $s\in \calS$ which is balanced at every polygon in $\Pi $.
Such $s$ is called an extension of $\sigma$, and the set of all extensions
of $\sigma$ is denoted by $\cS_\th$. For a polygon $p\in \Pi $ let
$\sigma(p)= s(p)$, which does not depend on the
extension $s$.
\newline
\rm{(b)}
The {\em $B$-norm} $\vv \th  \vv_B$ of $\th $ is
the number of $B$-polygons in $\Pi $.
\end{definition}
We allow the empty seed $\th =\emptyset$, in which case
$\calS_\th  = \calS$.
Next we define a partial order on the set of seeds.

\begin{definition}
\lbl{def.orderseed}
Suppose
$\th =(\Pi ,\sigma)$ and $\th '=(\Pi ',\sigma')$ are seeds. Then
$\th  \le \th '$ if $\Pi \subset \Pi '$ and $\sigma$ is the
restriction of $\sigma'$.
\end{definition}

Observe that $\emptyset \le \th $ for any seed $\th $. Moreover, if
$\th  < \th '$, then $\vv \th  \vv_B < \vv \th ' \vv_B $. Since
the number of $B$-polygons is finite, we have the following simple but
important fact.

\begin{lemma}
\label{lem.finite}
Every strictly increasing sequence of seeds is finite.
\end{lemma}


We now introduce relative seeds.

\begin{definition}
\lbl{def.relseeds}
Suppose $\th =(\Pi,\sigma) < \th' = (\Pi' ,\sigma')$.
\newline
\rm{(a)}
Let $|\th' \setminus\th| := \max_{p \in \Pi' \setminus \Pi}\sigma'(p)
= \max_{v \in \cV(\Pi' \setminus \Pi)} \sigma'(v)$.
\newline
\rm{(b)}
Let
$\calS_{\th < \th'}$ be the set of all $ s \in \calS_\th $ of the form
\be
\lbl{eq.sk}
s = s' +  \sum_{p \in (\Pi' \setminus \Pi)}\,  (m - \sigma'(p))\,  s_{p}
\ee
where
\be
\lbl{eq.para}
\quad s' \in \calS_{\th'}  \quad \text{and } |s'| < m.
\ee
\end{definition}

\subsection{A partition of the set of $k$-bounded states}
\lbl{sub.thinthick}

In this section we give a partition of the set of
$k$-bounded states $\calS^{(k)}=\calS_\emptyset^{(k)}$ and more generally, the
set $\calS_\th ^{(k)}$ of $k$-bounded states with seed $\th $. The next
proposition will be proven in Section \ref{sec.sk}.

\begin{proposition}
\lbl{prop.sk1}
For every non-negative integer $k$ and every seed $\th$ there exists a
constant $C>0$ 
such that if $|s| > C k^2$ and $s \in \calS_\th ^{(k)}$,
then $s \in \calS_{\th <\th' }$ for a unique seed
$\th'  > \th $ with $|\th'  \setminus\th <k$.
In other words, up to elements $s$ with $|s| \le  C k^2$, we have
the following finite partition of the set $\calS_\th ^{(k)}$ of $k$-bounded
states:
$$
\calS_\th ^{(k)} = \bigsqcup_{\th'  > \th , \, |\th'  \setminus\th| <k }
\calS_{\th <\th' }^{(k)} \,.
$$
\end{proposition}

When $\th $ is maximal we obtain the following.

\begin{corollary}
\lbl{cor.sk1.2}
For every non-negative integer $k$ and every maximal seed $\th $,
$\calS_\th ^{(k)}$ is a finite set.
\end{corollary}

The next proposition will also be proven in Section \ref{sec.sk}.

\begin{proposition}
\lbl{p.533}
 (a) Suppose $s \in S_{\th < \th'}$ with presentation \eqref{eq.sk} as in
Definition \ref{def.relseeds}. One has
 $|s|=m$ and

\begin{align}
\lbl{eq.Q2}
Q_2(s) - Q_2(s') & =
\sum_{p \in (\Pi \setminus \Pi ')}(m- \sigma(p))(\sigma(p)+1), \\
L_0(s) & \equiv L_0(s') \pmod 2.\lbl{eq.L0}
\end{align}
(b)
Fixing $\th < \th'$,
the presentation of $s\in \cS_{\th<\th'}$ given by Equation
\eqref{eq.sk}, with $(m,s')$ satisfying \eqref{eq.para} is unique.
In other words, the map $(m,s') \mapsto s$ given by \eqref{eq.sk}
is a bijection between the set of pairs $(m,s')$ satisfying
\eqref{eq.para} and $\cS_{{\th<\th'}}$.
\end{proposition}

\subsection{The weight of $k$-bounded states}
\lbl{sub.Fxqs}

In this section we express $F(x,q,s)$ in terms of $F(x,q,s')$ for
centered states $s,s'$ related by Equation \eqref{eq.sk}.

\begin{definition}
\lbl{def.Gxyq}
We say that a series
$G(x,y,q)= \sum_{i,j=0}^\infty  G_{i,j}(q) x^i y^j \in \BZ((q))[[x,y]]$ is
{\em weakly $x$-adequate} of order less $\le t$ if $G(xq^t,y,q)\in
q^{-C}\BZ[[q]][[x,y]]$ for some constant $C$ depending on $G$, i.e.
$$
\mindeg_q (G_{i,j}(q)) > -ti-C
$$
for every $i,j \ge 0$.
\end{definition}

The next lemma is elementary.

\begin{lemma}
\lbl{lem.weaklyx}
\rm{(a)} If $G(x,y,q)$ is weakly $x$-adequate of
order $\le t$, then $G(q^{k}, q^l, q) \in \BZ((q))$ for every $k\ge t+1$,
$l \ge 0$.
\newline
\rm{(b)}  If $G(x,y,q) \in q^{-C}\BZ[[q]][[x,y]]$ is weakly $x$-adequate of
order $\le t$, then for every $l \in \BN$, $q^{-C}G(x,q^l,q)\in\BZ[[q]][[x]]$
is $x$-adequate of order $t$.
\newline
\rm{(c)} The set of  weakly
$x$-adequate  series of order less $\le t$ is closed under addition
and multiplication, i.e. it is a $\BZ$-subalgebra of $\BZ((q))[[x,y]]$.
\newline
\rm{(d)}
If $G(x,y,q)$ is weakly $x$-adequate of order $\le t$,
then it is  weakly $x$-adequate of order $\le t'$ for every $t' \geq t$.
\end{lemma}

The next lemma uses the notation of
Definition \ref{def.relseeds}.

\begin{lemma}
\lbl{lem.FSk}
Given $s \in \calS_{\th < \th' }$ with presentation \eqref{eq.sk} and
let $\ell= \vv \th'  \vv_B - \vv \th  \vv_B$. Then there
exists a  weakly $x$-adequate series $G_{\th < \th' } (x,y,q)
\in  y^\ell \BZ((q))[[x,y]]$ of order $\le |\th' \setminus \th |$
such that for $n \ge |s|$ we have
$$
F(q^{n+1},q,s)=  G_{\th < \th' }(q^{n+1},q^{|s|},q)\, F(q^{n+1},q,s').
$$
Moreover, $ G_{\th < \th' }(q^{n+1},q^{|s|},q)\in \BZ[[q]]$.
\end{lemma}

\begin{proof}
Let 
$x=q^{n+1}$ and $y=q^{|s|}$.
By Proposition \ref{p.533},  $m= |s|$. For convenience we write
$\th'\setminus \th$ for $\Pi'\setminus\Pi$, and $\sigma$
for $\sigma(\th')$.
We have the following relations, followed directly from the definition.

 \begin{align}
 \sigma(v) & =  \sigma(e) + \sigma(\te)  \lbl{ee1}\\
   \sigma(p) & = \sum _{e \in \cE(p)} \sigma(\te) \lbl{ee2}\\
 s(e) & = m-\sigma(p) + {\sigma}(e) \lbl{ee3}\\
s(v) & = m-\sigma(p) + {\sigma}(v). \lbl{ee4}
\end{align}
Here $e$ is an edge and $v$ is a vertex  of a $B$-polygon $p$ in
$\th'\setminus \th$,  and $\te$ is defined as in Section
\ref{sub.Bstates}. Besides, in \eqref{ee1}, $v$ is the ending vertex of
the edge $e$. Besides, each of $\sigma(p), \sigma(v)=\sigma(e)
+ \sigma(\te)$ is bounded from above by $|\th'\setminus \th|$, by definition.

From the definition \eqref{eq.Fxqs2} and Proposition \ref{p.533},  we have

\begin{eqnarray*}\frac{ F(q^{n+1},q,s)}{F(q^{n+1},q,s')}
&=& q^{\sum_{p \in (\th'\setminus \th)}(m- \sigma(p) )(\sigma(p)+1)}
\prod_{e\in \cE(\th'\setminus \th )}\frac{(q)_{{\sigma}(e)} \,
(x q^{-{\sigma}(e)})_\infty }{(q)_{s(e)}\, (x q^{-s(e)})_\infty  }
\prod_{v \in \cV (\th'\setminus \th )} \frac{(x q^{-s(v)})_\infty}{(x q^{-{\sigma}(v)})_\infty}
\\
&=&
\prod_{p\in (\th'\setminus \th)} \left\{\frac
y {q^{\sigma(p)}}\, \prod_{e\in \cE(p)} \left[ \frac{(q)_{{\sigma}(e)}}{
(x q^{-{\sigma}(e)-{\sigma}(\te )})_{{\sigma}(\te )}} \right]\,
\left[ \frac{({y} q^{{\sigma}(e)+1-\sigma(p)})_\infty }{(q)_\infty} \right]\,
\right. \\
& & \left. \qquad\quad
\left[ \left( \frac y{q^{\sigma(p)}}  \right)^{{\sigma}(\te )}
\left( \frac {q^{\sigma(p)}}y x \, q^{- {\sigma}(e)-{\sigma}(\te )} \right)_{{\sigma}(\te )}
\right] \right\}
\end{eqnarray*}
where the second identity follows from  a simplification of $q$-factorial
using relations \eqref{ee1}--\eqref{ee4}. Let us look at the factors in
square brackets.

Since ${\sigma}(e)+{\sigma}(\te ) \le |\th'  \setminus \th |$, the
first square bracket factor   is $x$-adequate with order
$\le  |\th'  \setminus \th |$.

It is clear that for $a \ge 0$,  $(yq^{-a})_\infty \in q^{-a(a+1)/2}
\BZ[[q]][[y]]$. The second square bracket factor  is in $q^{-C}\,
\BZ[[q]][[y]]$, where $C =  |\th'  \setminus \th |( |\th'
\setminus \th |+1)/2$.

The third square bracket factor is a polynomial in $x,y$ with coefficients
in $\BZ[q^{\pm 1}]$, and it is $x$-adequate with order $\le  |\th'
\setminus \th |$.
\end{proof}

\subsection{Stability away from the region of linear growth}
\lbl{sec.kunboundedseed}

In this section we show the stability for the $k$-unbounded centered states.

\begin{proposition}
\lbl{prop.unbounded}
Fix $k,l \in \BN$.
Suppose  $\th  < \th' $ are seeds and
$G(x,y,s)\in \BZ((q))[[x,y]]$ is weakly $x$-adequate of order $\le l
+  |\th'  \setminus \th |$ .  Then
\begin{eqnarray*}
B_n(q) &:=& \sum_{s:\, |s| \leq n -l \,, s \in \calS_{\th < \th' } \setminus \calS_{\th <\th' }^{(k)}}
F(q^{n+1},q,s)\, G(q^{n+1},q^{|s|},q)
\end{eqnarray*}
is $k$-stable.
\end{proposition}

\begin{proof}
Recall $\ti F(x,q,s)$ from Equation \eqref{Fxqs}.
Expand
$$
\ti F(x,q,s) G(x,q^{|s|},q)=\sum_{m=0}^\infty a_m(q,s)x^m
$$
into a power series in $x$ and define
$$
\Phi_j(q)= \sum_{s:\,s \not\in S^{(k)}}  q^{Q_2(s)} \, a_j(q,s)
$$
for $j \leq k$. The weak $x$-adequate condition on $G$ and adequate
condition on $F$ (from Lemma \ref{lem.amb})
imply that for all but finitely many $s$ and for $j \leq k$ we have
$$
Q_2(s)-\mindeg_q(a_j(q,s)) > (k+1/3)|s|-\mindeg_q(a_j(q,s)) \geq
|s|/3-C \,,
$$
where $C \in \BZ$ is such that $G(x,y,s) \in q^{-C} \BZ[[q]][[x,y]]$.
It follows that $\Phi_j(q) \in \BZ((q))$ is convergent.
Let $f_n(q)=\sum_{s:\, |s| \leq n \,, s \not\in S^{(k)}} F(q^{n+1},q,s)$.
We now follow the proof of part (d) of Lemma \ref{lem.x}. We have:
\begin{equation*}
\left(f_n(q)-\sum_{j=0}^k \Phi_j(q) q^{j(n+1)}\right) q^{-k(n+1)} =
\S_{1,n}-\S_{2,n}
\end{equation*}
where
\begin{eqnarray*}
\S_{1,n}&=&  \sum_{s:\, |s| \leq n \,, s \not\in S^{(k)}}
\sum_{j=k+1}^\infty  q^{Q_2(s)} \, a_j(q,s)
q^{(j-k)(n+1)} \\
\S_{2,n}&=&  \sum_{s:\, |s| > n \,, s \not\in S^{(k)}}
\sum_{j=0}^k q^{Q_2(s)} \, a_j(q,s)
q^{(j-k)(n+1)}
\end{eqnarray*}
For $\S_{1,n}$ we use the $x$-adequacy of order $\leq |s|$ to  obtain
$$
Q_2(s)+\mindeg_q(a_j(q,s))+(j-k)(n+1)
\geq Q_2(s)-j|s|+(j-k)(n+1)
$$
Since the coefficient of $j$ in the above expression is $n+1-|s|>0$, it
follows that its minimum as a function of $j$ is attained at $j=k+1$, i.e.,
$$
Q_2(s)-j|s|+(j-k)(n+1) \geq Q_2(s)-(k+1)|s|+n+1
$$
Since $s \not\in S^{(k)}$ and $|s| \leq n$ it follows
$$
Q_2(s)-(k+1)|s|+n+1 \geq (k+1/3)|s|-(k+1)|s|+n+1=-2|s|/3+n+1>n/3\,.
$$
For $\S_{2,n}$ since $|s| >n$ we the fact that $s$ is not $k$-bounded
to obtain
$$
Q_2(s)-j|s|+(j-k)(n+1)  \geq  Q_2(s) -k|s| \geq |s|/3> n/3 \,.
$$
Thus,
$$
\left(f_n(q)-\sum_{j=0}^k \Phi_j(q) q^{j(n+1)}\right) q^{-k(n+1)} \in
q^{n/3} \BZ[[q]] \,.
$$
This completes the proof of the proposition.
\end{proof}



\section{Stability in the region of linear growth}
\lbl{sec.bounded}

\begin{theorem}
\lbl{thm.1stable.main}
Suppose $\th $ is a seed and  $G(x,y,q) \in \BZ((q))[[x,y]]$ is weakly
$x$-adequate of order $\le |\th |+l$, where $l \in \BN$.
Then the sequence
$$
H_n(q)=\sum_{s:\, |s| \leq n-l \,, s \in \calS_\th } F(q^{n+1},q,s)\, G(q^{n+1},q^{|s|},q)
$$
is stable. 
\end{theorem}

\begin{remark}
\lbl{rem.thm2follows}
In particular, the above theorem holds when $\th =\emptyset$, $l=0$ and
$G=1$. In that case, Proposition \ref{c.degree} implies that
$H_n(q)=\hat J_{K,n}(q)$ and we conclude the stability of the colored Jones
polynomial of an alternating link $K$.
\end{remark}

\begin{proof}
Fix a natural number $k$. We will prove that $H_n(q)$ is $k$-stable.
Subtracting the $k$-unbounded part from $H_n(q)$ and using Proposition
\ref{prop.unbounded}, it is enough to show that
$$
H'_n(q)=\sum_{s:\, |s| \leq n-l \,, s \in \calS_\th ^{(k)}} \cF_n(q,s)
$$
is $k$-stable.
We proceed by downwards induction, starting from the case when $\th $ is
maximal.
This case follows  from Corollary \ref{cor.sk1.2}, which states
that $\calS_{\th }^{(k)}$ is a finite set, and Lemma \ref{lem.301}.

Assume that the statement holds for all $\th' $ strictly greater than $\th $.
We will show that the statement holds for $\th $. Then Lemma
\ref{lem.finite} implies that the statement holds for any seed $\th $.

Using the partition of $\calS_\th ^{(k)}$ described in Proposition
\ref{prop.sk1}, and $n$ sufficiently large, we obtain that

\be
\lbl{eq.15}
\sum_{s:\, |s| \leq n-l  \,, s \in \calS_\th ^{(k)}} \cF_n(q,s)
= \sum_{\th'  > \th , \,  |\th' \setminus \th |  \le k}
\,\left(  \sum_{s:\, |s| \leq n-l \,, s \in \calS_{\th < \th' }^{(k)}}
\cF_n(q,s) \right)   + \mathrm{Err},
\ee
where $\mathrm{Err}$ is a finite alternating sum of terms of the form
$\cF_n(q,s)$ for some $s \in \calS_{\th }$. By Lemma \ref{lem.301},
$\mathrm{Err}$ is stable.
Because the outer sum on the right hand side of \eqref{eq.15}  is finite,
it is enough to prove
$k$-stability for each inner sum
$$
H''_n(q):=\sum_{s:\, |s| \leq n-l \,, s \in \calS_{\th < \th' }^{(k)}} \cF_n(q,s)\,.
$$
Adding back the $k$-unbounded part (using Proposition
\ref{prop.unbounded}), it is enough to show that
$$
H'''_n(q):=\sum_{s:\, |s| \leq n-l \,, s \in \calS_{\th < \th' }} \cF_n(q,s)
$$
is $k$-stable. Using the decomposition of Lemma \ref{lem.FSk}, we have
\begin{eqnarray}
\cF_n(q,s) &=& G (q^{n+1}, q^m, q) \,  G_{\th < \th' }(q^{n+1}, q^m, q) \,
F(q^{n+1},q,s')  \notag\\
&=& G'(q^{n+1}, q^m, q)\, \, F(q^{n+1},q,s')\,, \lbl{eq.FF}
\end{eqnarray}
where $G'(x,y,q)= G(x,y,q)  \,  G_{\th < \th' }(x,y,q)$, and $s' \in \calS_{\th'} $.
$G(x,y,q)$ is weakly $x$-adequate of order $\leq |\th|+l$ and
$|\th |+l \le |\th'+l |$. Moreover, $ G_{\th < \th' }(x,y,q)$ is weakly
$x$-adequate of order $\leq |\th' \setminus \th |$ and
$|\th' \setminus \th | \le |\th' | \leq |\th' |+l$. Lemma \ref{lem.weaklyx}
implies that $G'(x,y,q)$ weakly $x$-adequate of order $\le |\th' |+l$.

By part (b) of Proposition \ref{p.533},
$\calS_{\th  < \th' }$ is parametrized by pairs
$(m,s')$ with $s'\in \cS_{\th'} $ with $ |s'| < m$. We have

\begin{align*}
H'''_n(q) &= \sum_{s:\, |s| \leq n-l \,, s \in \calS_{\th < \th' }} \cF_n(q,s)
\\
& =
\sum_{m=1}^{n-l} G'(q^{n+1}, q^m, q) \sum_{s':\, |s'| <m \,, s' \in \calS_{\th'} }
F(q^{n+1},q,s')
\\
&= \sum_{s:\, |s| \le n-l-1 \,, s \in \calS_{\th' }} F(q^{n+1},q,s)
\sum_{m =|s|+1}^{n-l} G' (q^{n+1}, q^m, q)\\
&= \sum_{s:\, |s| \le n-l-1 \,, \, s \in \calS_{\th'} } F(q^{n+1},q,s)\,
G'' (q^{n+1}, q^{|s|},q)\,,
\lbl{eq.13}
\end{align*}
where the second identity follows from \eqref{eq.FF} and the above
mentioned parametrization of  $\calS_{\th  < \th' }$, the third identity follows
by changing notation $s'$ to $s$ and
exchanging the two summations, and the fourth identity follows from
Lemma \ref{lem.dumb} below, with $G'' (x,y,q)$ a weakly $x$-adequate
series of order $\le |\th' |+l$.
By induction hypothesis, the last sum of the above identity is $k$-stable.
This completes the proof of Theorem \ref{thm.1stable.main}.
\end{proof}

\begin{lemma}
\lbl{lem.301}
For a fixed $s \in \calS_\th $, and $G(x,y,q)$ weakly $x$-adequate of order
$\leq t$, the sequence  $\cF_n(q,s):=F(q^{n+1},q,s)\,
G(q^{n+1},q^{|s|},q)$ is stable.
\end{lemma}
\begin{proof}
Lemma \ref{lem.weaklyx} implies that $q^{-C}G(x,q^{|s|},q)$ is $x$-adequate
and part (a) of Lemma \ref{lem.x} implies that $q^{-C}F(x,q,s)G(x,q^{|s|},q)$
is $x$-adequate, too. The result follows from part (d) of Lemma
\ref{lem.x}.
\end{proof}

The next lemma is reminiscent to the notion of a $q$-Laplace transform.

\begin{lemma}
\lbl{lem.dumb}
Suppose $l, t \in \BN$, and
$G(x,y,q) \in q^{-C}\BZ[[q]] [[x,y]]$ is weakly $x$-adequate of order $\le l+ t$.
Then there exists a weakly $x$-adequate  series
$H(x,y,q) \in q^{-C} \BZ[[q]] [[x,y]]$ of order $\le l+ t$,
such that for every $a,n\in \BN$ with $n \ge l+t+1$ and $n \ge l+a+1$,
\be
\lbl{eq.14}
\sum_{m= a+1}^{n-l} G(q^{n+1}, q^m,q) = H(q^{n+1}, q^a,q) \,.
\ee
\end{lemma}

\begin{proof}
Let $G(x,y,q) = \sum G_{i,j}(q)  x^i y^j \in \BZ((q)) [[x,y]]$. We have:

\be
\sum_{m=a+1}^{n-l} q^{mj} = \frac{q^{j(a+1)} -q^{(n+1-l)j} }{1-q^j}
= \frac{y^jq^j -x^j q^{-lj} }{1-q^j}\Big |_{x= q^{n+1}, y= q^a}
\ee
Hence if we define
$$
H(x,y,q)= \sum_{i,j} G_{i,j}(q) x^i \frac{y^jq^j -x^j q^{-lj} }{1-q^j} \,,
$$
then \eqref{eq.14} holds. It is easy to see that $H$ is weakly
$x$-adequate of order $\le l+t$.
\end{proof}


\section{Partition of the set of  $k$-bounded states}
\lbl{sec.sk}

In this section, we will prove Propositions \ref{prop.sk1} and
Proposition \ref{p.533}. We will fix an $A$-infinite alternating, 
diagram $D$ with $c_D$ crossings. We assume that $D$ represent a
non-trivial link, hence $c_D \ge 2$.

\subsection{Some lemmas about $k$-centered states}
\lbl{sub.some}

Suppose $p$ is a $B$-polygon of $\bD$. Recall that the orientation of
every edge of $p$ is counterclockwise. Incident to the ending vertex of
en edge $e \in \cE(p)$
there are two edges of $\bD$ no belonging to $p$; one of them is $\te$
defined in Section \ref{sub.Bstates}, and  let $\ce$ be the other edge
as in Figure \ref{figureA}.

\begin{figure}[htpb]
$$
\psdraw{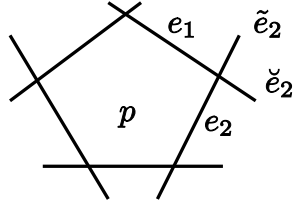}{1.5in}
$$
\caption{A vertex of a $B$-polygon $p$ and its neighboring edges.}
\lbl{figureA}
\end{figure}

Adding up Equations \eqref{c2} for all vertices of $p$,
and using the definition of $s(p)$ from Equation \eqref{def.spoly}
it follows that
\be
\lbl{eq.TP}
s(p)=\sum_{e\in \cE(p)} s(\ti e) =\sum_{e\in \cE(p)} s(\breve{e}) \,.
\ee

\begin{lemma}
\lbl{lem.494}
Suppose $p\in \cP(\bD)$ is a $B$-polygon, $s'$ is a centered state,
$l \in \BN$, and
$$
s= ls_p + s'\,.
$$
Then
$$
Q_2(s) = Q_2(s') + l(s'(p)+1)\,.
$$
\end{lemma}

\begin{proof}
Recall that for a centered state $s$,
\be
Q(s)= \frac12 \sum_\al ab, \lbl{e342}
\ee
 where the sum is over all angles $\al$ of type $A$, and $a$ and $b$
are the $s$-values of the two edges forming the angle $\al$.

Note that $s(e)=s'(e)$ except when $e$ is an edge of $p$. Hence
\be
\lbl{e356}
Q(s) - Q(s') =  \frac12 \sum_ {\al} \big( s(e)s(f) - s'(e) s'(f)\big)\,,
\ee
where the sum is over all $A$-angles $\al$ whose vertex is a vertex
of $p$. Each vertex $v$ has two $A$-angles, and each such $A$-angle
has one edge in $p$, denoted by $e$ in \eqref{e356},
and one edge not belonging to $p$, denoted by $f$ in \eqref{e356}.
Then $s(f)=s'(f)$ and $s(e)-s'(s)=l$, hence from \eqref{e356}
$$
Q(s) - Q(s') = l \sum_{\al } s'(f)/2= l s'(p)\,.
$$
Since $L_1$ is linear we have have $L_1(s)- L_1(s') = L_1(l s_p) =l$,
where the last identity comes from \eqref{e151}. Hence
\be
\lbl{eq.QQ2}
Q_2(s)- Q_2(s')  = Q(s)- Q_2(s')  + L_1(s) - L_1(s')  = l(s'(p)+1)
\,.
\ee
\end{proof}

\begin{lemma}
\lbl{lem.176}
Suppose $p$ is a $B$-polygon, $s$ a centered state, and
$m= \max_{v\in \cV(p)}s(v)$. Then for $s(e) \ge m - s(p)$ for every $e\in \cE(p)$.
\end{lemma}

\begin{proof}
Suppose $m=s(v)$, where $v$ is the ending vertex of the edge $e_1$.
Assume that $e_1,e_2,\dots e_t$ are all edges of $p$, counting clockwise,
as in Figure \ref{figureA}.
By identity \eqref{c2} at the ending vertex of $e_j$, we have
$s(e_j) - s(e_{j-1})= s(\ce_j)-s(\te_j)$. Hence
\be
\lbl{e367}
s(e_j) - s(e_{j-1}) \ge - s(\te_j) \,.
\ee
Summing the above inequalities with $j$ from $2$ to $n$, together with the
identity $s(e_1)= m - s(te_1)$, we have
$$
s(e_n) \ge m - \sum_{j=1}^n s(\te_j) \ge m - s(p) \,.
$$
\end{proof}

\subsection{Assumption on lower bound of $|s|$}

\begin{definition}
For a centered state $s$ and a positive integer $k$,
a polygon $p\in\cP(\bD)$  is $(k,s)$-big if $s$ achieves the maximal
value $|s|$ at one of the the vertices of $p$ and $s(p) <k$.
\end{definition}

For a vertex $v$ of $\bD$ there are 4 polygons in $\cP(\bD)$ incident to
$v$, i.e. having $v$ as a vertex, and two of them are $B$-polygons which
are opposite through $v$.

\begin{lemma}
\lbl{lem.000}
Suppose $s$ is a $k$-bounded centered state satisfying
\begin{equation}
\lbl{eq.assumeA}
|s|>12k(2k+1) c_D.
\end{equation}
\rm{(a)} Any $(k,s)$-big polygon is a $B$-polygon and any two $(k,s)$-big
polygons are disjoint.
\newline
\rm{(b)}  Suppose $s$ achieves maximal at a vertex $v$, i.e. $s(v)=|s|$. Then
exactly one of the two $B$-polygons
incident to $v$ is $(k,s)$-big.
\end{lemma}

\begin{proof}
(a)
If two edges $e,f \in \cE(\bD)$ form an $A$-angle, then from \eqref{e342}
we have $Q(s) \ge s(e) s(f) /2$. Hence if $s$ is $k$-bounded we have
\be
\lbl{e382}
(k+1/3)|s| \ge \frac {s(e) s(f)} 2 \,.
\ee
If $p$ is an $A$-polygon, then any two consecutive edges of $p$ form an
$A$-angle.
Suppose $p$ is $(k,s)$-big. Then $s(p)<k$, and by Lemma \ref{lem.176},
$s(e) > |s|-k$ for every edge $e$ of $p$.  Also, from \eqref{eq.assumeA}
it is clear that $k < |s|/2$. It follows from \ref{e382}  that
\be
 (k+1/3) |s| \ge (|s|-k)^2/2 > |s|^2 /8.
 \lbl{e384}
 \ee
Hence $|s|< 8(k+1/3)$, which contradicts \eqref{eq.assumeA}.

Now suppose  two $(k,s)$-big polygons share a common vertex $v$. Then for
any $A$-angle at $v$ the $s$-value of any edge is $\ge |s|-k$. We again
lead to  \eqref{e384}, which is a contradiction.

(b) To prove part (b) we first prove a few claims.

{\bf Claim 1.}
Suppose $p$ is a $B$-polygon of $\bD$. Assume that
$s(e) > |s| - 4k c_D$ for an edge $e$ of $p$.
Then $s(e') \ge s(e) -2k$ for any edge $e'$ of $p$ incident to $e$.

\begin{proof}[Proof of Claim 1]
Assume the contrary that $s(e') < s(e) -2k$.
Suppose $v$ is the common vertex of $e,e$ and $f,f'$ are the two remaining
edges incident to $v$ such that $f$ is opposite to $e$ as in
the following Figure \ref{f.Pv}.

\begin{figure}[htpb]
$$
\psdraw{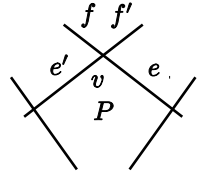}{1in}
$$
\caption{A vertex $v$ of $B$-polygon and its neighboring edges.}
\lbl{f.Pv}
\end{figure}

Since $ s(f')-s(f)  = s(e) - s(e')>2k$, we
have $s(f') \ge 2k+1$. Since the angle between $e$ and $f'$ is of type $A$,
$$
Q_2(s) \ge s(e)s(f')/2 \ge (2k+1) (|s|- 4k c_D)/2 > |s|(k+1/3),
$$
where the last inequality follows from \eqref{eq.assumeA}.
The above inequality contradicts the assumption that $s$ is $k$-bounded.
\end{proof}

 Let $v$ be a vertex of $\bD$
where $s(v)=|s|$. One of the four edges incident to $v$, say $e$,  has
$s$-value $\ge |s|/2$.
Let  $p$ be the unique $B$-polygon of $\bD$ having
$e$ as an edge on the boundary. We will prove that $p$ is $(k,s)$-big.

{\bf Claim 2.}
The $s$-value of every edge  of $p$ is $ \geq |s|- 4k c_D$.

\begin{proof}[Proof of Claim 2] Besides $e$, suppose $e'$ is the other
edge incident to $v$ which is also an edge of $p$, and $f,f'$ are the
other two edges which are not edges of $p$,
as in Figure \eqref{f.Pv}. Note that the number of edges of $p$ is less
than $2c_D$, the total number of edges of $\bD$.

By  \eqref{e382} and $s(e) \ge |s|/2$,  we have

$$
(k+ 1/3)|s| \ge s(e) s(f') /2 \ge |s| s(f')/4 \,.
$$
It follows that $s(f') \le 4(k+1/3) \le 4k+2$, and hence by Equations
\eqref{c2} we have $s(e') \ge |s| - 4k-2$.

If $g$ is an edge of $p$, then there is a path from $e'$ to $g$
consisting of at most $c_D-1$ edges. It follows from Claim 1 that
$$s(g) \ge s(e') - 2k(c_D-1) \ge  |s| - 4k-2 - 2k(c_D-1)
= |s| -(2k+2 + 2k c_D) \ge |s| -4k c_D.$$
 In the last inequality we used that fact that $k\ge 1$ and $c_D \ge 2$.
\end{proof}

Now we can  finish the proof of part (b) of Lemma \ref{lem.000}.
It remains to prove that $s(p) <k$.
 Assume the contrary  $s(p) \geq k$.
By Claim 1, $s':= s - (k- 4k c_D) s_p$ takes non-negative value at every
edge of $\bD$, hence $s'$ is a centered state, and $Q_s(s') \ge 0$.
Note that $s(p)= s'(p)$, since $s$ and $s'$ agree on any edge not
belonging to $p$. By Lemma \ref{lem.494},
$$
Q_2(s)= Q_2(s') + (|s|- 4 k c_D) (s(p)+1)  \ge (|s|- 4 k c_D)(k+1)
> (k+1/2)|s|\,,
$$
which contradicts the  $k$-boundedness of $s$. This completes the
proof of Lemma \ref{lem.000}.
\end{proof}

\subsection{Proof of Proposition \ref{p.533}}
\lbl{sub.p.533}

Part (a). Suppose $s \in \calS^{(k)}_{\th < \th'}$ has the presentation \eqref{eq.sk}
\be
\lbl{eq.sk10}
 s = s' + \sum_{p \in (\Pi' \setminus \Pi)}(m-s(p)) s_p
\ee
with $s' \in S_{\th'}$ and $|s'| < m$. The $s'(e)=s(e)$ for every edge
$e$ outside $\Pi' \setminus \Pi$. Hence if $v$ is is not a vertex of
any $p\in (\Pi' \setminus \Pi)$, then
$s(v) = s'(v) < m$.

On the other hand if $v$ is a vertex of $p\in (\Pi' \setminus \Pi)$, then

$$
s(v) = s'(v) + (m- s'(p)) \le m \,,
$$
where the inequality follows from the fact that $s'$ is balanced at $p$.
But there is a vertex of $p$ such that $s'(v)= s'(p)$, and for which
$s(v)=m$. It follows that the maximum of $s(v)$ is $m$, or
$|s|=m$.

Identity \eqref{eq.Q2} follows right away from Lemma \ref{lem.494}.
Identity \eqref{eq.L0} follows that the fact that $L_0$ is a linear map,
$L_0(s_p)=\equiv 2 L_1(s_p) \equiv 2\pmod 2$,  by Lemmas \ref{p.210} and
{lem.L1}(a).

Part(b). We have to show that $s'$ and $m$ are uniquely determined by
$s$. In fact, by part (a), $m = |s|$. Then \eqref{eq.sk10} shows that
$s'$ is determined by $s$ and $m$.
 This completes the proof of Proposition \ref{p.533}.
\qed

\subsection{Proof of Proposition \ref{prop.sk1}}
\lbl{sub.prop.sk1}
Suppose
$\th=(\Pi,\sigma)$ is a seed 
and consider a $k$-bounded centered state
$s \in \calS^{(k)}_{\th}$. Recall that $|\th|= \max_{v \in \cV(\Pi)} \sigma(v)$.
Assume that
\be
\lbl{e575}
|s|> \max(12k(2k+1) c_D, |\th|+k).
\ee
Will show that if $s \in \calS^{(k)}_{\th}$ satisfying the lower bound
\eqref{e575}, then there is a unique $\th' > \th$ with
$|\th'\setminus \th|<k$ such that
$s \in \calS^{(k)}_{\th < \th'}$. This will prove Proposition \ref{prop.sk1}.

{\bf Uniqueness.} Assume that  $s \in \calS^{(k)}_{\th < \th'}$ with
$|\th'\setminus \th|<k$. Then $s$ has presentation \eqref{eq.sk10}.
By Proposition \ref{p.533}(a),   $m= |s|$ is uniquely determined by $s$.
 In the proof of  Proposition \ref{p.533}(a)  in Section \ref{sub.p.533}
we showed that if  $p\in (\Pi' \setminus \Pi)$ then there is a vertex
$v$ of $p$ such that $s(v)=|s|$.
We also have  that $s(p) \le |\th'\setminus \th|<k$. Thus every
$p \in \Pi'\setminus \Pi$ is $(k,s)$-big.

Conversely, suppose $p$ is a $(k,s)$-big polygon. Then there is a
vertex $v$ of $p$ such that $s(v)=|s|$. The proof of  Proposition
\ref{p.533}(a) showed that $v$ is a vertex of
 a polygon $p'\in \Pi'\setminus \Pi$. Both $p$ and $p'$ are incident
to $v$ and both are $(k,s)$-big. By Proposition \ref{lem.000}(a), $p=p'$.

 Thus $\Pi'\setminus \Pi$ is the set of all  $(k,s)$-big polygons. This
determines $\Pi'$ uniquely. Then \eqref{eq.sk10} shows that $s'$ is
uniquely determined by $s$, and
 hence $\sigma'$, which is the restriction of $s'$ on $\nbd(\Pi')$ is
uniquely determined by $s$. This completes the proof of uniqueness.

{\bf Existence.} The proof of the uniqueness already shows us how to
construct a presentation \eqref{eq.sk10} for $s \in \calS^{(k)}_{\th}$.

Let $\Psi$ be the set of all $(k,s)$-big polygons. If $p$ is $(k,s)$-big,
 then by Lemma \ref{lem.176} and \eqref{e575}, $s(v) > |s| - k>|\th|$
for every vertex $v \in \cV(p)$.
This implies if   
$p$ is disjoint from any polygon in  $\Pi$. In particular,
$\Pi \cap \Psi = \emptyset$. Let $\Pi'= \Pi \cup \Psi$.

By Lemma \ref{lem.176}, for any edge $e$ of a $B$-polygon $p \in \Psi$,
$s(e) \ge |s| - \s(p)$. Then
\be s' := s - \sum_{p \in \Psi}(|\s|-s(p)) s_p
\lbl{e080}
\ee
takes non-negative integer value at every edge of $\bD$, and hence is a
centered state. Note that $s(p)= s'(p)$ for any $p \in \Psi$ since
$s$ and $s'$ agree on any edge outside $\Psi$.  We will show that
\eqref{e080} gives us the presentation \eqref{eq.sk10}.

If $v$ is any vertex of $\bD$ for which $s(v)= |s|$, then Lemma
\ref{lem.000}(b) shows that $v$ is a vertex of some polygon $p\in \Psi$.
Hence $s'(v) = s(v) - (|s|-s(p)) < s(v)$.
This means $|s'| <|s|$.

If $v$ is vertex of $p\in \Psi$, then
$$
s'(v) = s(v) - (|s| - s(p)) = (s(v) - |s|) - s(p) \le s(p)= s'(p)\,.
$$
 On the other hand, if $v$ is a vertex of $p \in \Psi$ for which
$s(v)=|s|$, then the above identity shows that $s'(v) =s'(p)$. This
means $s'$ is balanced at every $p\in \Psi$. Since $s'= \sigma$ in
 $\nbd(\Pi)$, it is balanced at every $p\in \Pi$. Thus $s'$ is balanced
at every $p \in \Pi'= \Pi \cup \Psi$.

 Let $\sigma'$ be the restriction of $s'$ on $\nbd(\Pi')$ and
$\th'=(\Pi', \sigma')$. Then $s' \in S_{\th'}$, and \eqref{e080} gives
us the presentation
 \eqref{eq.sk10}, and we have $s \in \calS^{(k)}_{\th < \th'}$.

 Let us estimate $|\th'\setminus \th|$. By definition \ref{def.relseeds},
 $$ |\th'\setminus \th|= \max_{v \in \cV(\Psi)}s'(v) = \max_{p\in \Psi}s'(p) < k.$$

 Thus we conclude that every $s \in \calS^{(k)}_{\th}$  satisfying
\eqref{e575} is an element of $\calS^{(k)}_{\th < \th'}$ for some
$\th' > \th$ with $|\th'\setminus \th| < k$.
 This concludes the proof of the existence, and whence Proposition
\ref{prop.sk1}.
\qed


\section{Proof of Theorem \ref{thm.qholo}}
\lbl{sec.thm.qholo}

In this section we prove Theorem \ref{thm.qholo}.
It is well-known that pointwise sums and products of $q$-holonomic sequences
are $q$-holonomic (see \cite{PWZ,Z}). Moreover, the colored Jones polynomial
$(J_{K,n}(q))$ of every link  is $q$-holonomic \cite{GL}.
Using \eqref{eq.decomposition} we deduce that
$(\hat J_{K,n}(q))$ is $q$-holonomic for every alternating link $K$.
Using a recursion relation \eqref{eq.def.qholo} for $f_n(q)=\hat J_{K,n}(q)$
and the stability Theorem \ref{thm.2}, and collecting powers of $q$ and
$q^n$, it follows that $\Phi_{K,k}(q)$ is $q$-holonomic.

Using a linear recursion for $\Phi_{K,k}(q)$, it is easy to see that
$\mindeg_q(\Phi_{K,k}(q))$ is bounded below by a
quadratic function of $k$; see for example \cite[Thm.10.3]{GL2}. A stronger
statement is known \cite{Ga3}, namely $\mindeg_q(\Phi_{K,k}(q))$ is a
quadratic quasi-polynomial of $k$. This proves Equation \eqref{eq.degPhik}.

Equation \eqref{eq.quadstable} follows from Equation \eqref{eq.degPhik}
using Lemma \ref{lem.quadrate} below. This concludes the proof of Theorem
\ref{thm.qholo}.
\qed

\begin{lemma}
\lbl{lem.quadrate}
Fix $f_n(q) \in \BZ((q))$ and $\Phi_k(q) \in \BZ((q))$ and let
$$
R_{k,n}(q)=\left(f_n(q)-\sum_{j=0}^k \Phi_k(q) q^{j(n+1)}\right)q^{-k(n+1)}
$$
Assume that $\lim_{n\to\infty} R_{k,n}(q)=0$ for all $k$.
Then the following are equivalent:
\begin{itemize}
\item[(a)]
$\mindeg_q (\Phi_k(q)) \geq -C_1 k^2 - C_2$ for all $k$.
\item[(b)]
$\mindeg_q (R_{k,n}(q)) \geq n+1 -C_1 (k+1)^2 -C_2$ for all $k$ and all $n$
large enough.
\end{itemize}
\end{lemma}

\begin{proof}
Let $v=\mindeg_q$. The assumption on $R_{k,n}(q)$ implies that
\be
\lbl{eq.degR}
\lim_{n\to\infty} \mindeg_q(R_{k,n}(q))=+\infty \,.
\ee
It is easy to see that for all $k$ and $n$ we have
\be
\Phi_k(q)=R_{k,n}(q)-q^{-n-1}R_{k-1,n}(q) \,.
\ee
It follows that
\be
\lbl{eq.PhiR2}
-n-1+v(R_{k-1,n}(q)) \geq \min\{v(R_{k,n}(q)), v(\Phi_k(q))  \}
\ee
and
\be
\lbl{eq.PhiR1}
v(\Phi_k(q)) \geq \min\{v(R_{k,n}(q)), -n-1+v(R_{k-1,n}(q))\}
\ee
Now, (a) implies (b) by Equations \eqref{eq.degR} and \eqref{eq.PhiR1}
and (b) implies (a) by Equations  \eqref{eq.degR} and \eqref{eq.PhiR2}.
\end{proof}


\section{An algorithm for the computation of $\Phi_{K,k}(q)$}
\lbl{sec.Phi1}

\subsection{A parametrization of $1$-bounded states}
\lbl{sub.1bdd}

In this section we will compute explicitly the series $\Phi_{K,1}(q)$
of an alternating knot in terms of a planar projection as in Theorems
\ref{thm.1} and \ref{thm.1a}. We begin with a corollary of
Proposition \ref{prop.sk1} for $k=$ and $\th=\emptyset$.

\begin{corollary}
\lbl{lem.1bdd}
Suppose that $s$ is a $1$-bounded centered state and $|s|>6$.
Then, there exists a
$B$-polygon $p$ and a state $s'$ such that
\be
\lbl{eq.sPs'}
s=|s|s_P + s'
\ee
and
$s(e)=0$ if $e$ is an edge of $\bD$ which contains a vertex of $p$.
Moreover, $(P,s')$ are uniquely determined by $s$.
\end{corollary}

\subsection{The computation of $\Phi_{K,1}(q)$ in terms of a planar
diagram}
\lbl{sub.Phi1}

We
start with the state-sum of $\hat J_{K,n}(q)$ over the set of states $s$ with
$|s| \leq n$ and separate it in two different sums: $Q_2(s) > 4|s|/3$
or $Q_2(s) \leq 4|s|/3$. Then we have
$$
\hat J_{K,n}(q)=f^{(1)}_{n}(q) + f^{(2)}_{n}(q)\,,
$$
where
$$
f^{(1)}_{n}(q)=
\sum_{s:\,|s| \leq n\,,Q_2(s) > 4|s|/3} F(q^{n+1},q,s),
\qquad f^{(2)}_{n}(q)=\sum_{s:\,|s| \leq n\,,Q_2(s) \leq 4|s|/3} F(q^{n+1},q,s)\,.
$$
We will show that $f^{(i)}_{n}(q)$ are $1$-stable for $i=1,2$ and compute their
$1$-stable limit. For the stability of $f^{(1)}_{n}(q)$, write
$$
F(x,q,s)= q^{Q_2(s)} \sum_{k=0}^\infty a_k(q,s) x^k \,,
$$
where $a_k(q,s) \in \BZ((q))$ satisfy $\mindeg_q(a_k(q,s)) \geq -k |s|$.
Define
$$
\Phi^{(1)}_{k}(q)=\sum_{s: Q_2(s) > 4|s|/3} q^{Q_2(s)} a_k(q,s)
$$
for $k=0,1$. Using Equation \eqref{Fxqs} we see that
\begin{eqnarray*}
a_0(s)&=& (q)_\infty^{c_D}
\frac{(-1)^{L_0(s)}}{\prod_{e \in \calE}(q)_{s(e)}} \\
a_1(s)&=& \frac{(q)_\infty^{c_D}}{1-q}
\frac{(-1)^{L_0(s)}}{\prod_{e \in \calE} (q)_{s(e)}}
\left(\sum_{e \in \calE} q^{-s(e)} - \sum_{v \in \cV}  q^{-s(v)} \right)
\end{eqnarray*}
where $\cV$ and $\calE$ are the vertices and the edges of $\bD$.

The series $\Phi^{(1)}_k(q)$ for $k=0,1$ are convergent since
$Q_2(s) -k|s| > 4|s|/3-k|s| \geq |s|/3$ for $k=0,1$. Moreover,
\begin{equation*}
(f^{(1)}_{n}(q)-\Phi^{(1)}_{0}(q)-q^{n+1}\Phi^{(1)}_{1}(q))q^{-n-1}=
\S_{1,n}-\S_{2,n}\,,
\end{equation*}
where
\begin{eqnarray*}
\S_{1,n}
&=& \sum_{Q_2(s) > 4|s|/3, \, |s| \leq n}  q^{Q_2(s)}
\ev_n\left(\sum_{k=2}^\infty a_k(q,s) x^{k-1}\right)
\\
\S_{2,n}&=& \sum_{Q_2(s) > 4|s|/3, \, |s| > n}  q^{Q_2(s)}
\ev_n\left(\sum_{k=0}^1 a_k(q,s) x^{k-1}\right) \,.
\end{eqnarray*}
Here, $\ev_n(f(x))=f(q^{n+1})$.
For the first sum, we have it suffices to consider $k=2$ and then
$$
Q_2(s) + n+1 -2 |s| > 4|s|/3+n+1-2|s|=|s|/3 + n+1-|s|
$$
Now,
$$
\min_{|s| \leq n} ( |s|/3 + n+1-|s|) =(|s|/3 + n+1-|s|)|_{|s|=n}=n/3+1
$$
thus the limit of the first sum is zero. For the second sum,
we have:
$$
Q_2(s)-n-1 \geq 4|s|/3-n-1 > 4n/3-n-1 \geq n/3-1
$$
and the limit is zero, too.

For the stability of $f^{(2)}_{n}(q)$, use Lemma \ref{lem.1bdd} to write
$s=m s_P + s'$ where $|s|=m$ and $p$ is a $B$-polygon with
$\kappa(P)$ edges. It follows that
$$
F(x,q,s)=\frac{q^m(q^{m+1})_\infty^{\kappa(P)} }{(q)_\infty^{\kappa(P)}} F(x,q,s')
\,.
$$
Since $|s|=m>n$, we change variables to $m=n+1+l$.
Then,
$$
q^{-(n+1)}f^{(2)}_{n}(q)=(q)_\infty^{c_D} \sum_P\sum_{l,s'} (-1)^{L_0(s')}
\frac{q^l(q^{n+2+l})_\infty^{\kappa(P)}}{(q)_\infty^{\kappa(P)}} \,\,
\frac{   q^{Q_2(s)}}{  \prod_{e\in \cE_P}\,  (q)_{s(e)} } \,.
$$
It follows that the limit is obtained by setting $q^n=0$ in the above
expression and summing over $l$, we obtain that
$$
\lim_{n\to\infty} q^{-(n+1)}f^{(2)}_{n}(q)=\Phi^{(2)}_0(q)
$$
where
$$
\Phi^{(2)}_1(q)=
\frac{(q)_\infty^{c_D}}{1-q}
\sum_P \sum_{s'} (-1)^{L_0(s')}
\frac{1}{(q)_\infty^{\kappa(P)}} \,\,
\frac{   q^{Q_2(s)}}{  \prod_{e\in \cE_P}\,  (q)_{s(e)} } \,.
$$
Setting $\Phi^{(2)}_0(q)=0$, it follows that
$$
\lim_{n\to\infty} (f^{(2)}_{n}(q)-\Phi^{(2)}_{0}(q)-q^{n+1}\Phi^{(2)}_{1}(q))q^{-n-1}=0\,.
$$
Using Section \ref{sec.bDtoD*} we can convert the above formula for
$\Phi_{K,1}(q)$ in terms of the Tait graph of an alternating planar
projection of $K$. This concludes the proof of Theorem \ref{thm.1a}.
\qed

The above algorithm can be extended to compute any $\Phi_{K,k}(q)$ as
follows. Separate the state-sum of Equation \eqref{e242} in two regions:
\begin{itemize}
\item
$s$ is not $k$-bounded.
\item
$s$ is $k$-bounded.
\end{itemize}
In the first region, use Proposition \ref{prop.unbounded} to compute the
$k$-stable limit.
In the second region, use Proposition \ref{prop.sk1} to write
$$
s=s^{(1)}+s', \qquad s=\sum_{j=1}^t (m-k_j)s_{P_j}\,.
$$
Observe that $t \leq k$. If $t=k$ we stop. Else, replace $(s,k)$ by
$(s',k-t)$ in the above step and and run it again. Keep going. Since each
step requires at least one new polygon of $B$-type which is vertex-disjoint
from the previous ones, this algorithm terminates in finitely many steps.


\section{$\Phi_0$ is determined by the reduced Tait graph}
\lbl{sec.tait}


In this Section we prove Corollary \ref{c001}.  Throughout we use the following convention on graphs: a graph is a  finite 1-dimensional CW-complex without loop edge.

Recall that a plane graph is a pair $\gamma=(\Gamma,f)$, where $\Gamma$ is a finite
connected planar graph and $f: \cT\to \BR^2 \subset S^2$ is an embedding. For example, if $D$ is an alternating nonsplit link diagram,
then the Tait graph $\gamma(D)=(\cT,f)$ is a plane graph. One can recover $K$ from $\cT$ up to orientation.

\subsection{From plane graph to non-oriented alternating link}
For a  plane graph $\gamma=(\Gamma,f)$ with $f(\Gamma)\subset \BR^2$ define an alternating $A$-infinite link diagram $D(\gamma)$ as follows.
If we replace $f(\Gamma)$ by a small normal neighborhood in $\BR^2$ and twist each edge as indicated below, then the boundary of the resulting surface
is $D(\gamma)$.

$$
\psdraw{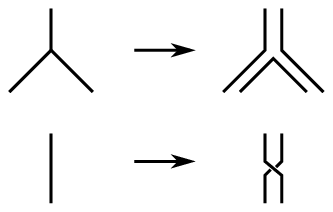}{2in}
$$

 Note that $D(\gamma)$ is a non-oriented alternating $A$-infinite link diagram. The resulting $D(\gamma)$, although alternating, maybe reduced.
If $D$ is an alternating link diagram, and $\cT$ be its Tait graph, then $D(\cT) = D$.

 \begin{exercise}
 Show that $D(\gamma)$ is reduced if and only if $\Gamma$ contains a cut edge, i.e. an edge $e$ such that removing in interior of $e$ make $\Gamma$ disconnected.
 \end{exercise}
For a plane graph $\gamma$ let $K(\gamma)$ be the non-oriented alternating link whose diagram is $D(\gamma)$.
Even when $D(\gamma)$ is reduced, it is still $A$-adequate. Hence we can use $D(\gamma)$ to calculate $\Phi_{K(\gamma),0}$, as in Theorem \ref{thm.1}, see
Remark \ref{rem.ade}. This means
\be  \Phi_{\gamma,0}(q)=  \Phi_{K(\gamma),0}(q),
\lbl{e909}
\ee
where $\Phi_{\gamma,0}$ is given by the the right hand side of \eqref{eq.nahm0} with $D=D(\gamma), \cT= \gamma$.

The dual $D^*(\gamma)$ (in $S^2$) of $D(\gamma)$  can be constructed directly from $\gamma$ as follows: in each region $p$ of $\gamma$ choose a point $u_p$
and connect $u_p$ to all the vertices of $p$ by edges inside the region
$p$ so that the edges do not intersect except at $u_p$. Then $D^*(\gamma)$ is the
plane graph whose vertex set is $\{u_p, p \in \cP(\gamma)\} \cup \cV(\gamma)$
and whose edges are all the edges just constructed. The edges of $\gamma$
are not edges of $D^*$.

\subsection{$k$-connected graph} Recall that a vertex $v$ of a graph $\Gamma$ is a cut vertex if $\Gamma$ is the
union of two proper subgraphs $\Gamma_1$  and $\Gamma_2$ so that $\Gamma_1 \cap \Gamma_2 = \{v\}$. A graph is 2-connected if it is connected and  has no cut vertex.

A pair $(u,v)$ of vertices of  $\Gamma$ is a {\em cut pair} if $\Gamma$ is the
union of two proper subgraphs $\Gamma_1$  and $\Gamma_2$, neither of which is an
edge, so that $\Gamma_1 \cap \Gamma_2 = \{u, v\}$.

Suppose $u,v$ are a cut pair for a 2-connected plane graph $\gamma$.  A {\em Whitney flip} is the
operation that replaces a plane graph $\ga$ by a plane graph $\ga'$
as follows
$$
\psdraw{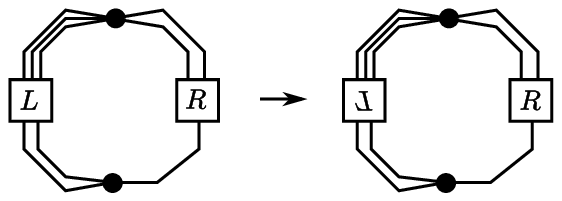}{3in}
$$

From the definition it is clear that $K(\gamma')$ is then obtained from $K(\gamma)$ by a Conway mutation, described in the following
figure :
$$
\psdraw{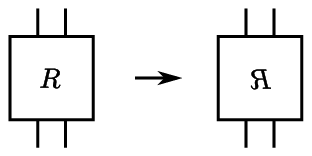}{2in}
$$

By Whitney theorem \cite{Whitney}, 
two planar embeddings of a 2-connected planar graph are
related by a sequence of Whitney flips, composed with a auto-homeomorphism of $S^2$. Since Conway mutation does not change the colored Jones polynomial \cite{MT}, from \eqref{e909}  we have the following.
\begin{lemma}
If $\gamma_1$ and  $\gamma_2$ are two planar embeddings of the same 2-connected graph, then $\Phi_{\gamma_1,0}= \Phi_{\gamma_2,0}$.
\lbl{lem.002}
\end{lemma}

\subsection{Planar collapsing of a bigon} Suppose $\gamma=(\Gamma,f)$ is a plane graph, and among the regions of $\BR^2 \setminus f(\Gamma)$ there is a bigon $u$ with vertices $v_1,v_2$ and edges $e_1, e_2$.
Let $\beta$ be the plane graph obtained from $\gamma$ by squeezing the bigon into one edge, called $e$, so that the bigon disappears and both $e_1$ and $e_2$ becomes $e$. We call $\gamma \to \beta$
a planar collapsing.

\begin{lemma} If $\gamma\to \beta$ is a planar collapsing, then
 $\Phi_{\gamma,0}= \Phi_{\gamma',0}$.
 \lbl{lem.planar}
\end{lemma}

\begin{proof} The bigon  contributes an $A$-vertex $v_u$ to the set of
vertices of $D^*(\gamma)$. Then $\deg(v_u)=2$, and hence $L(v_u)=0$. By
isolating
the factors in the formula \eqref{eq.nahm0} of $\Phi_{\gamma,0}$ involving the
vertex $v_u$ we have
$$
\Phi_{\gamma,0}= \Phi_{\beta,0}
\left[ (q)_\infty \sum_{a: a+ b_1 \ge 0, a+b_2 \ge0}
\frac{q^{(a+b_1)(a+b_2)}}{(q)_{a+b_1} (q)_{a+b_2} }\right] \,.
$$
Here $a,b_1,b_2$ are the coordinates of $\lambda$ at respectively
$v_u, v_1, v_2$; and $b_1$ and $b_2$ are fixed in the sum. By the well-known
Durfee's identity (see \cite[Eqn.2.6]{Andrews}), the factor in
the square bracket is equal to 1.
\end{proof}

\begin{remark}
Suppose $K_1$ and $K_2$ are alternating links such that after several planar
collapsings from $\cT(K_1)$ and and $\cT(K_2)$ one gets the same plane graph,
then
the above lemma says that $\Phi_{K_1,0}= \Phi_{K_2,0}$. This was proved in
\cite{ArDas} by another method.
\end{remark}

\subsection{Abstract collapsing}
\lbl{sub.collapsing}

Suppose $\Gamma_1$ is an abstract graph with a pair of parallel edges
$e_1,e_2$. Removing the interior of $e_1$, from $\Gamma_1$ we get a graph
$\Gamma_2$. We say that the move $\Gamma_1 \to \Gamma_1$ a collapsing.
Note that if $\Gamma_1$ is 2-connected  then $\Gamma_2$ is also 2-connected.

\begin{lemma}
\lbl{lem.0003}
Suppose $\gamma_1=(\Gamma_1,f_1)$ is a  plane graph, and  $\Gamma_2$ is
obtained from $\Gamma_1$ by collapsing a pair of parallel edges $e_1,e_2$.
Then there is a planar embedding $\gamma_2$ of $\Gamma_2$,
$\gamma_2=(\Gamma_2,f_2)$, such that $\Phi_{\gamma_1,0}= \Phi_{\gamma_2,0}$.
\end{lemma}

\begin{proof}
In the planar embedding $f_1(\Gamma_1) \subset \BR^2$,  $e_1$ and $e_2$
bound a region which may contain a subgraph $\Gamma_0$ of $\Gamma_1$.
Note that the common vertices $v_1$ and $v_2$ of $e$ and $e_2$ form a cut
pair for $\Gamma_1$.
By flipping $f_1(e_2)\cup f_2(\Gamma_0)$ through $v_1$ and $v_2$, from
$\gamma_1$ we get a new plane graph $\gamma_3=(\Gamma_1,f_3)$ in which
$f_3(e_1)$ and $f_3(e_2)$ form a bigon, and the result of planar collapsing
this bigon is denoted by $\gamma_2$.
By Lemmas \ref{lem.planar} and \ref{lem.002}, we have
$\Phi_{\gamma_1,0}= \Phi_{\gamma_3,0}= \Phi_{\gamma_2,0}$.
\end{proof}

\subsection{Proof of Corollary \ref{c001}}
We will first prove the following statement.

\begin{lemma}
\lbl{lem.0004}
Suppose $\gamma_i=(\Gamma_i,f_i)$ for $i=1,2$ are 2-connected graphs such
that $\Gamma_1'=\Gamma_2'$ as abstract graphs. Then
$\Phi_{\gamma_1,0}= \Phi_{\gamma_2,0}$.
\end{lemma}

\begin{proof}
Case 1: both $\Gamma_1$ and $\Gamma_2$ do not have multiple edges.
Then $\Gamma'_i=\Gamma_i$, hence $\Gamma_1=\Gamma_2$, and $\gamma_1$ and
$\gamma_2$ are planar embeddings of the same
2-connect graph. Lemma \ref{lem.002} tells us that
$\Phi_{\gamma_1,0}= \Phi_{\gamma_2,0}$.

Case 2: General case. This case is reduced to Case 1 by induction on the
number of total pairs of parallel edges in $\Gamma_1$ and $\Gamma_2$.
If there is no pair of parallel edges, this is Case 1.
Suppose $\Gamma_1$ has a pair of parallel edges, and let $\Gamma_3$ be
the result of abstract collapsing this pair of parallel edges. By Lemma
\ref{lem.0003}, there is a planar embedding $\gamma_3$ of $\Gamma_3$
such that $\Phi_{\gamma_1,0}= \Phi_{\gamma_3,0}$. Note that
$\Gamma_3'= \Gamma_1'= \Gamma_2'$. By induction, we have
$\Phi_{\gamma_3,0}= \Phi_{\gamma_2,0}$. This proves $\Phi_{\gamma_1,0}= \Phi_{\gamma_2,0}$.
\end{proof}

\noindent
Let us proceed to the proof of Corollary \ref{c001}.
Suppose $K_1$ and $K_2$ are alternating links such that $\cT(K_1)$ is
isomorphic to $\cT(K_2)$ as abstract graphs.
We can assume that both $K_1$ and $K_2$ are non-split. For $i=1,2$ let
$\gamma_i= (\cT(D_i), f_i)$ be the plane Tait graph of a reduced
$A$-infinite alternating link diagram of $K_i$.
Note that $\cT(D_i)$ is connected since $K_i$ is non-split. Moreover
$\cT(D_i)$ does not have a cut vertex since $D_i$ is reduced.
That is, $\cT(D_i)$ is $2$-connected. From Lemma \ref{lem.0004} and
Theorem \ref{thm.1} we have $\Phi_{K_1,0}= \Phi_{K_2,0}$. This completes
the proof of Corollary \ref{c001}.


\section{Examples}
\lbl{sec.examples}

In this section we give a formula for the $q$-series $\Phi_{K,0}(q)$ for all
twist knots and their mirrors, taken from unpublished work of the first
author and D. Zagier. In some caes, similar formulas have also been
obtained by Armond-Dasbach. Recall the family of {\em twist knots}
$K_p$ for an integer $p$ depicted as follows:
$$
\psdraw{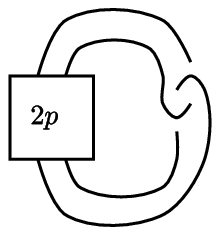}{0.9in} \qquad \text{where} \qquad
\psdraw{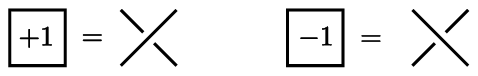}{2in}
$$
The planar projection of $K_p$
has $2|p|+2$ crossings, $2|p|$ of which come from the full twists,
and $2$ come from the negative {\em clasp}. For small $p$, the twist
knots appear in Rolfsen's table \cite{Rf} as follows:

$$
\begin{array}{|c|cccccccc|} \hline
\text{Twist knot} & K_{-4} & K_{-3} & K_{-2} & K_{-1} & K_1 & K_2 & K_3 & K_4
\\ \hline
\text{Rolfsen notation} & 10_1 & 8_1 & 6_1 & 4_1 & 3_1 & 5_2 & 7_2 & 9_2
\\ \hline
\end{array}
$$

Recall that $\sgn(n)=+1,0,-1$ when $n<0,n=0,n>0$ respectively.

\begin{theorem}
\lbl{thm.twist}
For $p<0$ we have:
\be
\lbl{eq.minus}
\Phi_{K_p,0}(q)= (q;q), \qquad
\Phi_{-K_p,0}(q)= \frac{(q;q)}{(q^2;q^{2|p|+1}) (q^3;q^{2|p|+1})
\dots (q^{2|p|-1};q^{2|p|+1})} \,.
\ee
For $p>0$ we have
\be
\lbl{eq.plus}
\Phi_{K_p,0}(q)=\sum_{n=0}^\infty q^{p n^2 + (p-1)n}-
\sum_{n=0}^\infty q^{p n^2 + (p+1)n+1} =
1+\sum_{n \in \BZ} \sgn(n) q^{p n^2 + (p-1)n}
\ee
\be
\Phi_{-K_p,0}(q)=(q;q)\,.
\ee
\end{theorem}
Equation \eqref{eq.minus} implies that for $p<0$, $\Phi_{\pm K_p,0}(q)$
are modular forms \cite{BGHZ}. On the other hand, when
$p>1$, $\Phi_{K_p,0}(q)$ is {\em not} modular of any weight, according
to K. Ono.
This disproves any conjectured modularity properties of $\Phi_0(q)$,
even for $5_2$. On the other hand, $\Phi_{\pm K_p,0}(q)$ is a {\em false
theta series} of Rogers.

The modular form $\Phi_{K_p,0}(q)$ for $p>0$ is a beautiful theta series,
with a factorization
\be
\lbl{eq.b}
\frac{(q;q)}{\prod_{k=2}^{2b-1}(q^k;q^{2b+1})}=\sum_{n \in \BZ}
(-1)^n q^{\frac{2b+1}{2} n^2 +\frac{2b-1}{2} n}
\ee
for all natural numbers $b$. It was pointed out to us by D. Zagier that
the above identity follows immediately from the Jacobi triple product
identity, discussed in detail in \cite{BGHZ}.


\appendix

\section{Proof of the state-sum formula for the colored Jones
function}
\lbl{sec.app1}

In this section we give a proof of Equation \eqref{eq.statesum} which
we could not find in the literature. We begin by
recalling the definition of the colored Jones polynomial using
$R$-matrix.


\subsection{Link invariant associated to a ribbon algebra}

Quantum link invariants can be defined using a ribbon Hopf algebra. We
recall the formula for the invariant here. For further details,
see \cite{RT} or \cite{Ohtsuki}.

A  ribbon Hopf algebra $\U$  over a ground field $\cF$ has an $R$-matrix
$R \in \cH \otimes \cH$ and a group-like element $g \in \cH$ satisfying
$$
S^2(x) = g x g^{-1} \quad \forall x \in \cH,
$$
where $S$ is the antipode of the $\cH$.

Suppose  $V$  is a $\cH$-module, and $K$ is a {\em framed} link with a
downward  planar diagram $D$, where the framing is the blackboard framing.
The dual space $V^*$ has a natural structure of a $\cH$-module.
Fix a basis $\{e_j\}$ of $V$ and a dual basis $\{e_j^*\}$ of $V^*$.

The quantum invariant $\hat J_K(V)$ is defined through tangle operator
invariants  as follows.

The six tangle diagrams in Figure  \eqref{f.6types} are called
{\em elementary tangle diagrams}. {\em An extension} of an elementary tangle
diagram is the result of adding some (maybe none) vertical lines to the left
and to the right of an elementary tangle diagram, with arbitrary orientations
on the added lines, as in the following figure
\be
\lbl{f.extended}
\psdraw{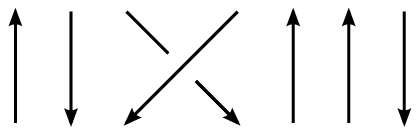}{2in}
\ee

Suppose $D$ is a downward link diagram which is in general position.
Using horizontal lines we cut $D$ into tangles, each is an extension of
an elementary tangle diagram.
Let  $T$ be one of the resulting tangles. On the bottom boundary of $T$
assign $V$  to each endpoint of $T$ where $T$ is  oriented down, and
the dual object $V^*$  to each endpoint where $T$ is oriented up. Tensoring
from left to right, this gives the boundary object $\partial_-T$. One
defines similarly $\partial_+T$, using the top boundary endpoints instead
of the bottom boundary ones. By convention, the empty product is the the
ground field $\cF$. It is clear that if $T'$ is the tangle right above $T$,
then $\partial_-(T') = \partial_+(T)$. For example, for the tangle $T$ of
Figure \eqref{f.extended} we have
$$
\partial_- T=\partial_+T=V^* \otimes V
\otimes V \otimes V \otimes V^* \otimes V^* \otimes V
$$

For each tangle $T$ as above we will define an operator
$$
\check J_T: \partial_-T \to \partial_+T
$$
as follows. First if $T$ is one of
the elementary tangles, then $\check J_T$ is given by

\begin{align}
T = \psdraw{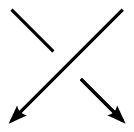}{0.5in} \qquad &
\check J_{T}:  V \otimes V \to V \otimes V \qquad
\text{given by}  \qquad \check J_T= \bb:= \sigma \, R  \\
 T = \psdraw{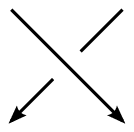}{0.5in} \qquad &
\check J_{T}:  V \otimes V \to V \otimes V \qquad
\text{given by} \qquad \check J_T=  \bb^{-1} = R^{-1} \, \sigma  \\
 T = \psdraw{capr}{0.5in} \qquad &
\check J_{T}:  V^* \otimes V \to \cF  \qquad
\text{given by} \qquad f \otimes x \to f(x) \\
 T = \psdraw{capl}{0.5in} \qquad &
\check J_{T}:  V \otimes V^* \to \cF  \qquad
\text{given by} \qquad x \otimes f \to f(gx) \\
 T = \psdraw{cupr}{0.5in}  \qquad &
\check J_{T}: \cF\to V \otimes V^* \qquad
\text{given by} \qquad 1  \to \sum_j e_j \otimes e_j^* \\
 T = \psdraw{cupl}{0.5in} \qquad &
\check J_{T}: \cF\to V^* \otimes V \qquad
\text{given by} \qquad 1  \to \sum_j e_j^* \otimes g^{-1} (e_j).
\end{align}
Here $\sigma: V\otimes V$ is the permutation,
$\sigma(x\otimes y) = y\otimes x$.
If $T$ is an extension of an elementary tangle $E$, say $T$ is the result
of adding $m$ vertical lines to the left and $n$ vertical lines to right
of $E$, then define
$$
\check J_T = \id^{\otimes m} \otimes \check J_E \otimes \id^{\otimes n}.
$$
Finally, if $T_1,\dots, T_m$ are the tangles in the decomposition of the
downward  diagram $D$ of the link $K$, counting from top to bottom, then
$$
\check J_K:= \check J_{T_1} \dots \check J_{T_m},
$$
is an element of $\Hom_\cH (\cF, \cF)$ which one identifies with $\cF$.

\subsection{The case $\cH= U_h(\fsl_2)$}
\def\tbb{\tilde \bb}

The colored Jones polynomial is the quantum link invariant corresponding
to the ribbon Hopf algebra $\U:=U_h(\fsl_2)$, the quantized enveloping algebra
of $\fsl_2$. There are two versions of $\U$ in the literature,
we will use here the version used in \cite{Kassel,Ohtsuki}, which has the
opposite co-product structure of the one used in \cite{Jantzen,Lusztig}.
The ground ring is $\BQ[[h]]$ is not a field, but the theory carries over
without changes.

Recall that  $\U=U_h(sl_2)$ is the $h$-adically completed $\BQ[[h]]$-algebra
generated by $H,E,F$ subject to the relations

$$
\quad HE= E(H+2), \quad H F= F(H-2), \quad EF-FE=\frac{K-K^{-1}}{v-v^{-1}}.
$$
where
$$
v= \exp(h/2)=q^{1/2}
$$
and $K= \exp(hH/2)$.
The group like element is $g=K$. Recall the balanced quantum integer,
and the corresponding balanced quantum factorials and binomials defined
by
\begin{eqnarray*}
[ a ] &=& \frac{v^a-v^{-a}}{v-v^{-1}},  \quad \left[a\right] ! = \prod_{k=1}^a [k] \quad \text{ for } \, a \in \BN
\\
\qbinom a b &=& \frac{[a]!}{[b]![a-b]!} \quad
\text{ for } a, b \in \BN, b \le a.
\end{eqnarray*}
The $R$-matrix is an element of $\U \ho \U$, the completed tensor product of
$\U$ and $\U$, given by
$$
R = D \sum_{k=0}^\infty \frac{ v^{k(k-1)/2} \, (v-v^{-1})^k }{[k]!}  E^n  \otimes
F^{n}
$$
where 
$D = \exp(hH \otimes H/4)$, which is called the diagonal part.

The inverse of $R$ is
$$
R^{-1} =
\left( \sum_{k=0}^\infty \frac{(-1)^k v^{-k(k-1)/2} \, (v-v^{-1})^k }{[k]!}
E^n  \otimes F^{n} \right) D^{-1}
$$
For each positive integer $n$ there is a unique $n+1$-dimensional $\U$-module
$V_n$ such that there is an element $e_0\in V_n$ satisfying
$$
H(e_0)= n e_0, \quad E(e_0) =0
$$
The module $V_n$ is freely spanned by $F^j(e_0), j=0,1,\dots, n$. The basis
$$
\{ u_j=F^j(e_0)/[j]! \, | \, j=0,\dots,n\}
$$
is known as the canonical basis of $V_n$.

For a framed link $K$ let $\check J_{K,n}$ be the invariant of $K$ with color
$V_n$. It is known that if one increases the framing of a component of $K$ by one,
then $J_{K,n}$ gets multiplied by $v^{(n^2 + 2n)/2}$.

Define $\tilde J_{K,n}$ in the same way as in the definition  of
$\check J_{K,n}$, with $\bb$ replaced by $\tbb :=  v^{-\frac{n^2}{2}-n} \bb$ and
$\bb^{-1}$ replaced by $\tbb^{-1}$. Then $\tilde J_{K,n}$ is an invariant of unframed links.
Since $\tilde J_{K,n}= q^a\, \check  J_{K,n}$ for some $a\in \frac 14 \BZ$, when dividing by the smallest monomial, both $\tilde J_{K,n}$ and $\check  J_{K,n}$
are the same.

\subsection{$R$-matrix in the canonical basis}
\lbl{sub.Rcanonical}

In this section we calculate the matrix of $\tbb$ in the product of the
canonical basis. The action of $H, E^k$ and $F^k$ on the canonical basis
is given by

$$
F^k(u_a) = \frac{[a+k]!}{[a]!}\, u_{a+k}, \quad
E^k(u_a)= \frac{[n+k-a]!}{[n-a]!}\, u_{a-k}, \quad H(u_{j}) = (n-2j)e_j
$$
where we assume  $u_j=0$ if $j <0$ or $j>n$. From here one can easily
calculate the formula of $\tbb$ and $\tbb^{-1}$,

\begin{align}
\tbb(u_a \otimes u_b) & =
\sum_k v^{-n-na-nb + 2ab+2ak-2kb-\frac{3k^2+k}{2}} \{k\}!\,\qbinom{n+k-a}{k}
\, \qbinom{b+k}{k} \, u_{b+k} \otimes u_{a-k} \lbl{e40a}\\
\tbb^{-1}(u_a \otimes u_b) & = \sum_k (-1)^k v^{n+ nb+na-2ab-\binom k2}
\{k\}!\,\qbinom{n+k-b}{k} \, \qbinom{a+k}{k} \, u_{b-k} \otimes u_{a+k}
\lbl{e40b}
\end{align}
Let us denote by $\tbb_{a,b}^{c,d}$ the matrix entry of $\tbb$, i.e.,
$$
\tbb(u_a \otimes u_b) = \sum_{c,d} \tbb_{a,b}^{c,d} \, u_c \otimes u_d.
$$
Then $\tbb_{a,b}^{c,d} =0$ and $(\tbb^{-1})_{a,b}^{c,d}=0$ unless the numbers
$a,b,c,d$ form an $n$-admissible state for the crossing, i.e.,
$a+b=c+d$, $\ve(C)(a-d) \ge 0$ and $a,b,c,d \in [0,n] \cap \BZ$. Here
$\ve(C)$ is the sign of the crossing $C$. If $a,b,c,d$ form an
$n$-admissible state for the crossing, then from \eqref{e40a} and
\eqref{e40b} we have

\begin{align}
 (\tbb)_{a,b}^{c,d} & = v^{-n-nd-nb +ab+dc } (q^{-1};q^{-1})_k
\,\binom{n-d}{a-d}_{q^{-1}} \, \binom{c}{c-b}_{q^{-1}}  \lbl{e101a} \\
 (\tbb^{-1})_{a,b}^{c,d}  & = (-1)^k v^{n+ nb+nd -bd -ac + b-c} \, (q^{-1};q^{-1})_k
\, \binom{n-c}{b-c}_{q^{-1}}  \, \binom{d}{d-a}_{q^{-1}}
  \lbl{e101b}
\end{align}

Choose the following basis $\{f_0,\dots,f_n\}$ for the dual $V_n^*$ such
that $f_j= v^{-(n-2j)/2} e_j^*$. Then

\begin{align}
 T = \psdraw{capr}{0.5in}   \qquad & \check J_{T}:  V^*_n\otimes V_n\to \cF
\qquad \text{given by }  f_a  \otimes e_b  \to \delta_{ab} v^{-(n-2a)/2} \\
 T = \psdraw{capl}{0.5in}  \qquad & \check J_{T}:  V_n \otimes V^*_n\to \cF
\qquad \text{given by }  e_a \otimes f_b \to   \delta_{ab} v^{(n-2a)/2} \\
 T = \psdraw{cupr}{0.5in}  \qquad & \check J_{T}: \cF\to V_n\otimes V^*_n
\qquad \text{given by } 1  \to \sum_a v^{(n-2a)/2}  e_a \otimes f_a \\
 T = \psdraw{cupl}{0.5in}    \qquad & \check J_{T}: \cF\to V^*_n \otimes V_n
\qquad \text{given by } 1  \to \sum_j  v^{-(n-2a)/2} f_a \otimes e_a.
 \lbl{e109b}
\end{align}

From Equations \eqref{e101a}--\eqref{e109b} we see that

$$
J_{K,n}(q) = \tilde J_{K,n}(q^{-1}),
$$
where $J_{K,n}$ is the given in section \ref{sub.weights}.


\section{The lowest degree of the colored Jones polynomial of an
alternating link}
\lbl{sub.lowest}

We fix an $A$-infinite, reduced, alternating, downward  diagram $D$ of a
link $K$. Corollary \ref{c.degree} shows that
the minimal degree of the colored Jones polynomial $J_{K,n}$ is given by
$P_1(n) =  \frac{n}{2} c_+ - \frac{n^2 +2n}{2} c_- -\frac{n}{2} \sum_M W(M)$,
where the sum is over all the local extreme points of $D$.
On the other hand, the Kauffman bracket approach gives the minimal degree as
$\frac{n}{2} c_+ - \frac{n^2 +n}{2} c_-  - \frac{n}{2} s_A$, see \cite{Tu3}
and also \cite{Le,Ga1}. Here we show that the two results agree.
Recall that $s_A$ is the number of circles obtained from $D$ after doing
$A$-smoothenings at every vertex crossing of $D$. If $D$ is a connected graph, then $s_A$ is the number of $A$-vertices of $\D^*$.

\begin{lemma}
\lbl{lem.c-s+}
We have
\be
\lbl{e200} c_- = s_A - \sum_M W(M)
\ee
Consequently,
\be
P_1(n)= \frac{n}{2} c_+ - \frac{n^2 +n}{2} c_-  - \frac{n}{2} s_A.
\lbl{e06}
\ee
\end{lemma}

\begin{remark}
\lbl{rem.c-s+1}
If $D$ is not $A$-infinite, Equation \eqref{e200} fails. For example, it
fails for for the following diagram of the right handed trefoil
$$
\psdraw{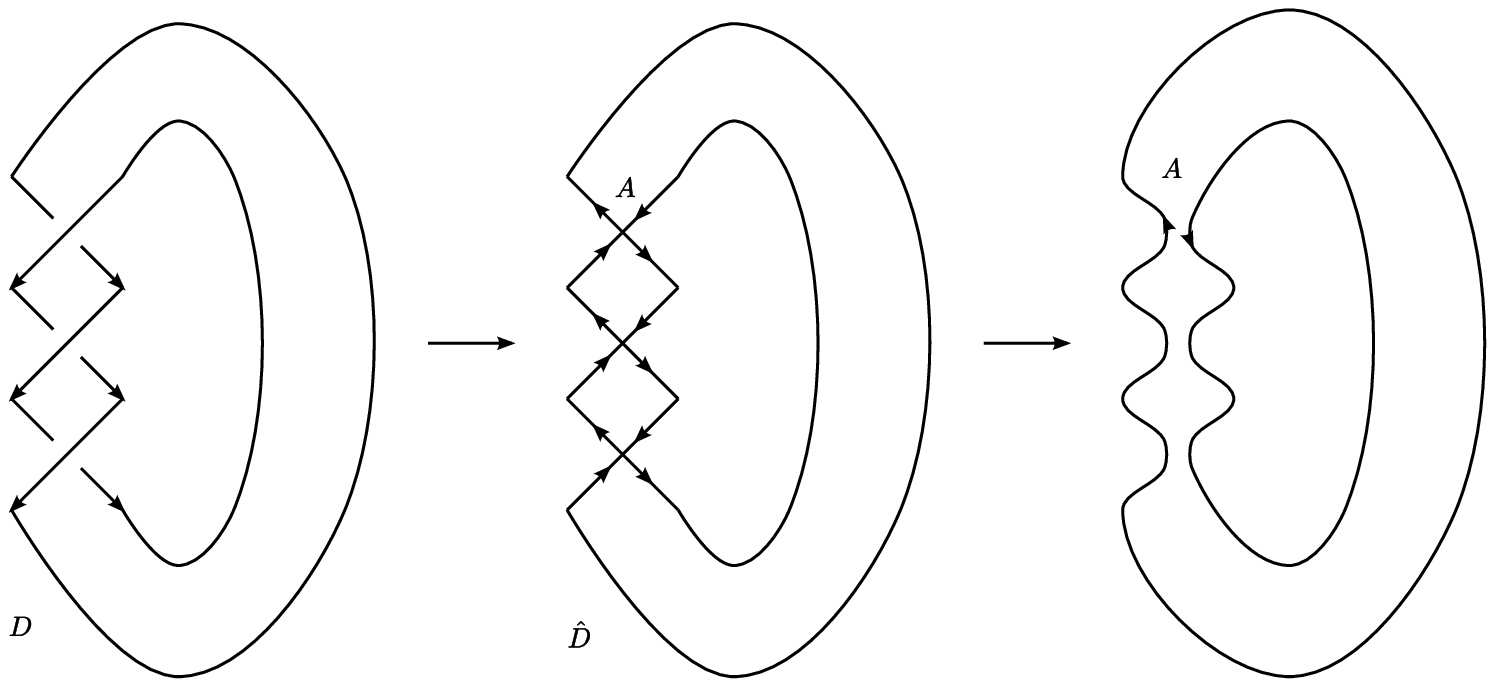}{4.5in}
$$
we have
$c_-=0$, $s_+=2$, and there are two clockwise local extrema and two
counterclockwise local extrema giving $\sum_P \ve(P)=0$.
\end{remark}

\begin{proof}(of Lemma \ref{lem.c-s+})
Suppose $D'$ is the result of doing  $A$-type resolution at every crossing
of $D$. Since all crossings of $D$ are downward, an $A$-type resolution
at a positive  crossing of $\bD$ creates no local extreme point, while an
$A$-type resolution at a negative crossing created 2 local extreme points,
each has winding weight $1/2$.

It follows that
$$
c_- +  \sum_{M \in D} W(M)=
 \sum_{M \in D'} W(M)
$$
By Lemma \ref{l67}(b),  $D'$ consists of $s_A$ circles, each having
winding number 1. Hence $\sum_{M \in D'} W(M)=s_A$, and we get
$c_-+ \sum_{M \in D} W(M)= s_A$.
\end{proof}


\section{Regularity of Nahm sums}
\lbl{sub.regularity}

In this section we will give a necessary and sufficient criterion for
regularity of a Nahm sum. This section is logically independent from the
proof of Theorems \ref{thm.1} and \ref{thm.2}, but we include it for
completeness. Fix a pointed cone $C$ in the Euclidean space
$(\BR^r,|\cdot|_2)$
with apex the origin that intersects the orthant $[0,\infty)^d$ other
than in the origin and consider a polynomial function of degree $d$
$$
f: \BR^r \longto \BR, \qquad f(s)=\sum_{i=0} f_i(s)
$$
where $f_i$ are homogeneous polynomials of degree $i$ and $f_d$ not
identically zero. Assume that $f(\BZ^r) \subset \BZ$. Let
$X=  C \cap \BN^d$, $X_{\BQ}=C \cap \BQ^d$ and $X_{\BR}=C \cap [0,\infty)^r$.

\begin{proposition}
\lbl{prop.proper}
The following are equivalent.
\begin{itemize}
\item[(a)] $f: X \longto \BZ$ is proper (i.e., the preimage
of a finite set is finite) and bounded below.
\item[(b)] For $s \in X$ there exists $i_0$ such that $f_i(s)=0$ for all
$i<i_0$ and $f_{i_0}(s) \geq 0$.
\item[(c)] There exists $c>0$ such that $F(s) \geq c |s|_2$
for all but finitely many $s \in C \cap \BN^d$
\end{itemize}
\end{proposition}

\begin{proof}
$\rm{(a)} \implies \rm{(b)}$
Fix $s \in C \cap \BN^d$. By properness, it follows that
the sequence $f(ns)$ is unbounded thus it has a subsequence that
goes to infinity. Suppose that $f_i(s)=0$ for $i<i_0$
and $f_{i_0}(s) \neq 0$. Then, $f(ns)=n^{i_0} f_{i_0}(s)+O(n^{i_0-1})$ goes
to infinity. It follows that $f_{i_0}(s) > 0$.

$\rm{(b)} \implies \rm{(c)}$ Indeed, (b) implies that
for all $s \in X$ we have $f_d(s) \geq 0$, and if $f_d(s)=0$ then
(without loss of generality) we assume that $f_{d-1}(s) \geq 0$. Since $f_i$
is homogeneous, it follows that for all $s \in X_{\BQ}$ we have
$f_{d}(s) \geq 0$, and if $f_d(s)=0$ then $f_{d-1}(s) \geq 0$. Since $f_i$
is continuous, it follows that for all $s \in X_{\BR}$ we have
$f_{d}(s) \geq 0$, and if $f_d(s)=0$ then $f_{d-1}(s) \geq 0$. Let $S$ denote the
unit sphere in the Euclidean space $\BR^r$ with the Euclidean norm
$|\cdot|_2$. Then, $f_d: [0,\infty)^r \cap S \longto [0,\infty)$, and
if $Z_d=\{x \in [0,\infty)^r \cap S \, | \, f_d(s)=0\}$, then
$f_{d-1}(Z_d)>0$. By continuity, choose an open neighborhood $U_d$
of $Z_d$ in $[0,\infty)^r \cap S$ with closure $\bar Z_d$
such that $f_{d-1}(\bar Z_d) \subset (0,\infty)$. Then,
$f_d(([0,\infty)^r \cap S)\setminus Z_d) \subset (0,\infty)$ and by compactness
it follows that there exists $c'>0$ such that
$f_d(([0,\infty)^r \cap S)\setminus Z_d) \in [c',\infty)$ and
$f_{d-1}(\bar Z_d) \subset [c',\infty)$. It follows that
$(f_d+f_{d-1})([0,\infty)^r \cap S) \subset [c',\infty)$, and
by homogeneity this implies that for all
$s \in X_{\BR}$, we have $(f_{d}+f_{d-1})(s) \geq c' |s|_2^{d-1}$. On the other
hand, by homogeneity, we have $f_i(s)<c''|s|_2^i$ for $i<r-1$. Since
$c'x^{r-1}-c'' x^{r-2} \geq c x$ for some $c>0$ and for all $x$ sufficiently
large, (c) follows.

$\rm{(c)} \implies \rm{(a)}$ is immediate.
\end{proof}

For example, the function $Q_2$ in Section \ref{Q2} is proper and bounded
from below. Actually, $Q_2\ge 0$ on the cone $S_{\bD,\BN}$.


\section{Experimental formulas for knots with a low number
of crossings}
\lbl{sec.experiment}

Theorem \ref{thm.2} gives an explicit Nahm sum formula for an alternating
knot $K$. The first author programmed the above formula with input an
alternating, reduced, $A$-infinite downward diagram of a knot, and with
the help of D. Zagier computed the first 50 terms of the corresponding
$q$-series for several examples, and then guessed the answer (in all but
the case of $8_5$ knot below). Every such guess is a $q$-series identity,
whose proof is unknown to us. We thank D. Zagier for guidance and stimulating
conversations. For an alternating knot $K$, let
$$
\ga(K)=(c_+,c_-,\s)
$$
denote the triple of positive crossings, negative crossings and the
signature of $K$. $K$ has $c=c_++c_-$ crossings, and writhe $w=c^+-c^-$.
Let $\d^*_K(n)$ and $\d_K(n)$ denote the minimum and maximum degree of the
colored Jones polynomial $J_{K,n}(q)$. Note that $\d^*_K(n)$ and $\d_K(n)$
are determined by $\ga(K)$ by:
$$
\d^*_K(n)=-c_-\frac{n(n+1)}{2}-\frac{\s}{2} n-\frac{n}{2},
\qquad \d_K(n)=c_+\frac{n(n+1)}{2}-\frac{\s}{2} n+\frac{n}{2}
$$
Let $\Phi^*_{K,0}(q) \in \BZ[[q]]$ and $\Phi_{K,0}(q) \in \BZ[[q]]$
denote the stable limit of the colored Jones
polynomial from the left and from the right.
The involution $K \mapsto - K$ given by the mirror image
acts as follows:
$$
c_\pm \mapsto c_\mp, \qquad \s \mapsto -\s, \qquad \d \mapsto -\d^*,
\qquad \d^* \mapsto -\d, \qquad \Phi \mapsto \Phi^*, \qquad
\Phi^* \mapsto \Phi
$$

The formulas for $\Phi_0(q)$ presented below agree to the first 8
values with the {\tt KnotAtlas} table of Bar-Natan, and also to 50
values with the Nahm sum formula of Theorem \ref{thm.2}. The formulas
are proven only for $3_1$ and $4_1$ knot, and remain conjectural for
all others.

$K_p$ is the $p$th {\em twist knot} obtained by $-1/p$ surgery on
the Whitehead link for an integer $p$ and $T(a,b)$ is the
left-handed $(a,b)$ torus knot. The results below are expressed in terms
of the following series for a positive natural number $b$:
\be
\lbl{eq.ha}
h_b(q)=\sum_{n \in \BZ} (-1)^n q^{b n(n+1)/2-n},
\qquad
h^*_b(q)=\sum_{n \in \BZ} \ve(n) q^{b n(n+1)/2-n}
\ee
Observe that
$$
h_1(q)=0, \qquad h^*_2(q)=1, \qquad h_3(q)=(q)_\infty
$$

$$
\begin{array}{|c|c|c|c|l|l|l|l|} \hline
K & c_- & c_+ & \s & \Phi^*_{K,0}(q) & \Phi_{K,0}(q) \\ \hline
\hline
3_1=-K_1 & 3 & 0 & 2 & h_3 & 1 \\ \hline \hline
4_1=K_{-1} & 2 & 2 & 0 & h_3 & h_3 \\ \hline \hline
5_1 & 5 & 0 & 4 & h_5 & 1 \\ \hline
5_2=K_2 & 0 & 5 & -2 & h^*_4 & h_3 \\ \hline \hline
6_1=K_{-2} & 4 & 2 & 0 & h_3 & h_5  \\ \hline
6_2 & 4 & 2 & 2 &  h_3 h^*_4 & h_3 \\ \hline
6_3 & 3 & 3 & 0 & h_3^2 & h_3^2 \\ \hline \hline
7_1 & 7 & 0 & 6 & h_7 & 1 \\ \hline
7_2=K_3 & 0 & 7 & -2 & h^*_6  & h_3 \\ \hline
7_3 & 0 & 7 & -4 & h^*_4 & h_5  \\ \hline
7_4 & 0 & 7 & -2 & (h^*_4)^2 & h_3 \\ \hline
7_5 & 7 & 0 & 4 &  h^*_4  & h^*_4 \\ \hline
7_6 & 5 & 2 & 2 &  h_3 h^*_4 & h_3^2  \\ \hline
7_7 & 3 & 4 & 0 & h_3^3 & h_3^2 \\ \hline \hline
8_1=K_{-3} & 6 & 2 & 0 & h_3 & h_7 \\ \hline
8_2 & 6 & 2 & 4 & h_3 h^*_6 & h_3 \\ \hline
8_3 & 4 & 4 & 0 & h_5 & h_5 \\ \hline
8_4 & 4 & 4 & 2 & h^*_4 h_5  & h_3    \\ \hline
8_5 & 2 & 6 & -4 & h_3   &  ??? 
\\ \hline
\hline
K_p, p>0 & 0 & 2p+1 & -2 & h_{2p}^* & h_3  \\ \hline
K_p, p<0 & 2|p| & 2 & 0  & h_3 & h_{2|p|+1} \\ \hline
T(2,p),p>0 & 2p+1 & 0 & 2p & h_{2p+1} & 1 \\ \hline
\end{array}
$$

For $8_5$, we have computed the first 100 terms using an 8-dim Nahm sum.
The result slightly simplifies when divided by $h_3(q)$:

\begin{math}
\Phi_{8_5}(q)/h_3(q)=
1-q+q^2-q^4+q^5+q^6-q^8+2 q^{10}+q^{11}+q^{12}-q^{13}-2 q^{14}+2 q^{16}+3 q^{17}+2 q^{18}+q^{19}-3 q^{21}-2 q^{22}+q^{23}+4 q^{24}+4 q^{25}+5 q^{26}+3 q^{27}+q^{28}-2 q^{29}-3 q^{30}-3 q^{31}+5 q^{33}+8 q^{34}+8 q^{35}+8 q^{36}+6 q^{37}+3 q^{38}-2 q^{39}-5 q^{40}-6 q^{41}-q^{42}+2 q^{43}+9 q^{44}+13 q^{45}+17 q^{46}+16 q^{47}+14 q^{48}+9 q^{49}+4 q^{50}-3 q^{51}-8 q^{52}-8 q^{53}-5 q^{54}+3 q^{55}+14 q^{56}+21 q^{57}+27 q^{58}+32 q^{59}+33 q^{60}+28 q^{61}+21 q^{62}+11 q^{63}+q^{64}-9 q^{65}-11 q^{66}-11 q^{67}-2 q^{68}+9 q^{69}+27 q^{70}+40 q^{71}+56 q^{72}+60 q^{73}+65 q^{74}+62 q^{75}+54 q^{76}+39 q^{77}+23 q^{78}+4 q^{79}-9 q^{80}-16 q^{81}-14 q^{82}-3 q^{83}+16 q^{84}+40 q^{85}+67 q^{86}+92 q^{87}+114 q^{88}+129 q^{89}+135 q^{90}+127 q^{91}+115 q^{92}+92 q^{93}+66 q^{94}+35 q^{95}+9 q^{96}-12 q^{97}-14 q^{98}-11 q^{99}+13 q^{100}+O(q)^{101}
\end{math}

Let us summarize some observations.

\begin{itemize}
\item
$\Phi_{K,0}(q)$ is not determined by $\ga(K)$ alone: see for
instance $(7_2,7_4)$ and $(7_3,- 7_5)$.
\item
In all knots above except $8_5$, $\Phi_{K,0}(q)$ is a finite product
of the form $\prod_i h_{a_i}^{b_i} (h^*_{a_i})^{c_i}$ where $b_i,c_i \in \BN$.
\item
The modularity properties of $\Phi_{8_5}(q)$ are completely unknown, and so
is its behavior at $q=1$ or at any complex root of unity.
\end{itemize}

\bibliographystyle{hamsalpha}
\bibliography{biblio}
\end{document}